\documentclass[12pt]{amsart}
\usepackage[a4paper,margin=2cm]{geometry}
\usepackage[OT2,OT1]{fontenc}
\usepackage[oldstyle]{libertine}
\usepackage[backend=bibtex,
	sorting=anyt,
	isbn=false,
	url=false,
	doi=false,
	maxbibnames=100
	]{biblatex}
\usepackage[dvipsnames]{xcolor}
\usepackage{mathtools,stmaryrd,amssymb}

\newcommand{\kk}{\Bbbk}

\usepackage{microtype}
\usepackage[colorlinks,final,hyperindex]{hyperref}

\makeatletter
\AtBeginDocument{%
  \let\[\@undefined
  
\DeclareRobustCommand{\[}{\begin{equation}}%
  \let\]\@undefined
  
  \DeclareRobustCommand{\]}{\end{equation}}%
}

\definecolor{Chocolat}{rgb}{0.36, 0.2, 0.09}
\definecolor{BleuTresFonce}{rgb}{0.215, 0.215, 0.36}
\definecolor{EgyptianBlue}{rgb}{0.06, 0.2, 0.65}

\hypersetup{citecolor=BleuTresFonce, 
			urlcolor=EgyptianBlue, 
			linkcolor=Chocolat}

\mathtoolsset{showonlyrefs=true}
\usepackage{tikz}
\usepackage{pgf}
\usetikzlibrary{fit}

\pgfdeclarelayer{bg}    % declare background layer
\pgfsetlayers{bg,main}  % set the order of the layers (main is the standard layer)

\theoremstyle{plain}
\newtheorem{theorem}{Theorem}[subsection]
\newtheorem{proposition}[theorem]{Proposition}
\newtheorem{lemma}[theorem]{Lemma}
\newtheorem{corollary}[theorem]{Corollary}

\theoremstyle{remark}

\newtheorem{remark}[theorem]{Remark}
\newtheorem{definition}[theorem]{Definition}

\newcommand{\Hycomm}{\mathrm{HyperCom}}
\newcommand{\ExtHycomm}{\mathrm{HC}}
\newcommand{\BV}{\mathrm{BV}}
\newcommand{\Gerst}{\mathrm{Gerst}}
\newcommand{\Grav}{\mathrm{Grav}}

\newcommand{\bBV}[1]{\mathrm{BV}_{#1}}

\newcommand{\bHycomm}[1]{\mathrm{HyperCom}_{#1}}

\newcommand{\ncHycomm}{\mathrm{ncHyperCom}}
\newcommand{\ncExtHycomm}{\mathrm{ncHC}}
\newcommand{\ncBV}{\mathrm{ncBV}}
\newcommand{\ncGerst}{\mathrm{ncGerst}}
\newcommand{\ncGrav}{\mathrm{ncGrav}}

\newcommand{\bncBV}[1]{\mathrm{ncBV}_{#1}}

\newcommand{\bncHycomm}[1]{\mathrm{ncHyperCom}_{#1}}

\newcommand{\lmHycomm}{\mathrm{tHyperCom}}
\newcommand{\lmExtHycomm}{\mathrm{tHC}}
\newcommand{\lmBV}{\mathrm{tBV}}
\newcommand{\lmGerst}{\mathrm{tGerst}}
\newcommand{\lmGrav}{\mathrm{tGrav}}

\newcommand{\blmBV}[1]{\mathrm{tBV}_{#1}}

\newcommand{\blmHycomm}[1]{\mathrm{tHyperCom}_{#1}}

\newcommand{\lmk}{\mathsf{h}}
\newcommand{\lml}{\mathsf{l}}
\newcommand{\lmm}{\mathsf{m}}
\newcommand{\lmn}{\mathsf{n}}
\newcommand{\lmg}{\mathsf{g}}

\DeclareMathOperator{\Hom}{Hom}

\newcommand{\ac}{\scriptstyle \text{\rm !`}}

\newcommand\cyrillic[1]{
	{\fontencoding{OT2}\fontfamily{wncyr}\selectfont #1}
		}

\newcommand\mathcyr[1]{\text{\cyrillic{#1}}}
\newcommand\Sha{\textnormal{\mathcyr{Sh}}} 
	
	\usepackage{titletoc}

\begin{document}

\title[Generalized cohomological field theories in the higher order formalism]{Generalized cohomological field theories\\ in the higher order formalism}

\author{Vladimir Dotsenko}
\address{Institut de Recherche Math\'ematique Avanc\'ee, UMR 7501, 
Universit\'e de Strasbourg et CNRS, 7 rue Ren\'e-Descartes, 67000 
Strasbourg, France}
\email{vdotsenko@unistra.fr}

\author{Sergey Shadrin}
\address{Korteweg-de Vries Institute for Mathematics, University of 
Amsterdam, P. O. Box 94248, 1090 GE Amsterdam, The Netherlands}
\email{s.shadrin@uva.nl}

\author{Pedro Tamaroff}
\address{Non-linear Algebra Group, Max-Planck-Institut f\"ur 
Mathematik in den Naturwissenschaften, Inselstra\ss e 22, 04103 
Leipzig, Germany}
\email{tamaroff@mis.mpg.de}

\begin{abstract} 
In the classical Batalin--Vilkovisky formalism, the BV operator $\Delta$ is a differential operator of order two with respect to 
the commutative product. In the differential graded setting, it is known that if the BV operator is homotopically trivial, then there 
is a tree level cohomological field theory induced on the homology; this is a manifestation of the fact that the 
homotopy quotient of the operad of BV algebras by $\Delta$ is represented by the operad of hypercommutative algebras.  
In this paper, we study generalized Batalin--Vilkovisky algebras where the operator $\Delta$ is of the given finite order. In that case,
we unravel a new interesting algebraic structure on the homology whenever $\Delta$ is homotopically trivial. 
We also suggest that the sequence of algebraic structures arising in the higher order formalism is a part of a ``trinity'' of remarkable 
mathematical objects, fitting the philosophy proposed by Arnold in the 1990s. 
\end{abstract}

\maketitle

{
\tableofcontents
}

\section{Introduction} 

\subsection{Homotopy quotients by operators of higher order}
Using ideas coming from the BCOV theory \cite{MR1301851}, Barannikov and Kontsevich gave in \cite{MR1609624} a method 
to construct, out of a differential graded Batalin--Vilkovisky (dg BV) algebra satisfying certain conditions, a formal Dubrovin--Frobenius manifold
(see also \cite{MR2291791,MR2359274} for other proofs of this result). 
They applied it in the case of the dg BV algebra of polyvector fields on a Calabi--Yau manifold $X$, in which case their results produce 
a structure of a formal Dubrovin--Frobenius manifold on the Dolbeault cohomology of $X$. There is a ``mirror partner'' of the
latter result for the de Rham cohomology of a symplectic manifold satisfying the hard Lefschetz theorem, due to Merkulov \cite{MR1637093}. In principle, one can drop some pieces of structure of a Dubrovin--Frobenius 
manifold, like the inner product (then one gets a flat F-manifold structure on the cohomology of $X$) and the unit (then the resulting structure is that of a hypercommutative 
algebra, in which case the result of Merkulov is available for any Poisson manifold, see \cite{MR3323198}). In this case, the underlying algebraic statement follows from the explicit description of the homotopy quotient of the operad $\BV$ by the Batalin--Vilkovisky operator $\Delta$, established by Drummond-Cole and Vallette \cite{MR3029946}; we recall that the homotopy quotient encodes dg BV algebras in which the action of $\Delta$ on the homology is trivial. In general, a homotopy quotient of an operad is not a very tractable thing: it is an object of the homotopy category of dg operads, and to represent it one needs to consider either an operad with nontrivial differential or an operad with zero differential but nontrivial higher structures. Miraculously, the homotopy quotient of $\BV$ by $\Delta$ can be represented by a non-differential operad, meaning that an unexpected formality theorem holds. The operad in question is the operad of hypercommutative algebras $\Hycomm$. A hypercommutative algebra is a graded vector space $A$ equipped with infinitely many fully symmetric operations $m_n$, $n\geqslant 2$, of degree $2(n-2)$ that satisfy a system of relations stating that the $\hbar$-linear binary operation
 \[
a,b\mapsto \sum_{n\geqslant 0}\frac{\hbar^n}{n!}m_{n+2}(a,b,x,\ldots,x)\in A[[\hbar]]
 \]  
is associative for any $x\in A$ \cite{MR1363058}; a multilinear operadic version of these relations is
 \[
\sum_{\substack{I\sqcup J =\{1,\ldots,n\} \\ i\in I;\, j,k\in J}} m_{I\sqcup\{\star\}}\circ_\star m_{J}=\sum_{\substack{I\sqcup J = \{1,\ldots,n\} \\ j\in I;\, i,k\in J}} m_{I\sqcup\{\star\}}\circ_\star m_{J} . 
 \]

Algebraically, a BV algebra may be defined as a commutative algebra equipped with a unary operation $\Delta$ of homological degree one which squares to zero and is a differential 
operator of order at most two with respect to the commutative product. About 20 years ago Losev posed the following question: what would change in the result of homotopical trivialization if one considers as an 
input an analog of a BV algebra, where the ``BV operator'' is a differential operator of an arbitrary finite order, not necessarily of order two. Losev suggested that while operators 
of order other than two would not be relevant for string theory applications, the answer might be algebraically meaningful and beautiful. 
(In fact, higher order BV-like operators did actually emerge in the theoretical physics literature, see, for example, \cite{batber1,batber2}.) The mathematical proposal to generalize BV 
algebras in this way goes back to Akman \cite{MR1466615} who also highlighted the importance of Koszul's definition of higher order differential operators \cite{MR837203} in 
this context. The higher order condition also emerges when considering the ``commutative homotopy BV algebras'' \cite{MR1764440}, however, in this paper we view an operator of higher order as
an object of its own merit rather than a higher homotopy for an operator of order two.

In the past decade or so, a number of results emerged that strongly suggest that the case of an operator of order different from two may be both tractable and interesting. Most of these results 
are related to the Givental group action on the space of hypercommutative algebra structures on a given vector space $V$, or, more generally, the space $\Hom(\Hycomm,\mathcal O)$ for an arbitrary target 
operad $\mathcal O$ (and not merely the endomorphism operad of $V$). In particular, considering a commutative associative algebra as a trivial example of a hypercommutative algebra, one can use 
the Givental group action in order to formulate for $\Delta$ the condition of being an operator of order at most two~\cite{MR2568443}, and the condition of being an operator of given finite 
order \cite{MR3019721}. It is also possible to use the Givental group action to give an explicit formula for the quasi-isomorphism between $\Hycomm$ and the homotopy quotient of $\BV$ by $\Delta$ 
\cite{KMS}. This suggests that one may be able to use the Givental group action to compute the homotopy quotient by an operator $\Delta$ by higher order, thus answering the question of Losev. This 
is done in the present paper. The resulting description of the homotopy quotient is as simple and beautiful as it is in the case of order two. Namely, in the case of order $k$, the homotopy quotient is once again represented by a non-differential operad; an algebra over that operad is a graded vector space $A$ equipped with infinitely many fully symmetric operations $m_n$, $n=2+t(k-1)$, $t\geqslant 0$, of degree $d=2t$ that satisfy a system of relations stating that the $\hbar$-linear binary operation
 \[
a,b\mapsto \sum_{\substack{t\geqslant 0, \\ n=t(k-1)}}\frac{\hbar^n}{n!}m_{n+2}(a,b,x,\ldots,x)\in A[[\hbar]]
 \]  
is associative for any $x\in A$. The relation between the arity and the degree is displayed in Figure~1. Note that the case $k=2$ gets us back to the usual operad $\Hycomm$.

\begin{center}
\begin{figure}[t]
\begin{tikzpicture}[scale = 0.5]

%Draw a nice grid
\draw[help lines, color=gray!30, dashed] 
	(-0.5,-0.5) grid (14.5,12.5);

%Draw the two axes
\draw[-stealth,line width = .5 pt ] 
	(0,0) -- (15.5,0);
\draw[-stealth,line width = .5 pt] 
	(0,0) -- (0,13);

%Mark numbers in the axes
\foreach \x in {1,2,...,8}
	{\node (\x) at (2*\x-2,-0.5) {\x};}
\foreach \x in {2,4,6,8,10,12}
	{\node (2\x) at (-0.5,\x) {\x};}
	
%Draw the three lines with slopes	
\foreach \d in {1,2,3} 
\draw[line width = 1.25 pt] (2,0) -- (14, 12/\d);
\draw[line width = 1.25 pt] (2,0) -- (2, 12);

%Draw the integer points 
%of the lines for our arities
\foreach \d in {1,2,3}
   \pgfmathsetmacro\result{ceil{6/\d}} 
  \foreach \t in {0,...,\result}
   \node (\d\t ) at (2+ 2*\d*\t, 2*\t) {$\bullet$};
   
%Draw the line with slope infinity  
   \foreach \t in {0,...,6}
   \node (A\t) at (2,2*\t) {$\bullet$} ;

\node (10) at (-0.5,13) {$d$};
\node (11) at (15.5,-0.5) {$n$};
\end{tikzpicture} 
\caption{The lines $a=2+(k-1)d$ and their points corresponding to the generators for $k=1,2,3,4$.}
\end{figure}
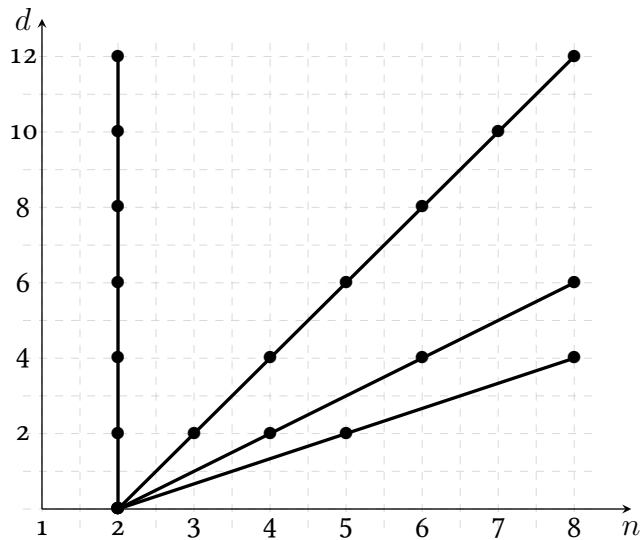
\end{center}

Methods of this paper are entirely algebraic. In particular, our work offers an important methodological feature: an elementary algebraic approach to computations with psi-classes and Givental symmetries 
(in the case of tree level cohomological field theories) which completely bypasses the intersection theory on moduli spaces of curves. The geometry behind the generalized cohomological field theories and the higher order formalism is a topic of our ongoing project that will be featured in a separate paper.  

\subsection{Arnold's trinity of algebraic gravity packages}

In fact, our work led us to study three different types of operad-like structures, each of which is organised into what we are tempted to call an algebraic gravity package, a collection of algebraic 
structures related to what is conventionally called a cohomological field theory. As we mentioned above, in the case of actual cohomological field theories, 
we drop the condition of existence of a unit and a scalar product on the space of 
primary fields, thus restricting ourselves to the tree level case, and making (symmetric) operads an algebraic structure of interest. The basic 
example here is the operad $\Hycomm$ of hypercommutative algebras, considered together with its close relatives: the operad $\BV$ (whose relation to $\Hycomm$ 
we explained above), the operad $\Gerst$ of Gerstenhaber algebras, and the operad $\Grav$ of gravity algebras. One very particular aspect of the algebraic operad $\Hycomm$ 
that comes from geometry and accounts for many remarkable phenomena is the existence of psi-classes on the moduli spaces of stable curves. Those classes control the 
excess intersection theory of the stratifications of moduli spaces, capture the choice of homotopy trivialization of the circle action when the operad of moduli spaces is 
realized as the homotopy quotient of the operad of framed little disks, and give rise to a huge group of automorphisms of the space of representations of $\Hycomm$ called the Givental 
group~\cite{KMS,DSV-circle}.
Another instance of an algebraic gravity package was found in~\cite{DSV-ncHyperCom}; the corresponding algebraic structure is that of a non-symmetric operad. In this case, we have 
a system of operads $\ncHycomm$, $\ncBV$, $\ncGerst$, $\ncGrav$, whose interaction and interrelations repeat those in the symmetric case. In particular, $\ncHycomm$ is the homology operad of the 
non-symmetric operad formed by brick manifolds $\mathcal{B}(n)$, and those manifolds are also equipped with a version of psi-classes whose properties repeat, with the necessary modifications, 
the properties of the psi-classes on the moduli spaces of curves, including an analogue of the Givental group action in this context. 
The third instance of an algebraic gravity package has been only partly known, and is discussed in detail in this article. In this case, the 
corresponding algebraic structure is that of a twisted associative algebra. The part of the picture which is best known is given by a 
topological twisted associative algebra consisting of the Losev--Manin spaces $\mathcal L\mathcal M(n)$ and, on the algebraic level, 
the Losev--Manin cohomological field theories~\cite{ShaZvo-LMCohft}. Note that the spaces $\mathcal L\mathcal M(n)$ once again have a natural 
stratification and the corresponding psi-classes that control the excess intersection theory and give rise to an analog of the Givental 
group action. It turns out that there also exist twisted associative algebra analogues of the operads $\BV$, $\Gerst$, and $\Grav$. We 
introduce them in this paper, noting however that these objects essentially capture the structures observed in topological quantum 
mechanics~\cite{LosevPolyubin1,Lysov}.

An important metamathematical idea proposed by Arnold in the 1990s\footnote{In fact, the first two authors of the paper were present, still as high school students, at 
one of the first lectures he gave on this topic --- or at least one of the first lectures of which there are lecture notes \cite{Arnold1,Arnold2}} is that interesting mathematical objects and theories systematically come in threes. Arnold referred to these triples of objects as ``trinities''; this name appears to have become an accepted technical 
term~\cite{trinities1,MR3910357,trinities2,trinities3}. We argue that the three algebraic gravity packages we discuss are an instance of the same phenomenon. 
One of the most speculative and least developed themes in the theoretical physics literature is a possibly existing range of Gromov--Witten style theories with the sources whose dimensions could be reduced to $1$, $2$, and $4$, with real, complex, and hyperk\"ahler structures, respectively. The case of the source varieties of dimension $2$ with complex structure is the standard Gromov--Witten theory, and it is fully developed and provided a lot of instances of all algebraic structures involved in the symmetric operad instance of algebraic gravity package~\cite{KontMan,Get1,Get2}. 
We hope that the algebraic structures that we present in this paper can help to find a proper setup for $1$-dimensional real case (which we relate to the twisted associative algebra structure, as the similar algebraic setup is emerging in symplectic field theory~\cite{SFT}) and $4$-dimensional hyperk\"ahler case (see, for example, the first steps in that direction in~\cite{LMNS-four,LNS-issues}: though these works do not provide any relevant algebraic structure at all, a wild hope would be that the right Gromov--Witten type theory in this framework would feature the ns operad structures that we mentioned above), and for now we can only point to a few spots in the literature where the related discussions are initiated, see e.g.~\cite{soukhanov2015operad,chekeres2021theory}. 

Another less direct way to check whether talking about Arnold's trinities is the right way to unify the results in the parallel theories (on the levels of the twisted associative algebras, symmetric operads, and non-symmetric operads) that we are presenting is to check whether the additional structures and/or properties uniformly extend to all three theories at hand. Remarkably, it is indeed the case for the results that we discussed above on homotopy quotients by operators of higher order; they also appear to be valid in all three instances of algebraic gravity packages; in particular, in each of those cases the homotopy quotient is represented by a non-differential object exhibiting the same pattern as in the case of quotients by higher order operators described above. The structure of this paper is intended to highlight those similarities: we present the full story for each of the algebraic structures, letting the reader compare them and be convinced of the general pattern present here, which we believe to be unique for the three theories we present. Moreover, there are striking similarities in the combinatorics and geometry of the corresponding topological spaces, which we shall address elsewhere. 

\medskip

\noindent\textbf{Acknowledgements.} We thank Basile Coron, Paul Laubie, Andrey Losev, Sergei Merkulov and Bruno Vallette for useful discussions. 
V.~D. was supported by Institut Universitaire de France, by University of Strasbourg Institute for Advanced Study (USIAS) through the Fellowship USIAS-2021-061 within the French national program ``Investment for the future'' (IdEx-Unistra), and by the French national research agency project ANR-20-CE40-0016. S.~S. was supported by the Netherlands Organization for Scientific Research. S.~S. was also partially supported by International Laboratory of Cluster Geometry NRU HSE, RF Government grant, ag.~№ 075-15-2021-608 dated 08.06.2021

\section{Recollections and preparatory results}

In this section, we recall the key relevant 
definitions and results used in the paper. The
%Removed possisive from Arnold 
conjectural Arnold trinity we consider is a trinity of 
operad-like structures: the twisted associative 
algebra of Losev--Manin spaces, the symmetric 
operad of genus zero Deligne--Mumford spaces, and 
the non-symmetric operad of brick manifolds; this 
trinity is studied using Gr\"obner bases and certain techniques 
of homotopical 
algebra, most notably the Koszul duality (and its 
inhomogeneous version). 

This dictates the structure of our 
recollections section. 
We start with summarizing the 
necessary definitions of the operad-like 
structures we consider, briefly recall the Koszul 
duality theory for those structures, and then 
give a concise summary of what the reader should 
know about the trinity of protagonists of the 
paper. 

\subsection{Conventions}

We work over the ground field $\kk$ of zero 
characteristic. Vector spaces we consider usually 
carry homological degrees, and we view them as 
chain complexes with zero differential; an 
important consequence of the presence of 
homological degrees is the Koszul sign rule for 
symmetry isomorphisms of tensor products. All 
chain complexes are homologically graded, with 
the differential of degree $-1$. To handle 
suspensions of chain complexes, we use the formal 
symbol $s$ of degree $1$, and let $sC_\bullet=\kk 
s\otimes C_\bullet$; similarly, the symbol 
$s^{-1}$ implements desuspensions. We use the 
``topologist's notation'' $\underline{n}$ for 
$\{1,2,\ldots,n\}$. 

\subsection{Species, L-species, and monoidal structures}

The existing terminology in the current literature on operad-like structures has mostly converged, but there is still a large variety of choices used when naming the underlying objects. Throughout this paper, we choose to follow the suit of the foundational paper of Joyal \cite{MR633783} and use the term ``species'' when talking about contravariant functors from the category of finite sets with bijections as morphisms (these are also called symmetric collections, symmetric sequences, and $\mathbb{S}$-modules in the literature), and the term ``L-species'' when talking about contravariant functors from the category of finite totally ordered sets with order preserving bijections as morphisms (these are also called non-symmetric collections, non-symmetric sequences, and $\mathbb{N}$-modules in the literature). Note that a species is completely determined by its values on all finite sets $\underline{n}$ for $n\geqslant 0$; if we consider these spaces with their canonical orders (induced from the ambient order of $\mathbb{Z}$), and the same applies to L-species.

The species we consider take values in one of two types of symmetric monoidal categories (whose monoidal product is denoted $\otimes$ and whose unit is denoted $\mathbf{1}$): the ``set-theoretic'' ones (the category of sets with all maps as morphisms, the category of topological spaces with continuous maps as morphisms etc.; in this case, the monoidal structure is given by the Cartesian product) and the ``$\kk$-linear'' ones (the category of $\kk$-vector spaces with linear maps as morphisms, the category of chain complexes with chain maps as morphisms etc.; in this case, the monoidal structure is given by the tensor product). 
For any species $\mathcal A$, we shall refer to $I$ as the ``set of arguments'' of an element of $\mathcal A(I)$, and to $|I|$ as the ``arity'' of such element, even though in general elements of $\mathcal A$ are not operations acting somewhere. 

\subsubsection{Monoidal structures on species}
There are several monoidal structures which are of crucial importance for us. 
The \emph{Cauchy product} of species $\mathcal A$ and $\mathcal B$ is defined by the formula
 \[
(\mathcal A\otimes\mathcal B)(I)=\coprod_{\substack{J,K\subseteq I,\\ I=J\sqcup K}}\mathcal A(J)\otimes\mathcal B(K). 
 \]
The Cauchy product makes the category of species into a symmetric monoidal category with the unit $\mathbb{K}$, the species for which $\mathbb{K}(\varnothing)=\mathbf{1}$ and $\mathbb{K}(I)=0$ for $I\ne\varnothing$. Monoids in this category are referred to as \emph{twisted associative algebras}. 

The \emph{composite product} of species $\mathcal A$ and $\mathcal B$ is defined by the formula
 \[
(\mathcal A\circ\mathcal B)(I)=\coprod_{n\ge 0}\mathcal A(\underline{n})\otimes_{S_n}(\mathcal B^{\otimes n})(I).
 \]
Here on the right $\mathcal B^{\otimes n}$ is the $n$-fold Cauchy tensor product of $\mathcal B$ (by definition, $\mathcal B^{\otimes 0}=\mathbb{K}$), and $\otimes_{S_n}$ is the coequalizer taking into account the right $S_n$-action on $\mathcal A(\underline{n})$ arising from the contravariance of $\mathcal A$ and the left $S_n$-action on $\mathcal B^{\otimes n}$ which exists since the category of species equipped with the Cauchy product is symmetric monoidal. The composite product makes the category of species into a (very much not symmetric) monoidal category with the unit $\mathbf{1}$, the species for which $\mathbf{1}(I)=\mathbf{1}$ if $|I|=1$ and $\mathbf{1}(I)=0$ for $|I|\ne 1$. Monoids in this category are referred to as \emph{symmetric operads}.

\begin{remark}
The notion of a twisted associative algebra is equivalent to the notion of a left module over the associative operad, for the monoidal structure on the category of species given by the composite product $\circ$; this version of the definition can be used to define twisted $\mathcal P$-algebras for any operad $\mathcal P$. In particular, the embedding of the Lie operad into the associative operad means that every twisted associative algebra $\mathcal A$ has a structure of a twisted Lie algebra given, for two elements $\mathsf{x}\in\mathcal A(I)$ and $\mathsf{y}\in\mathcal A(J)$ by $[\mathsf{x},\mathsf{y}] = \mathsf{x}\mathsf{y}-(-1)^{|\mathsf{x}||\mathsf{y}|}\mathsf{y}\mathsf{x}\in \mathcal A(I\sqcup J)$. We shall use this Lie algebra structure on several occasions throughout the paper
\end{remark}

There are direct analogues of these two monoidal structure for L-species.  
The \emph{Cauchy product} of L-species $\mathcal A$ and $\mathcal B$ is defined by the formula
 \[
(\mathcal A\otimes\mathcal B)(I)=\coprod_{\substack{J,K\subseteq I,\\ I=J+K}}\mathcal A(J)\otimes\mathcal B(K). 
 \]
Here $+$ denotes the ordinal sum: $J+K$ is the ordered set in which every element of $J$ is less than every element of $K$. The Cauchy product makes the category of L-species into a symmetric monoidal category with the unit which we denote by the same letter $\mathbb{K}$, since it is defined by the same rule, even though formally it is an object of a different category. Monoids in this category form a category which is equivalent to the category of non-negatively graded associative algebras; they will not play a special role in this paper and are only defined to give a concise definition of the composite product that follows.

The \emph{composite product} of L-species $\mathcal A$ and $\mathcal B$ is defined by the formula 
 \[
(\mathcal A\circ\mathcal B)(I)=\coprod_{n\ge 0}\mathcal A(\underline{n})\otimes(\mathcal B^{\otimes n})(I).
 \]
The composite product makes the category of species into a (very much not symmetric) monoidal category with the unit which we denote by the same symbol $\mathbf{1}$, since it is defined by the same rule, even though formally it is an object of a different category. Monoids in this category are referred to as \emph{non-symmetric operads}. 

\subsubsection{Shuffle monoidal structures}
There are also two other monoidal structures on L-species whose main purpose within our work is to reduce questions about species to questions about L-species. 
The \emph{shuffle product} of L-species $\mathcal A$ and $\mathcal B$ is defined by the formula
 \[
(\mathcal A\otimes_{\Sha}\mathcal B)(I)=\coprod_{\substack{J,K\subset I,\\ I=J\sqcup K}}\mathcal A(J)\otimes\mathcal B(K). 
 \]
The shuffle product makes the category of L-species into a symmetric monoidal category with the unit $\mathbb{K}$. Monoids in this category are referred to as \emph{shuffle algebras}. 

The \emph{shuffle composite product} of L-species $\mathcal A$ and $\mathcal B$ with $\mathcal B(\varnothing)=0$ is defined by the formula 
 \[
(\mathcal A\circ_{\Sha}\mathcal B)(I)=\coprod_{n\ge 0}\mathcal A(\underline{n})\otimes(\mathcal B^{(n)})(I).
 \]
Here $\mathcal B^{(n)}$ is the ``divided tensor power'' defined by the formula 
 \[
(\mathcal B^{(n)})(I)=
	\coprod_{
		\substack{(I_1,\ldots,I_n)\vdash I
			\\ \text{shuffle partitions}}}
	\mathcal B(I_1)\otimes\cdots \otimes\mathcal B(I_n).
 \]
as the sum runs through all {shuffle partitions}
$I_1\sqcup\cdots\sqcup I_n=I$ of $I$ into 
nonempty subsets; the shuffle condition
requires that $\min(I_s)<\min(I_{s+1})$ for
all $s\in \underline{n-1}$.
The composite product makes the category of species into a (very much not symmetric) monoidal category with the unit $\mathbf{1}$. Monoids in this category are referred to as \emph{shuffle operads}. 

\subsubsection{The forgetful functor}
The importance of the shuffle monoidal structures 
above comes from the following favorable properties they exhibit. To state them, let us use the forgetful functor $\mathcal A \mapsto {\mathcal A}^{\mathrm{f}}$ from the category of finite ordered sets to the category of finite sets which literally forgets the order of a set. This functor defines a forgetful functor from species to L-species given by 
 \[
\mathcal A^{\mathrm{f}}(I)=\mathcal A(I^{\mathrm{f}}).
 \]  
For any two species $\mathcal A$ and $\mathcal B$, we have an isomorphism of L-species
 \[
(\mathcal A\otimes\mathcal B)^{\mathrm{f}}\cong \mathcal A^{\mathrm{f}}\otimes_{\Sha}\mathcal B^{\mathrm{f}} . 
 \]
For any species $\mathcal B$ with $\mathcal B(\varnothing)=0$, we have an isomorphism of L-species
 \[
(\mathcal B^{\mathrm{f}})^{(n)}\cong ((\mathcal B^{\otimes n})_{S_n})^{\mathrm{f}},
 \]
where on the right the coinvariants are computed by a coequalizer taking into account the left $S_n$-action on $\mathcal B^{\otimes n}$ which exists since the category of species equipped with the Cauchy product is symmetric monoidal. This formula easily implies that for any species $\mathcal A$ and $\mathcal B$ with $\mathcal B(\varnothing)=0$ we have an isomorphism of L-species
 \[
(\mathcal A\circ\mathcal B)^{\mathrm{f}}\cong \mathcal A^{\mathrm{f}}\circ_{\Sha}\mathcal B^{\mathrm{f}} .
 \]
These isomorphisms allow to define Gr\"obner bases for twisted associative algebras and for symmetric operads. Due to the presence of symmetries, the usual approach to Gr\"obner bases (relying on an ordering of monomials in the free monoid) fails, see \cite[Prop.~4.2.2.4]{MR3642294} in the case of twisted associative algebras and \cite[Prop.~5.2.2.5]{MR3642294} in the case of symmetric operads. However, the monoidality of the forgetful functor allows us to regard any twisted associative algebra as a shuffle algebra and any symmetric operad as a shuffle operad; the forgetful functor forgets the symmetries but remembers other key properties (cardinality, dimension etc.), so this point of view ends up being highly beneficial. We refer the reader to the monograph \cite{MR3642294} for a systematic information on the arising notions of Gr\"obner bases and to the articles \cite{MR2667136,MR3095223} where this point of view originally emerged (the shuffle monoidal structures for L-species were known to combinatorialists \cite{MR1629341}, but have not been used for algebraic or homotopical purposes). 

\subsection{Bar-cobar duality, Koszul duality, and Gr\"obner bases}

In this section, we offer a brief recollection of some of the key methods of this paper. It is not sensible to offer a stand-alone course in homotopical algebra, so a reader who has not worked with Koszul models of algebras and operads before will end up at a disadvantage; we however hope that this section offers some useful intuition that will allow one to better comprehend the logic of the main part of the paper. 

\subsubsection{Bar-cobar duality}
As we alluded above, the protagonists of this paper are a certain twisted associative algebra, a certain symmetric operad, and a certain non-symmetric operad. Each of these is a monoid in a monoidal category; moreover, these monoids are augmented, that is, equipped with a morphism to the trivial monoid (the unit of the monoidal structure) identical on the unit. In each of these three cases, our monoids can be described using binary operations only; it is clear for the Cauchy product, and well known for (both symmetric and non-symmetric) operads, where the corresponding binary operations are called ``partial compositions'' \cite[Sec.~5.3.4]{MR2954392}. This allows one to define bar constructions of augmented monoids; the bar construction $\mathsf{B}(M)$ of a monoid $M$ is a differential graded cofree conilpotent comonoid co-generated by $s\overline{M}$ (suspension of the kernel of augmentation), and has a differential which is the coderivation that is a sum of elements obtained by ``computing all possible binary operations''. Similarly, for a conilpotent comonoid $C$ one can define its cobar construction $\Omega(C)$ which is the differential graded monoid freely generated by $s^{-1}C$ with the differential arising from the comonoid structure of $C$ (expressed in terms of binary (co)operations). These constructions can actually be slightly modified to allow for differential graded $M$ and $C$, and it turns out that the bar construction of an augmented monoid and the cobar construction of a conilpotent comonoid form a pair of adjoint functors implementing a Quillen equivalence between the corresponding categories. We refer the reader to \cite{MR2954392,MR4105949} for details. 

\subsubsection{Koszul duality}
There is an important class of situations for which 
the bar-cobar adjunction has a simple ``shadow'' in 
the realm of non-differential (co)monoids. Suppose 
that $\mathcal A=\bigoplus_{n\ge 0} \mathcal A_n$ is a non-
negatively graded monoid with $\mathcal A_0=\mathbf{1}$ 
being the unit; such monoids are automatically 
augmented. The bar construction $\mathsf{B}(\mathcal A)$ 
has two gradings, the ``weight'' grading arising from 
the grading of $\mathcal A$, and the syzygy degree which 
counts the number of the suspension symbols $s$ used 
to form an element. The monoid $\mathcal A$ is said to be 
Koszul if the homology of its bar construction is 
concentrated on the diagonal, that is if for each $n
\ge 0$, every element of syzygy degree $n$ of 
$H(\mathsf{B}(\mathcal A))$ has weight grading $n$. It is 
easy to show that a Koszul monoid is generated by 
elements of weight~$1$, and has relations of weight~$2$; 
in other words, it is a \emph{quadratic monoid}. 
In general, the homology of the bar construction of a 
monoid is a comonoid, and the diagonal part of the 
homology is a subcomonoid of that comonoid, called the 
\emph{Koszul dual comonoid} of $\mathcal A$ and denoted $
\mathcal A^{\ac}$. The condition that $\mathcal A$ is Koszul means 
that on the level of the homotopy categories, $
\mathcal A^{\ac}$ is a representative of $\mathsf{B}(\mathcal A)$. 
Koszul duality for operads goes back to \cite{MR1301191} 
and is presented in the way useful for our purposes in 
\cite{MR2954392}. Koszul duality for twisted associative 
algebras is less known, so we feel the need to give the 
reader a more extensive set of references. A 
combinatorial viewpoint of the Koszul duality for twisted 
associative algebras can be found in the paper 
\cite{MR2593312}, a recent paper where explicit 
definitions are not given but their meaning shines 
through the applications is \cite{10.1093/imrn/rnaa215}, 
a recent paper where the Koszul duality for twisted 
associative algebras is discussed with the emphasis on 
the cohomology of coalgebras in \cite{tamaroff2020species}. 
The reader is also invited to consult \cite{MR4103332} for a 
completely different viewpoint using operadic categories. 

\subsubsection{Inhomogeneous Koszul duality}
While there are many Koszul monoids, it is fair to say that 
the property of being Koszul is sufficiently rare. The next 
best thing to hope for is an \emph{inhomogeneous Koszul 
presentation} of a monoid. A quadratic-linear presentation 
(that is, a presentation by generators and relations for 
which any relation is a combination of elements of weights 
one and two in generators) for a monoid $\mathcal A$ leads to a 
\emph{differential graded} Koszul dual comonoid $\mathcal A^{\ac}
$; essentially, as a comonoid it is the Koszul dual of the 
quadratic monoid obtained from $\mathcal A$ by forgetting the 
linear parts of the relations, and its differential 
``restores'' those forgotten linear parts. A quadratic-
linear presentation of $\mathcal A$ is said to be 
\emph{inhomogeneous Koszul} if the homology of the cobar 
construction of $\mathcal A^{\ac}$ is isomorphic to $\mathcal A$. This 
notion goes back to the foundational work of Priddy 
\cite{MR265437}; we refer the reader to more recent articles 
\cite{MR2956319,MR1250981} for details, including a version 
for operads.  A word of warning is in order: in the existing 
literature, one often talks about inhomogeneous Koszul 
monoids, though, unlike being Koszul, the inhomogeneous 
version of the Koszul duality is merely a property of a 
presentation, and every monoid admits at least one 
inhomogeneous Koszul presentation if one picks a linear 
basis as a set of generators, and the corresponding
tautological quadratic-linear relations recording the
`multiplication table' of the monoid.

\subsubsection{Dualizing the comonoids}
The presentation above is, in our view, the ``nicest'' way to deal with the Koszul duality and its inhomogeneous version, since one never has to dualize vector spaces, and so it is easier to work with monoids that have infinitely many generators. However, the price to pay is working with comonoids, and not everyone is willing to pay that price. The Koszul duality assigning to each quadratic monoid its Koszul dual (which is also a monoid) has some standard conventions which we shall now recall. In particular, it is important to recall how suspensions of operads and twisted associative algebras are handled. We introduce the operadic suspension $\mathcal S$ and the operadic desuspension $\mathcal S^{-1}$ as the species with components given by the respective formulas $\mathcal S(n)=\Hom_\kk((\kk s)^{\otimes n}, \kk s)$ and $\mathcal S^{-1}(n)=\Hom_\kk((\kk s^{-1})^{\otimes n}, \kk s^{-1})$ (or, if we work with L-species, the L-species with the same components) with the operad structures given by substitution of multilinear maps into each other. We also introduce the twisted associative algebra suspension $\mathsf{S}$ and the twisted associative algebra desuspension $\mathsf{S}^{-1}$ as the species with components given by the respective formulas $\mathsf{S}(n):=(\kk s^{-1})^{\otimes n}$ and $\mathsf{S}^{-1}(n):=(\kk s)^{\otimes n}$. The Koszul dual operad $\mathcal O^!$ of a quadratic operad $\mathcal O$ is the species with components $\mathcal O^!(n):=\mathcal S^{-1}(n)\otimes(\mathcal O^{\ac}(n))^*$ and the operad structure given by the tensor product of two operad structures of the factors. Similarly, the Koszul dual twisted associative algebra $\mathcal A^!$ of an associative algebra $\mathcal A$ is the species with components $\mathcal A^!(n):=\mathsf{S}^{-1}(n)\otimes(\mathcal A^{\ac}(n))^*$ and the twisted associative algebra structure given by the tensor product of two twisted associative algebra structures of the factors. In
particular, if $\nu$ is a generator of $\mathcal{O}$,
the corresponding dual generator of $\mathcal{O}^!$
will have the same arity, but degree $\mathrm{ar}(\nu)-2-|\nu|$ in the operad case, and $\mathrm{ar}(\nu)-1-|\nu|$
in the twisted associative algebra case. 

\subsubsection{Gr\"obner bases and PBW bases}
To conclude, we must explain how Gr\"obner bases enter into 
our considerations. It happens in two different ways. First, 
we use Gr\"obner bases in the more classical fashion: 
knowing a Gr\"obner basis of relations of a monoid allows 
one to find a linear basis of that monoids by looking at 
\emph{normal forms} (cosets of monomials not divisible by 
the leading monomials of elements of the Gr\"obner basis). 
Second, we may use them to prove Koszulness of various 
monoids. Namely, a monoid with a quadratic Gr\"obner basis 
of relations is Koszul, and a Gr\"obner basis of relations 
consisting of quadratic-linear elements gives an 
inhomogeneous Koszul presentation: the first of these claims 
is proved in \cite[Th.~5.1]{MR3084563}, and the second claim 
is obtained from \emph{op. cit.} by a trivial modification. 
In each of these cases, the linear basis of the monoid is 
made of monomials not divisible by certain quadratic 
elements, and such basis is referred to as the PBW-basis of 
$M$ corresponding to the given (quadratic or quadratic-
linear) Gr\"obner basis of relations. For further 
information on PBW bases, we refer the reader to 
\cite{MR265437} and \cite{MR2574993}.

\subsection{Losev--Manin spaces and their homology}

The remarkable algebraic varieties that we recall in this section was probably first studied by Procesi \cite{MR1252661}; they were then rediscovered by Losev and Manin \cite{MR1786500} when searching for a meaningful notion of an extended cohomological field theory; while not historically precise, we shall refer to them as Losev--Manin spaces. In this section, we give all the necessary statements without proofs, and we invite the reader to consult \cite{MR1786500} for details. 

\subsubsection{Losev--Manin moduli spaces}

\begin{definition}
The $n$-th Losev--Manin space $\mathcal L\mathcal M(n)$ is the toric variety associated to the fan of Weyl chambers of type A in $\mathbb{R}^n$ (the subdivision of $\mathbb{R}^n$ in $n!$ convex cones by the diagonal hyperplanes $x_i=x_j$). The dual polytope of this fan is the permutohedron, and so it is not uncommon in the literature to refer to this fan as the permutohedral fan, and to the space $\mathcal L\mathcal M(n)$ as the permutohedral toric variety.  
\end{definition}

It is easy to show that the space $\mathcal L\mathcal M(n)$ admits a stratification with the open strata indexed by ordered set partitions of $\underline{n}$. The stratum $\mathcal L\mathcal M(n,P)$ indexed by a partition $P=B_1\sqcup \cdots\sqcup B_k$ is, as an algebraic variety, isomorphic to
 \[
\prod_{i=1}^k (\mathbb{C}^\times)^{B_i}/\mathbb{C}^\times .
 \]
In plain words, each point $i\in \underline{n}$ is labelled by a nonzero complex number $z_i$, and the points are considered modulo the simultaneous rescaling of labels in each block $B$ by a nonzero complex number~$\lambda_{B}$. The closure of the stratum $\mathcal L\mathcal M(n,P)$ in 
$\mathcal L\mathcal M(n)$ is the union of all strata $\mathcal L\mathcal M(n,P')$ such that the partition $P$ is obtained from $P'$ by merging some groups of neighbouring blocks. 

Classically, the Losev--Manin spaces are indexed by positive integers, but they may as well be indexed by non-empty finite sets, thus forming a species.
The stratification of Losev--Manin spaces leads to a twisted associative algebra structure on $\mathcal L\mathcal M$: the product $\mathcal L\mathcal M(J)\times\mathcal L\mathcal M(K)\to\mathcal L\mathcal M(I)$ arises naturally from the stratification: the closure of the stratum corresponding to the two-block partition $I=J\sqcup K$ is isomorphic to $\mathcal L\mathcal M(J)\times\mathcal L\mathcal M(K)$, and we take the inclusion of the closure of that stratum as the corresponding product map in the monoid $\mathcal L\mathcal M$. Associativity of this product is obvious. 

\subsubsection{Homology and Koszul duality}

The twisted associative algebra structure on the collection of Losev--Manin spaces induces, via the K\"unneth formula, the same kind of structure on the collection of homologies of these spaces. Let us describe the corresponding twisted associative algebra by generators and relations. 

\begin{proposition} \label{prop:homology-LM}
The homology of the twisted associative algebra of Losev--Manin spaces is isomorphic to the \emph{Losev--Manin algebra}, the twisted associative algebra $\lmHycomm$ generated by fully symmetric elements $\lmm_{\underline{t}}$ of arity $t\geqslant 1$ and of homological degree $2t-2$, subject to the relations 
\begin{equation}
\sum_{\substack{\underline{n}=I\sqcup J,\\ i\in I, j\in J}} \lmm_{I} \lmm_{J}
 -  \sum_{\substack{\underline{n}=I\sqcup J,\\ i\in I, j\in J}} \lmm_{J} \lmm_{I}
    = 0
\end{equation}
for all $n\geqslant 2$
and all $i\ne j\in\underline{n}$.
\end{proposition}  

\begin{proof}
We shall denote by $\lmm_{\underline{t}}$ the fundamental class of the 
space $\mathcal L\mathcal M(\underline{t})$. Note that the product $\lmm_{I} 
\lmm_{J}\in\mathcal L\mathcal M(\underline{n})$ is the fundamental class of the 
closure of the stratum corresponding to the two-block partition 
$\underline{n}=I\sqcup J$, and so the relations of $\lmHycomm$ should be 
viewed as linear relations between the closures of strata indexed by 
$2$-partitions. From the original work of Losev and Manin it already 
follows that these relations hold in the homology of $\mathcal L\mathcal M$, so 
there is a surjective map of twisted associative algebras 
 \[
\lmHycomm\twoheadrightarrow H_\bullet(\mathcal L\mathcal M). 
 \]
Moreover, the total dimension of the homology is equal to the number of chambers of its fan, that is $n!$, and thus in order to establish that the existing surjection is an isomorphism it is sufficient to prove that $n!\geqslant \dim\lmHycomm(\underline{n})$, since a tight upper bound does not leave any room for a kernel. Let us outline an argument using Gr\"obner bases for twisted associative algebras \cite{MR3642294}. The ordering we use here was found in \cite{Laubie}; another ordering that works will be discussed in Proposition \ref{prop:bLMHycomm-Koszul} below. First, we find a linear basis of defining relations. For that, we note that for the relators  
 \[ 
R_{i,j}=\sum_{\substack{\underline{n}=I\sqcup J,\\ i\in I, j\in J}} \lmm_{I} \lmm_{J}
 -  \sum_{\substack{\underline{n}=I\sqcup J,\\ i\in I, j\in J}} \lmm_{J} \lmm_{I}
 \]
of $\lmHycomm$, we have $R_{i,k}+R_{k,j}=R_{i,j}$, and therefore the basis of the vector space of relations is given by the relations $R_{i,i+1}$, so that there are $n-1$ relations in the given arity $n$. Next, we shall define an ordering of monomials of the corresponding free shuffle algebra. For each subset $I$ of $\underline{n}$, let us list its elements in increasing order, and let us order subsets by the reverse dictionary order of these lists. Then the leading term of the relation $R_{i,i+1}$ is $\lmm_{\{1,\ldots,i\}}\lmm_{\{i+1,\ldots,n\}}$. It is clear that the monomials in the corresponding shuffle algebra that are not divisible by these leading terms are 
 \[
\lmm_{I_1}\lmm_{I_2}\cdots\lmm_{I_p}
 \]
where $\max(I_k)>\min(I_{k+1})$ for all $k$. Such elements of arity $n$ are in an obvious one-to-one correspondence with permutations of $\{1,\ldots,n\}$: each permutation may be uniquely written as a sequence of increasing sequences for which the last element of a sequence and the first element of the next one form a descent. This implies that the normal monomials of arity $n$ with respect to the leading terms of quadratic relations of $\lmHycomm$ span a vector space of dimension $n!$, and therefore the dimension of $\lmHycomm(n)$ is at most $n!$. As we indicated before, this is enough to conclude that the surjection $\lmHycomm\twoheadrightarrow H_\bullet(\mathcal L\mathcal M)$ is an isomorphism.
\end{proof}

%It is also useful to describe this algebraic structure from the point of view of the cohomology ring.
%
%\begin{proposition}[\cite{MR1366622}]
%Consider the Boolean lattice $\mathrm{B}_n$ of subsets of $\underline{n}$, ordered by inclusion, and its maximal building set $\overline{\mathrm{B}}_n$ consisting of non-empty sets. We have 
% \[
%D(\mathrm{B}_n,\overline{\mathrm{B}}_n)\cong H^\bullet(\mathcal L\mathcal M(n),\mathbb{Z}), 
% \]
%where $\mathcal L\mathcal M(n)$ is the Losev--Manin moduli space \cite{MR1786500}. 
%Moreover, for $H\ne\underline{n}$, the generator $X_H$ is Poincar\'e dual to the natural divisor $D_H$ of $\mathcal L\mathcal M(n)$, the image of the pushforward of cohomology 
% \[
%H^\bullet(\mathcal L\mathcal M(H^c\sqcup\{\star\})\otimes H^\bullet(\mathcal L\mathcal M(H),\mathbb{Z})\cong H^\bullet(D_H,\mathbb{Z})\to H^{\bullet+2}(\mathcal L\mathcal M(n),\mathbb{Z})
% \]
%is identified with the principal ideal $(X_H)\subset D(\mathrm{B}_n,\overline{\mathrm{B}}_n)$, so that the images of the two special psi-classes of $D_H$ are $X_H\psi^+_H$ and $X_H\psi^-_H$.
%end{proposition}

Our next result is concerned with Koszul duality for our twisted associative algebra. The following result can also be established geometrically in the spirit of \cite[Sec.~4.6]{MR1363058} and of \cite[Sec.~5.1.5]{DSV-ncHyperCom}; we chose to present an algebraic proof since some aspects of that proof will be useful in one of the main results of this article.
 
\begin{proposition}\label{prop:tHycommDual}
The twisted associative algebra $\lmHycomm$ is 
Koszul. The twisted associative algebra 
desuspension of its Koszul dual twisted 
associative algebra is generated by fully 
symmetric elements $\lmg_{t}$ for $t\geqslant 0$ 
all of homological degree $1$, subject to 
the relations 
\begin{align}
\sum_{i\in I} 
	\lmg_{\{i\}}\lmg_{\underline{n}
		\setminus\{i\}} &=\lmg_I\lmg_J 
	\text{ for $\underline{n}=I\sqcup J$ 
		and $|I|>1$}\\
\sum_{i\in\underline{n}}\lmg_{\{i\}}
\lmg_{\underline{n}\setminus\{i\}}&=0 .
\end{align}
\end{proposition}

\begin{proof}
The Gr\"obner basis found in the proof of Proposition \ref{prop:homology-LM} is quadratic, and so implies that the shuffle algebra associated to $\lmHycomm$ is Koszul (monomial quadratic shuffle algebras are Koszul because one can explicitly construct a linear resolution of the trivial module \cite{MR3095223}, and the passage from Koszulness of monomial algebras to Koszulness of algebras with a Gr\"obner basis is done by a usual spectral sequence argument). Because of the properties of the forgetful functor from twisted associative algebras to shuffle algebras \cite{MR3642294}, we conclude that the twisted associative algebra $\lmHycomm$ is Koszul.\\

The Koszul duality between the two twisted associative algebras is established as follows. First, we note that the relations listed above annihilate the relations of $\lmHycomm$ under the pairing between the tensor squares of spaces of generators (the desuspension ensures that we must compute the usual pairing without twisting the scalar products by the sign representations of the symmetric groups). Next, it is clear that all these relations but the last one correspond to subsets of cardinalities $2,\ldots,n-1$, so they span a subspace whose dimension is equal to the dimension of the annihilator of the $(n-1)$-dimensional space of quadratic relations of $\lmHycomm$ in arity~$n$, and hence coincides with it. 
\end{proof}

To highlight similarity with the corresponding results for the Deligne--Mumford spaces, we shall denote the Koszul dual twisted associative algebra of $\lmHycomm$ by $\lmGrav$.

\subsection{Genus zero Deligne--Mumford moduli spaces and their homology}\label{sec:DMrecall}

\subsubsection{Genus zero Deligne--Mumford moduli spaces}

\begin{definition}
The $n$-th genus zero moduli space $\mathcal M_{0,n}$ parametrises smooth complex curves of genus $0$ with $n$ distinct marked points labelled $1, \ldots, n$. The $n$-th \emph{Deligne--Mumford moduli space} $\overline{\mathcal M}_{0,n}$ stable complex stable curves of genus $0$ with $n$ distinct marked points labelled $1, \ldots, n$. By definition, $\overline{\mathcal M}_{0,n}=\varnothing$ for $n<3$. 
\end{definition}

It is well known that the space $\overline{\mathcal M}_{0,n}$ admits a stratification with the open strata indexed by trees with the set of leaves $\underline{n}$. The stratum $\mathcal M_{0,n}(T)$ indexed by tree $T$ is, as an algebraic variety, isomorphic to
 \[
\prod_{v\in V(T)} \mathcal M_{0,n(v)} ,
 \]
where $V(T)$ is the set of internal vertices of $T$, and where $n(v)$ is the number of edges incident to the vertex $v$. The closure of the stratum $\mathcal M_{0,n}(T)$ in 
$\mathcal M_{0,n}$ is the union of all strata $\mathcal M_{0,n}(T')$ such that the tree $T$ is obtained from $T'$ by collapsing some internal edges. 

Let us shift indices by one and consider the space  $\overline{\mathcal M}_{0,1+n}$. We shall label the leaves of trees indexing the stratification by $\{0,1,\ldots,n\}$; the label $0$ will be always chosen as the root of the tree, and we consider $\overline{\mathcal M}_{0,1+n}$ just with the action of $S_n$ permuting the leaves. The Deligne--Mumford spaces $\overline{\mathcal M}_{0,1+n}$ are indexed by positive integers, but they may as well be indexed by non-empty finite sets, thus forming a species $\overline{\mathcal M}_0$. If we formally adjoint the unit, the stratification of Deligne--Mumford spaces leads to a symmetric operad structure on the species $\overline{\mathcal M}_0$. Specifically, to compute the operadic composition of the closed stratum corresponding to a tree $T$ with $k$ leaves and the closed strata corresponding to the trees $T_1$, \ldots, $T_k$, we form the tree obtained by grafting $T_1$, \ldots, $T_k$ at the leaves of $T$, and use the inclusion of the corresponding closed stratum. The symmetric operad axioms are automatic.

%It is also well known (and goes back to \cite{MR1366622}) that for the lattice $\Pi_n$ of set partitions of $\underline{n}$, ordered by reverse refinement, and its minimal building set $\mathrm{G}_n$ consisting of partitions all whose parts but one are singletons, we have $\overline{Y}_{\Pi_n,\mathrm{G}_n}\cong \overline{\mathcal M}_{0,1+n}$.

\subsubsection{(Co)homology and Koszul duality}

The operad structure on the collection of Deligne--Mumford spaces induces, via the K\"unneth formula, the same kind of structure on the collection of homologies of these spaces. Let us describe the corresponding operad  by generators and relations. 

\begin{definition}
The \emph{hypercommutative operad}, denoted $\Hycomm$, is the operad generated by fully symmetric operations $m_t$ of all possible arities $t\geq 2$ and of homological degrees $2t-4$, subject to the relations (for each arity $n\geqslant 3$ and each choice of $i,j,k\in\underline{n}$ such that $|\{i,j,k\}|=3$)
\begin{equation}\label{eq:Hycom}
\sum_{\substack{I\sqcup J =\underline{n} 
	\\ i\in I;\, j,k\in J}} 
		m_{I\sqcup\{\star\}}\circ_\star m_{J}
			=\sum_{\substack{I\sqcup J = \underline{n} \\ j\in I;\, i,k\in J}} m_{I\sqcup\{\star\}}\circ_\star m_{J} .
\end{equation}
\end{definition}

\begin{proposition}[Theorem~3.3 in \cite{MR1363058}]
The homology of the operad of Deligne--Mumford spaces is isomorphic to $\Hycomm$.
\end{proposition}

The Koszul dual operad of $\Hycomm$ is denoted $\Grav$ and is referred to as the operad of gravity algebras. To avoid unnecessary signs, it is advantageous to describe its operadic desuspension. 

\begin{proposition}[Theorem~4.6 in \cite{MR1363058}]\label{prop:DMHycommDual}
The operad $\Hycomm$ is Koszul. The operadic desuspension of its Koszul dual operad is generated by fully symmetric elements $g_n$ of all possible arities $n\geq 2$, each of homological degree $1$ and subject to the relations 
\begin{align}
\sum_{\{i,j\}\subset J} g_{\underline{n}\setminus\{i,j\}\sqcup\{\star\}}\circ_\star g_{\{i,j\}} &= g_{I\sqcup\{\star\}}\circ_\star g_{J} \text{ for 
$\underline{n}=I\sqcup J$ and $|J|>2$} \label{eq:Grav1}\\
\sum_{\{i,j\}\subset\underline{n}} g_{\underline{n}\setminus\{i,j\}\sqcup\{\star\}}\circ_\star g_{\{i,j\}}&=0 .\label{eq:Grav2}
\end{align}
\end{proposition}

\begin{proof}
As a first step, one checks by a direct inspection that the operadic desuspension of the Koszul dual operad is defined precisely by the relations listed in the statement. It is thus sufficient to check that the shuffle operad obtained from the symmetric operad with these relations by applying the forgetful functor is Koszul. For that, we shall use a particular admissible ordering of monomials in the free shuffle operad generated by the elements $g_p$. First, one considers a weight grading that assigns weight $1$ to each generator $g_p$ with $p>2$, and weight $0$ to the generator $g_2$. This leads to a partial order which compares weights of monomials. This already ensures that for $|J|<n-1$, the monomial on the right hand side of \eqref{eq:Grav1} is the leading monomial of that relation. To proceed, one uses word operads \cite{MR4114993}, and considers the monoid of ``quantum monomials'' $\mathsf{QM}=\langle x,y,q\mid xq=qx,yq=qy,yx=xyq\rangle$ and the map from the free shuffle operad generated by the elements $g_p$ to the word operad associated to $\mathsf{QM}$ sending $g_2$ to $(y,y)$ and $g_p$ to $(x,x,\ldots,x)$ for $p>2$. This leads to a partial order which compares the elements of the word operad associated to the given monomials. This ensures that for $|J|=n-1$, the monomial on the right hand side of \eqref{eq:Grav1} is the leading monomial of that relation. Finally, one considers the total order extension of the partial order defined so far using the reverse path-permutation lexicographic ordering \cite[Def.~5.4.1.8]{MR3642294}, so that the leading term of \eqref{eq:Grav2} is $g_{n-1}\circ_{n-1}g_2$. It is established in \cite[Th.~5.16]{MR3084563} that the number of monomials of arity $n$ is equal to the dimension of $\Hycomm^!(n)$, so the defining relations we consider form a Gr\"obner basis. By duality, the same is true for $\Hycomm$ if one considers the opposite ordering for the corresponding free shuffle operad.
\end{proof}

\subsection{Brick manifolds and their homology}

The algebraic varieties that we recall in this section were first studied by Escobar \cite{MR3466428} as a particular case of more general varieties arising in the context of subword complexes for Coxeter groups. The non-symmetric operad structure on the collection of this varieties was studied in detail by the first two authors and Vallette in \cite{DSV-ncHyperCom}. 

\subsubsection{Brick manifolds}

\begin{definition}
The $n$-th \emph{brick manifold} $\mathcal B(n)$ is the toric 
variety associated to the dual fan of Loday's realisation 
of the $n$-th associahedron $K^{n-2}$.
\end{definition}

It is known \cite[Prop.~5.1.3]{DSV-ncHyperCom} that the space $\mathcal B(n)$ admits a stratification with the open strata indexed by planar rooted trees with $n$ leaves. The stratum $\mathcal B(n,T)$ indexed by tree $T$ is, as an algebraic variety, isomorphic to
 \[
(\mathbb{C}^\times)^{n-2-n(T)},
 \]
where $n(T)$ is the number of internal edges of $T$. The closure of the stratum $\mathcal B(n,T)$ in 
$\mathcal B(n)$ is the union of all strata $\mathcal B(n,T')$ such that the tree $T$ is obtained from $T'$ by collapsing some internal edges. 

The brick manifolds are indexed by positive integers, but the definition is amenable to allow any non-empty finite ordered sets of indices, thus leading to an L-species $\mathcal B$.
If we formally adjoin a unit, the stratification of brick manifolds leads to non-symmetric operad structure on~$\mathcal B$. Specifically, to compute the operadic composition of the closed stratum corresponding to a tree $T$ with $k$ leaves and the closed strata corresponding to the trees $T_1, \ldots, T_k$, we form the tree obtained by grafting $T_1, \ldots, T_k$ at the leaves of $T$, and use the inclusion of the corresponding closed stratum. The non-symmetric operad axioms are automatic.

\subsubsection{Homology and Koszul duality}

The non-symmetric operad structure on the collection of brick manifolds induces, via the K\"unneth formula, the same kind of structure on the collection of homologies of these spaces. Let us describe the corresponding non-symmetric operad  by generators and relations. 

\begin{definition}
The \emph{noncommutative hypercommutative operad}, 
denoted $\ncHycomm$, is the non-symmetric operad 
generated by operations $\mu_t$ of all possible 
arities $t\geqslant 2$ and of homological degrees 
$2t-4$, subject to the relations (for each arity 
$n\geqslant 3$ and each choice of 
$k\in \underline{n-1}$)
\begin{equation}\label{eq:ncHycomm}
\sum_{i=1}^{k-1} m_{n-k+i}\circ_i m_{k-i+1}=\sum_{i=k+1}^n m_{n-i+k}\circ_k m_{i-k+1} .
\end{equation}
\end{definition}

\begin{proposition}[Theorem 5.1.1. in \cite{DSV-ncHyperCom}]
The homology of the non-symmetric operad of brick manifolds is isomorphic to $\ncHycomm$.
\end{proposition}

The Koszul dual non-symmetric operad of $\ncHycomm$ is denoted $\ncGrav$ and is referred to as the operad of noncommutative gravity algebras. To avoid unnecessary signs, it is advantageous to describe its operadic desuspension. The following result was proved in \cite{DSV-ncHyperCom}; we recall the proof since some of its aspects will be useful for us. 

\begin{proposition}\label{prop:NCHycommDual}
The non-symmetric operad $\ncHycomm$ is Koszul. The operadic desuspension $\mathcal S^{-1}\ncGrav$ of its Koszul dual operad $\ncGrav$ is generated by operations $\gamma_t$ of all possible arities $t \geqslant 2$ and of homological degrees $1$ each, 
subject to the relations 
\begin{align}
\sum_{j=r}^{r+p-2}\gamma_{n-1}\circ_j\gamma_2 &=\gamma_{n-p+1}\circ_r\gamma_p  \,\text{ for all  $3\leqslant p\leqslant r+p-1\leqslant n$,}\label{eq:nsGrav1}\\
\sum_{j=1}^{n-1}\gamma_{n-1}\circ_j\gamma_2 &=0 \, \text{ for all $n\geqslant 3$}.\label{eq:nsGrav2}
\end{align}
\end{proposition}

\begin{proof}
As a first step, one checks by a direct inspection that the operadic desuspension of the Koszul dual operad is defined precisely by the relations listed in the statement.
It is shown in \cite[Th.~4.2.1]{DSV-ncHyperCom} that there exists an ordering of monomials in the free non-symmetric operad generated by the elements $\gamma_p$ for which the monomial on the right hand side of \eqref{eq:nsGrav1} is the leading monomial of that relation, and the monomial $\gamma_{n-1}\circ_1\gamma_2$ is the leading term of  \eqref{eq:nsGrav2}, and that the normal monomials with respect to these leading terms form a PBW basis of the corresponding operad, so the defining relations we consider form a Gr\"obner basis. By duality, the same is true for $\ncHycomm$ if one considers the opposite ordering for the corresponding free non-symmetric operad.
\end{proof}

\section{Homotopy quotients in the case of twisted associative algebras}

\subsection{Twisted associative differential operators}

We begin with defining what plays the role of Batalin--Vilkovisky algebras and their analogues for higher order derivations in the context of twisted associative algebras. For that, we use interpretation of higher order derivations via the  ``Koszul braces'', and we use the construction of the latter using the twisting procedure and gauge symmetries \cite{MR3510210,MR3629664,MR3385702} to introduce their suitable analogue in the case of twisted associative algebras. 

\begin{definition}
Let $\mathcal A$ be a twisted associative algebra and let $\mathsf{m}\in\mathcal A(\underline{1})_0$ satisfy $[\mathsf{m}_{\{1\}},\mathsf{m}_{\{2\}}]=0$. For an element $\mathsf{f}\in\mathcal A(\varnothing)$, we define the $k$-th \emph{Koszul-like brace} $\mathsf{b}_k^\mathsf{f}\in\mathcal A(\underline{k})$ recursively by the formulas
\begin{align*}
\mathsf{b}_0^\mathsf{f} &:=\mathsf{f},\\
\mathsf{b}_{k+1}^\mathsf{f} &:=[\mathsf{b}_{k}^\mathsf{f},\mathsf{m}_{\{k+1\}}] .
\end{align*}
\end{definition}

As an example, $\mathsf{b}_1^\mathsf{f}=\mathsf{f}\mathsf{m}_{\{1\}}-\mathsf{m}_{\{1\}}\mathsf{f}$, and 
\begin{align*}
\mathsf{b}_2^\mathsf{f}&=(\mathsf{f}\mathsf{m}_{\{1\}}-\mathsf{m}_{\{1\}}\mathsf{f})\mathsf{m}_{\{2\}}-\mathsf{m}_{\{2\}}(\mathsf{f}\mathsf{m}_{\{1\}}-\mathsf{m}_{\{1\}}\mathsf{f})\\
&=  
\mathsf{f}\mathsf{m}_{\{1\}}\mathsf{m}_{\{2\}}-\mathsf{m}_{\{1\}}\mathsf{f}\mathsf{m}_{\{2\}}-\mathsf{m}_{\{2\}}\mathsf{f}\mathsf{m}_{\{1\}}+\mathsf{m}_{\{2\}}\mathsf{m}_{\{1\}}\mathsf{f} .
\end{align*}

\begin{definition} For a twisted associative algebra $\mathcal A$ and the given element $\mathsf{m}\in\mathcal A(\underline{1})_0$ with 
 \[
[\mathsf{m}_{\{1\}},\mathsf{m}_{\{2\}}]=0 ,
 \] 
we say that an element $\mathsf{f}\in\mathcal A(\varnothing)$ is of differential order at most $k$ (with respect to $\mathsf{m}$) if $\mathsf{b}_{k}^\mathsf{f}=0$.
\end{definition}

In the ``classical'' case, the operad of Batalin--Vilkovisky algebras can be defined as the operad controlling a commutative algebra equipped with an odd element squaring to zero which is a differential operator of order at most two. Mimicking this definition, we define the following algebraic object.

\begin{definition}
The twisted associative algebra $\lmBV$ is generated by an element $\Delta$ or arity $0$ and homological degree $1$ and an element $\lmm$ of arity $1$ and homological degree $0$ subject to the relations
\begin{align*}
\Delta^2 &=0,\\
[\lmm_{\{1\}},\lmm_{\{2\}}]&=0,\\
\mathsf{b}_{2}^\Delta &=0.
\end{align*}
\end{definition}

It is well known that in the case of operads, the operad of Batalin--Vilkovisky algebras admits a particularly useful quadratic-linear presentation. Let us introduce a similar presentation in our case. 

\begin{proposition}\label{prop:tBV-qlin}
Consider the twisted associative algebra generated by an element $\Delta$ or arity $0$ and homological degree $1$, an element $\lmm$ of arity $1$ and homological degree $0$, and an element $\lml$ of arity $1$ and homological degree $1$ subject to the relations
\begin{align*}
\Delta^2 &= 0  &
[\lmm_{\{1\}},\lmm_{\{2\}}]&=0
\\
[\Delta,\lmm_{\{1\}}]&=\lml_{\{1\}} &
[\Delta,\lml_{\{1\}}]&=0 
\\
[\lml_{\{1\}},\lmm_{\{2\}}]&=0 
&[\lml_{\{1\}},\lml_{\{2\}}]&=0. 
\end{align*}
This twisted associative algebra is isomorphic to $\lmBV$. 
\end{proposition}

\begin{proof}
The third defining relation allows us to eliminate the element $\lml_{\{1\}}$. The relation $[\Delta,\lml_{\{1\}}]=0$ becomes $[\Delta,[\Delta,\lmm_{\{1\}}]]=0$, which follows from $\Delta^2=0$ by noticing that $\Delta^2=\frac12[\Delta,\Delta]$ and using the Jacobi identity. The relation $[\lml_{\{1\}},\lmm_{\{2\}}]=0$ becomes $[[\Delta,\lmm_{\{1\}}],\lmm_{\{2\}}]=0$, which is precisely $\mathsf{b}_{2}^\Delta=0$. Finally, the relation $[\lml_{\{1\}},\lml_{\{2\}}]=0$ becomes $[[\Delta,\lmm_{\{1\}}],[\Delta,\lmm_{\{2\}}]]=0$, which, thanks to the Jacobi identity, follows from $[\Delta,[\Delta,\lmm_{\{1\}}]]=0$ and $[[\Delta,\lmm_{\{1\}}],\lmm_{\{2\}}]=0$. We see that after the elimination of the redundant element, we end up with the defining relations of $\lmBV$. 
\end{proof}

\begin{remark}
The quadratic-linear presentation of $\lmBV$ we just obtained leads to a useful observation. The evaluation $\lmBV(V)$ of this species on a vector space $V$ produces an associative algebra which is made of $S(V\oplus sV)$ and $\kk[\Delta]/(\Delta^2)$: precisely, it is the semidirect product for which $\Delta$ acts on $S(V\oplus sV)$ as the Koszul differential. In this way, one can say that $\lmBV$ is the ``universal Koszul complex''. This structure has also appeared in ``topological quantum mechanics'', see \cite{Lysov}. The twisted associative algebra may also interpreted as a subalgebra of the twisted associative algebra of differential operators on the tensor algebra of the super vector space $\kk^{1\mid 1}$, see \cite{MR2734329}. Finally, forgetting about $\Delta$ leads to a natural definition of the twisted associative algebra $\lmGerst$, which fits in the general picture of algebraic gravity packages alluded to in the introduction.
\end{remark}

Examining our definition, it is very easy to generalise it to use other values of the differential order. 

\begin{definition}
The twisted associative algebra $\blmBV{k}$ is generated by an element $\Delta$ or arity $0$ and homological degree $1$ and an element $\lmm$ of arity $1$ and homological degree $0$ subject to the relations
\begin{align*}
\Delta^2&=0,\\
[\lmm_{\{1\}},\lmm_{\{2\}}]&=0,\\
\mathsf{b}_{k}^\Delta&=0.
\end{align*}
\end{definition}

Our goal in this section is to determine the homotopy quotient of $\blmBV{k}$ by $\Delta$. For the time being, this calculation may be seen as a toy model for two similar calculations of the following sections, dealing with generalisations of Batalin--Vilkovisky algebras and their noncommutative versions. We argue that the three calculations follow the exact same pattern, and so one may as well learn that pattern in the case of twisted associative algebras which exhibit slightly simpler combinatorics than the two other cases. We shall see that the description of the homotopy quotient for the given differential order $k$ merges, in a certain way, the cases of $k=1$ and $k=2$, and so we start with discussing those two cases. 

\subsection{The case of differential order one}\label{sec:orderone} The twisted associative algebra $\blmBV{1}$ is generated by an element $\Delta$ of arity $0$ and homological degree $1$ and an element $\lmm$ of arity $1$ and homological degree $0$ subject to the relations
\begin{equation}
\Delta^2=[\lmm_{\{1\}},\lmm_{\{2\}}]=[\Delta,\lmm_{\{1\}}]=0.
\end{equation}
It is a (homogeneous) quadratic twisted associative algebra, and one can immediately see that it is Koszul; in fact, it has a quadratic Gr\"obner basis for which the corresponding PBW basis consists of elements 
 \[
\lmm_{\{1\}}\cdots\lmm_{\{n\}}  \text{ and } \Delta\lmm_{\{1\}}\cdots\lmm_{\{n\}} \text{ for $n\geqslant 0$.}
 \]
The Koszul dual twisted associative coalgebra $\blmBV{1}^{\ac}$ has a PBW basis consisting of the elements 
 \[
u_{n,r}:=(s\Delta)^r(s\lmm_{\{1\}})\cdots (s\lmm_{\{n\}})  
	\text{ for $r$ and $n\geqslant 0$},
 \]
with the obvious ``shuffle'' coproduct 
 \[
u_{\underline{n},r}\mapsto \sum_{r_1+r_2=r}\sum_{I_1\sqcup I_2=\underline{n}} u_{I_1,r_1}\otimes u_{I_2,r_2}.
 \]
The homotopy quotient of $\blmBV{1}$ by $\Delta$ is represented by the dg twisted associative algebra which is the quotient of the cobar construction $\Omega(\blmBV{1}^{\ac})$ by the two-sided ideal generated by the elements $s^{-1}u_{0,r}$, $r>0$. Thus quotient creates from the cobar construction another free twisted associative algebra equipped with a quadratic differential: it is generated by the elements $s^{-1}u_{\underline{n},r}$ with $n>0$, and the differential 
 \[
\partial(s^{-1}u_{\underline{n},r})=\sum_{r_1+r_2=r}\sum_{\substack{I_1\sqcup I_2=\underline{n}\\ I_1,I_2\ne\varnothing }} s^{-1}u_{I_1,r_1}\otimes s^{-1}u_{I_2,r_2}.
 \]
This differential is the differential of the cobar construction of a twisted associative coalgebra which is the quotient of $\blmBV{1}^{\ac}$ by the \emph{subspace} spanned by the elements $u_{\varnothing,r}$ for all $r>0$. The linear dual twisted associative algebra is generated by the elements $u_{\{1\},p}$ for $p\geqslant 0$, of arity $1$ and homological degree $2p+1$ subject to the relations
\begin{align*}
[u_{\{1\},p},u_{\{2\},q}] &=0,\\
u_{\{1\},p}u_{\{2\},q} &=u_{\{1\},p-1}u_{\{2\},q+1}.
\end{align*}
It is easy to see that this algebra is Koszul; in fact, it has a quadratic Gr\"obner basis for which the linear basis $\{u_{\underline{n},r}\}_{n>0}$, corresponds to a PBW basis $\{u_{\{1\},0}\cdots u_{\{n-1\},0} u_{\{n\},r}\}$. Therefore, the homology of its cobar complex is the Koszul dual twisted associative algebra which has generators $\lmm^{2p}$ of arity $1$ and homological degree $2p$ for $p\geqslant 0$ with the relations 
 \[ 
\left[\sum_{p=0}^\infty \lmm_{\{1\}}^{2p},\sum_{q=0}^\infty \lmm_{\{2\}}^{2q}\right]=0.
 \]
To make sense of these relations, one has to separate the terms by total homological degree $2d\geqslant 0$ into infinitely many relations 
 \[
\sum_{p+q=d}\left[\lmm_{\{1\}}^{2p},\lmm_{\{2\}}^{2q}\right]=0
 \]
involving finitely many terms each. We established the following result.

\begin{theorem}
The homotopy quotient of $\blmBV{1}$ by $\Delta$ is represented by the twisted associative algebra with generators $\lmm^{2p}$ of arity $1$ and homological degree $2p$ for $p=0,1,2,\dots$ subject to the relations 
 \[
\sum_{p+q=d}[\lmm_{\{1\}}^{2p},\lmm_{\{2\}}^{2q}]=0 \text{
for $d\geqslant 0$}.
 \]
\end{theorem}

\subsection{The case of differential order two}\label{sec:ordertwo} Our arguments here are mimicking those of \cite{MR3029946}. 
The first step is to follow the path of \cite{MR2956319} and to apply the inhomogeneous Koszul duality theory \cite{MR1250981}. Specifically, we start with the quadratic-linear presentation of $\lmBV$ found in Proposition \ref{prop:tBV-qlin}. One has to show that this presentation of $\lmBV$ is inhomogeneous Koszul, so we have a model of this algebra given by the cobar construction of a certan dg twisted associative coalgebra. The next step is to use that model to determine the homotopy quotient. For that, one computes the homology of the differential induced by $d_1$, determine sufficient information about the minimal model of $\lmBV$, and then take the quotient of the latter by the two-sided ideal generated by the elements $s^{-1}(\delta^r)^\vee$ for $r>0$. 

The fact that the quadratic-linear presentation of $\lmBV$ is inhomogeneous Koszul is established by a direct calculation. To make the further steps, we shall temporarily dualise everything, and work with the Koszul dual dg twisted associative algebra; it is generated by an element $\delta$ of arity $0$ and homological degree $-2$, an element $\lmn$ of arity $1$ and homological degree $-1$, and an element $\lmk$ of arity $1$ and homological degree $-2$ subject to the relations 
\begin{align}
[\lmn_{\{1\}},\lmn_{\{2\}}]&=[\lmn_{\{1\}},\lmk_{\{2\}}]=[\lmk_{\{1\}},\lmk_{\{2\}}]=0,\\
[\delta,\lmn_{\{1\}}]&=[\delta,\lmk_{\{1\}}]=0,
\end{align}
and additionally it has a differential $\partial$ given by 
 \[
\partial(\delta)=\partial(\lmn_{\{1\}})=0, \quad \partial(\lmk_{\{1\}})=\delta\lmn_{\{1\}} .
 \]
On the level of coalgebras, we would compute the homology of the differential on the Koszul dual coalgebra and then use homological perturbation to construct a new differential on the cobar complex. Since we working with duals, we need to compute the homology algebra with respect to the above differential, and then construct a twisted $A_\infty$-algebra structure on that homology using homotopy transfer. 

A $\kk$-linear basis of the component $\lmBV^!(n)$ is given by the elements 
\begin{align*}
    \delta^r \lmk_{\{i_1\}}\cdots \lmk_{\{i_p\}}\lmn_{\{j_1\}}\cdots \lmn_{\{j_q\}} & \text{ for all $r\geqslant 0$, 
    each $I\sqcup J=\underline{n}$ with}
    \\   I=\{i_1<\cdots < i_p\} &\text{ and } J=\{j_1<\cdots < j_q\} .
\end{align*}
Since the differential $\partial$ is ``almost'' the Koszul differential, the homology of $\lmBV^!(\underline{n})$ almost vanishes: it is equal to the span of all $\delta^r$ in arity $0$ and, in each arity $n>0$, to the vector space spanned by the elements
 \[
u_{{\underline{n}},J}:= \left(\sum_{u=1}^p\lmk_{i_1}\cdots\lmn_{i_u}\cdots\lmk_{i_p}\right)\lmn_{j_1}\cdots\cdots\lmn_{j_q} 
 \]
with $I=\{i_1,\ldots,i_p\}\subset\underline{n}$, $J=\{j_1,\ldots,j_q\}=\underline{n}\setminus I$. Such an element may be informally thought of as $\frac{1}{\delta}\partial(\lmk_{i_1}\cdots\lmk_{i_p})\lmn_{j_1}\cdots\lmn_{j_q}$; it would have been absent in the homology of the actual Koszul complex but is not a boundary of $\partial$ because of the absence of $\delta$ in it. 

As it was the case in \cite{MR2956319}, we do not really need to know the \emph{full} twisted $A_\infty$-structure on the homology. Since we shall take the quotient of the cobar construction by $s^{-1}(\delta^r)^\vee$ for $r>0$, it is only necessary for us to compute the $A_\infty$-operations $\mu_t(u_{\underline{n_1},J_1},\ldots,u_{\underline{n_t},J_t})$. Since one can choose the contracting homotopy for $d_1$ which vanishes on elements that do not contain $\delta$, we see that $\mu_t(u_{\underline{n_1},J_1},\ldots,u_{\underline{n_t},J_t})=0$ for $t\geqslant 3$, so the twisted $A_\infty$-algebra structure on the quotient is simply a twisted associative algebra structure. Moreover, this algebra is multiplicatively generated by the elements $\lmg_{\underline{n}}:=u_{{\underline{n}},\varnothing}$, and these elements manifestly satisfy the relations
\begin{align}
\sum_{i\in I} \lmg_{\{i\}}
	\lmg_{\underline{n}\setminus\{i\}}
		&=\lmg_I\lmg_J \text{ for all $\underline{n}=I\sqcup J$ with $|I|>1$},\\
\sum_{i\in\underline{n}}
	\lmg_{\{i\}}
		\lmg_{\underline{n}\setminus\{i\}} &=0 \text{ for all $n\geqslant 3$}.
\end{align}
which are precisely the relations of the twisted associative algebra $\lmHycomm^!$ determined in Proposition \ref{prop:tHycommDual}. 
Indeed, the vanishing of
 \[
\sum_{i\in\underline{n}}\lmg_{\{i\}}\lmg_{\underline{n}\setminus\{i\}}
 \]
is just another way to say that the Koszul differential squares to zero, and
 \[
\sum_{i\in I} \lmg_{\{i\}}\lmg_{\underline{n}\setminus\{i\}}=
\sum_{i\in I} \lmn_{\{i\}}\left(\lmg_{I\setminus\{i\}}\lmk_{j_1}\cdots\cdots\lmk_{j_q}+ 
\prod_{i'\in I, i'\ne i}\lmk_{I'}\lmg_J\right)=\lmg_I\lmg_J,
 \]
since we already know that $\sum_{i\in I} \lmn_{\{i\}}\lmg_{I\setminus\{i\}}=0$. 

From that proposition, we know that the relations of that twisted associative algebra form a Gr\"obner basis, leading to a PBW basis of the quotient of the form 
 \[
\lmg_{\{1,\ldots, p_1\}}\lmg_{\{p_1+1,\ldots, p_1+p_2\}}\cdots\lmg_{\{p_1+\cdots+p_{m-1}+1,\ldots, p_1+\cdots+p_m\}}.
 \]
The cardinality of the set of such elements for fixed arity $n$ is equal to the number of compositions of $n$ with non-zero parts, which, by the usual ``stars-and-bars'' method, is equal to $2^{n-1}$. It remains to note that the dimension of the vector space spanned by the elements $u_{\underline{n},I}$ is also equal to $2^{n-1}$: the total number of elements $\lmk_{\{i_1\}}\cdots \lmk_{\{i_p\}}\lml_{\{j_1\}}\cdots \lml_{\{j_q\}}$ is $2^{n}$, and we take the image of the acyclic Koszul differential on that space, reducing the dimension by half. This shows that the algebra we consider is precisely $\lmHycomm^!$, and the homotopy quotient we wish to compute is the homology of the cobar construction of its dual, that is $\lmHycomm$. We established the following result.

\begin{theorem}
The homotopy quotient of $\lmBV=\blmBV{2}$ by $\Delta$ is represented by $\lmHycomm$.  \qed
\end{theorem}

\subsection{Extended Losev--Manin algebra and Givental symmetries} 

To proceed with the computation of homotopy quotients, we shall introduce a new construction that will allow us to describe the end result. In a way, the twisted associative algebra that we shall describe here is a mix of the answers for the differential orders $1$ and $2$; as we shall see below, the homotopy quotient for every order emerges as its quotient. 

\begin{definition}
The twisted associative algebra $\lmExtHycomm$, which we shall refer to as the \emph{extended Losev--Manin algebra} is the twisted associative algebra generated by fully symmetric elements $\lmm_t^{2p}$ of all possible arities $t\geqslant 1$ and all possible non-negative even homological degrees $2p\geqslant 0$, subject to the relations that can be described as follows. Consider the sums of all generators of a fixed arity $\overline \lmm_t\coloneqq \sum_{p=0}^\infty \lmm_t^{2p}$, and impose, for each arity $n\ge 2$ and each choice of $i,j\in\underline{n}$, the relation
\begin{equation}
\sum_{\substack{\underline{n}=I\sqcup J,\\ i\in I, j\in J}} \overline \lmm_I \overline \lmm_J
=  \sum_{\substack{\underline{n}=I\sqcup J,\\ i\in I, j\in J}} \overline \lmm_J \overline \lmm_I
   .
\end{equation}
Note that if each operation $\overline \lmm_t$ were of homological degree $2t-2$, these identities would define the Losev--Manin twisted associative algebra from Proposition \ref{prop:homology-LM}. In our case, each of these operations is an infinite sum of operations of different homological degrees, and so one should separate these identities by homological degrees into infinitely many relations involving finitely many terms each; the same is implicit in most formulas that follow.
\end{definition}

In the case of moduli spaces and the Losev--Manin spaces, it is well known that the so called psi-classes play important role in algebraic and geometric results involving those spaces. Let us define a formal algebraic version of those classes in the case of $\lmExtHycomm$, where there is no longer a geometric interpretation of those as cohomology classes. We start with a preparatory lemma.

\begin{lemma}\label{lm:action-indep}
Each of the two elements
\[
	\sum_{\substack 
		{I\sqcup J = \underline{n} 
			\\  j\in J
			} 
				}
	\overline \lmm_I  \overline \lmm_{J}
	 , \quad
\sum_{\substack 
	{I\sqcup J = 
		\underline{n} \\ 
		j\in J
		}
			} 
		 \overline \lmm_J 
		 \overline \lmm_I  
\]
do not depend on the choice of $j\in 
\underline{n}$.
\end{lemma}

\begin{proof}
We shall show that the elements of the twisted associative algebra $\lmExtHycomm$ given by the formulas
 \[
\sum_{\substack {I\sqcup J = \underline{n} \\ j\in J} } \overline \lmm_I  \overline \lmm_J,  \quad
\sum_{\substack {I\sqcup J = \underline{n} \\ j\in J} } \overline \lmm_J  \overline \lmm_I,  
 \]
do not depend on the choice of $j\in\underline{n}$. We shall do it for the first element, since the argument for the second one is completely analogous. To show that
 \[
\sum_{\substack {I\sqcup J = \underline{n} \\ j\in J} } \overline \lmm_I  \overline \lmm_J
=\sum_{\substack {I\sqcup J = \underline{n} \\ j'\in J} } \overline \lmm_I  \overline \lmm_J,
 \]
we note that both sides are obtained from the defining relation
 \[
\sum_{\substack {I\sqcup J = \underline{n} \\ j'\in I, j\in J} } \overline \lmm_I  \overline \lmm_J
=\sum_{\substack {I\sqcup J = \underline{n} \\ j\in I, j'\in J} } \overline \lmm_I  \overline \lmm_J
 \]
of $\lmExtHycomm$ by adding to both sides the same element
 \[
\sum_{\substack {I\sqcup J = \underline{n} \\ j,j'\in J} } \overline \lmm_I  \overline \lmm_J.
 \]
This completes the proof of the
lemma.
\end{proof}

\begin{proposition}\label{prop:psi-bimodule}
If, for a formal variable $\psi$ of homological degree $-2$, we let 
\begin{gather}
\psi \overline\lmm_{\underline{n}}:= \sum_{\substack {I\sqcup J = \underline{n} \\ j\in J }} \overline \lmm_I  \overline \lmm_J,\label{eq:leftact}  \\ 
\overline\lmm_{\underline{n}}\psi :=\sum_{\substack {I\sqcup J = \underline{n} \\ j\in J } }\overline \lmm_{J} \overline \lmm_I  \label{eq:rightact} 
\end{gather}
for some $j\in\underline{n}$, these 
formulas define a $\kk[\psi]$-bimodule 
structure on~$\lmExtHycomm$. 
\end{proposition}

\begin{proof}
From Lemma \ref{lm:action-indep}, it 
follows that the right hand sides of our 
formulas do not depend on the choice of
$j\in\underline{n}$. Let us verify that 
Formulas \eqref{eq:leftact} and \eqref{eq:rightact} are compatible with 
the relations of $\lmExtHycomm$, that is 
the action on each relation produces an 
element that vanishes as a consequence of 
some relations. 
We shall check this for the left action 
(the argument for the right action is analogous). We should check that
 \[
\sum_{\substack {I\sqcup J = \underline{n} \\ i\in I, j\in J} } \psi\overline \lmm_I  \overline \lmm_J
=\sum_{\substack {I\sqcup J = \underline{n} \\ j\in I, i\in J} } \psi\overline \lmm_I  \overline \lmm_J.
 \]
Let us use the formula for the action using $i\in I$ on the left, and the formula for the action using $j\in I$ on the right. We obtain
 \[
\sum_{\substack {I'\sqcup I''\sqcup J = \underline{n} \\ i\in I'', j\in J} } \overline \lmm_{I'}\overline \lmm_{I''}  \overline \lmm_J
=\sum_{\substack {I'\sqcup I''\sqcup J = \underline{n} \\ j\in I'', i\in J} } \overline \lmm_{I'} \overline \lmm_{I''}   \overline \lmm_J,
 \]
which, once rewritten as
 \[
\sum_{\substack {I'\sqcup I''\sqcup J = \underline{n} \\ i\in I'', j\in J} } \overline \lmm_{I'}[\overline \lmm_{I''}, \overline \lmm_J]
=0, 
 \]
is easily seen to follow from the defining relations. 

It remains to check that the two actions define a bimodule action, for which it is enough to check that
$
(\psi\overline \lmm_{\underline{n}})\psi=\psi(\overline \lmm_{\underline{n}}\psi)$, 
or, in other words, fixing some $i\in\underline{n}$, that
 \[
\left(\sum_{\substack {I\sqcup J = \underline{n} \\ i\in J} } \overline \lmm_I  \overline \lmm_J\right)\psi=\psi\left(\sum_{\substack {I\sqcup J = \underline{n} \\ i\in I} } \overline \lmm_I  \overline \lmm_J\right). 
 \]
Using that same $i$ once again, we obtain
 \[
\sum_{\substack {I\sqcup J'\sqcup J'' = \underline{n} \\ i\in J'} } \overline \lmm_I  \overline \lmm_{J'}\overline \lmm_{J''}=\sum_{\substack {I'\sqcup I''\sqcup J = \underline{n} \\ i\in I''} } \overline \lmm_{I'} \overline \lmm_{I''}  \overline \lmm_J,
 \]
which are manifestly equal. 
\end{proof}

Let us show how the psi-classes can be used to define symmetries of representations of $\lmExtHycomm$ that generalise the Givental action \cite{ShaZvo-LMCohft}. Precisely, let $\mathcal A$ be an arbitrary dg twisted associative algebra, and consider the set of dg twisted associative algebra morphisms $\Hom(\lmExtHycomm,\mathcal A)$. For a formal variable $z$ of degree $-2$, let us consider the complete Lie algebra $\mathfrak{g}_{\mathcal A}$ which is the kernel of the ``augmentation''
 \[
\bigoplus_{p=0}^\infty \mathcal A(\varnothing)_{2p}\otimes\kk
	\llbracket z\rrbracket
		\twoheadrightarrow \mathcal A(\varnothing)_0
 \]
annihilating elements of $\mathcal A$ of positive homological degrees as well as all positive powers of $z$. The reader is invited to consult \cite{dotsenko2018twisting} for information on complete Lie algebras and the corresponding exponential map. 
We shall now define an action of $\mathfrak{g}_{\mathcal A}$ by infinitesimal symmetries of $\Hom(\lmExtHycomm,\mathcal A)$. For that, we put
\begin{equation}\label{eq:lm-Givental-action}
((rz^p).f) (\overline \lmm_{n}) =  
r f(\psi^{p}\overline \lmm_{n}) + (-1)^{p-1} f(\overline \lmm_{n}\psi^p) r  
+ \sum_{i+j=p-1} (-1)^{j+1} \sum_{I \sqcup J = \underline{n}} f(\overline \lmm_{J}\psi^{j}) r f(\psi^i\overline \lmm_I),
\end{equation}
and extend it linearly to $\mathfrak{g}_{\mathcal A}$. Let us show that this indeed defines infinitesimal symmetries of $\Hom(\lmExtHycomm,\mathcal A)$, that is, it allows us to 
deform morphisms to order one.

\begin{lemma}\label{lem:lm:GiventalLie} 
For any $f\in \Hom(\lmExtHycomm,\mathcal A)$ we have $f+\epsilon r.f\in \Hom_{\kk[\epsilon]/(\epsilon^2)}(\lmExtHycomm[\epsilon]/(\epsilon^2),\mathcal A[\epsilon]/(\epsilon^2))$, and for all $\lambda_1,\lambda_2\in \mathfrak{g}_{\mathcal A}$, we have
 \[
[\lambda_1,\lambda_2].f=\lambda_1.(\lambda_2.f)-\lambda_2.(\lambda_1.f)
 \]  
\end{lemma}

\begin{proof}
We need to check that for each $n\geqslant 1$, and each $f\in \Hom(\lmExtHycomm,\mathcal A)$, the element
\[ f(\overline\lmm_{\underline{n}})+\epsilon ((rz^p).f)(\overline\lmm_{\underline{n}})\in \mathcal A[\epsilon]/(\epsilon^2) \] 
satisfy the defining relations of $\lmExtHycomm$. For the 
constant coefficient, this is exactly the same property of 
$f$. For the coefficient of $\epsilon$, we need to check 
that 
 \[
\sum_{\substack {I\sqcup J = \underline{n} \\ i\in I, j\in J} } ((rz^p).f)(\overline \lmm_I)  f(\overline \lmm_J)+
f(\overline \lmm_I)  ((rz^p).f)(\overline \lmm_J)
 \]
is symmetric in $i,j\in\underline{n}$. We note that $((rz^p).f)(\overline \lmm_I)  f(\overline \lmm_J)$ is equal to 
\begin{multline}\label{eq:der-left}
\underbrace{r f(\psi^{p}\overline \lmm_I)f(\overline \lmm_J)}_{\textrm{Type A1}}+ (-1)^{p-1} 
\underbrace{f(\overline \lmm_I\psi^p) r f(\overline \lmm_J)}_{\textrm{Type A2}}
+ \sum_{a+b=p-1} (-1)^{b+1} \sum_{I' \sqcup I'' = I} \underbrace{
f(\overline \lmm_{I''}\psi^{b})
	 r f(\psi^a\overline \lmm_{I'})
	 	f(\overline \lmm_J)}_{\textrm{Type A3}} ,
\end{multline}
and that $f(\overline \lmm_I)  ((rz^p).f)(\overline \lmm_J)$ is equal to
\begin{multline}\label{eq:der-right}
\underbrace{f(\overline \lmm_I)
r f(\psi^{p}\overline \lmm_J)}_{\textrm{Type B1}} + (-1)^{p-1}
\underbrace{ f(\overline \lmm_I)f(\overline \lmm_J\psi^p) r }_{\textrm{Type B2}}
+ \sum_{a+b=p-1} (-1)^{b+1} \sum_{J' \sqcup J'' = J} \underbrace{
	f(\overline \lmm_I)f(\overline \lmm_{J''}\psi^{b}) r f(\psi^a\overline \lmm_{J'})}_{\textrm{Type B3}}.
\end{multline}
As shown, we see that $\sum_{\substack {I\sqcup J = \underline{n} \\ i\in I, j\in J} }((rz^p).f)(\overline \lmm_I)  f(\overline \lmm_J)$
contains three types of terms. The terms of
Type A1 add up to 
\[
\sum_{\substack {I\sqcup J = \underline{n}
	 \\ i\in I, j\in J} }
	 	r f(\psi^{p}\overline \lmm_I)
	 		f(\overline \lmm_J) =
	 		r f \psi^{p}\sum_{\substack {I\sqcup J = \underline{n}
	 \\ i\in I, j\in J} }
	 	\overline \lmm_I\overline\lmm_J
\]
which is symmetric in $i$ and $j$ due to the defining
relations of our algebra.  
The terms of Type A3 can be rewritten
using the explicit formulas for the action of $\kk[\psi]$
 from Proposition \ref{prop:psi-bimodule}, as 
follows:
\begin{align*}
	\sum_{I' \sqcup I'' = I}
		 f(\overline \lmm_{I''}\psi^{b}) r 
		f(\psi^a\overline \lmm_{I'})
			f(\overline \lmm_J) 		
															&= 
	\sum_{\substack {I\sqcup J = \underline{n}
	 \\ i\in I, j\in J} }\sum_{I' \sqcup I'' = I}
		 f(\overline \lmm_{I''}\psi^{b}) r 
		f(\psi^a\overline \lmm_{I'}\overline \lmm_J)
				\\											 &= 	 		\sum_{\substack 
		{ I \subseteq 
			\underline{n}\\ i\in I\supseteq I''}
	 			}
	 			\sum_{
	 				\substack{
	 	I' \sqcup J = 
	 \underline{n}\smallsetminus I''\\  j\in J}}
		 f(\overline \lmm_{I''}\psi^{b}) r 
		f(\psi^a\overline \lmm_{I'}\overline \lmm_J)
		\\ &= 
	\sum_{\substack 
		{ I\subseteq \underline{n}\\ i\in I\supseteq I''}
	 			} 
	 			f(\overline \lmm_{I''}\psi^{b}) r 
		f(\psi^{a+1}\overline \lmm_{\underline{n}\smallsetminus I''}	).	 
\end{align*}
At the same time, this term can be rewritten
in the following form
\[
\sum_{\substack 
		{ A\sqcup B = \underline{n}\\ i\in A}
	 			} 
	 			f(\overline \lmm_A\psi^{b}) r 
		f(\psi^{a+1}\overline \lmm_{B}	)	
		+
		\sum_{\substack 
		{ A\sqcup B = \underline{n}\\ i\in B}
	 			} 
	 			f(\overline \lmm_A\psi^{b}) r 
		f(\psi^{a+1}\overline \lmm_{B}	).	 
\]
In a similar fashion, we see that in
$\sum_{\substack {I\sqcup J = \underline{n} \\ i\in I, j\in J} }f(\overline \lmm_I)  ((rz^p).f)(\overline \lmm_J)$
the terms of Type B2 add up to the element \[
\sum_{\substack {I\sqcup J = \underline{n}
	 \\ i\in I, j\in J} }
	 	f(\overline \lmm_I)
	 		f(\overline \lmm_J\psi^p)r =
	 		f\sum_{\substack {I\sqcup J = \underline{n}
	 \\ i\in I, j\in J} }
	 	\overline \lmm_I\overline\lmm_J\psi^p r
\]
that is symmetric in $i$ and $j$, and the terms
of Type B3 can be rewritten into the sum
\[
\sum_{\substack 
		{I\sqcup J \subseteq \underline{n}\\ j\in J\supseteq J'}
	 			} 
	 			f(
	 		\overline{\lmm}_{\overline{n}\smallsetminus J'} \psi^{b+1}) r f(\psi^a\overline{\lmm}_{J'}) = 
	 		\sum_{\substack 
		{A\sqcup B \subseteq \underline{n}\\ j\in A}
	 			} 
	 			f(
	 		\overline{\lmm}_A \psi^{b+1}) r f(\psi^a\overline{\lmm}_B) +
	 		\sum_{\substack 
		{A\sqcup B \subseteq \underline{n}\\ j\in B}
	 			} 
	 			f(
	 		\overline{\lmm}_A \psi^{b+1}) r f(\psi^a\overline{\lmm}_B).	 		
\]
In both terms of Type A2 and B1, we can further
divide the sums in four situations: 
\begin{enumerate}
\item The case when $i,j\in A$
\item The case when $i\in A,j\in B$,
\item The case when $j\in A,i\in B$ and,
\item The case when $i,j\in B$. 
\end{enumerate}
We see that the terms (1)
and (2)
where $i$ and $j$ are ``separated by $r$''
cancel out, either in pairs or with the
terms of Type A1 and B2, while the
terms (2) and (3) yields summands that
are symmetric in $i$ and $j$.
This completes the proof of the first assertion. 

The commutator formula is checked by direct 
inspection: when computing 
$\lambda_1.(\lambda_2.f)-\lambda_2.(\lambda_1.f)$, 
there are 
terms where $\lambda_1$ and $\lambda_2$ are 
applied ``in the same place'' and terms where 
they are applied ``at different locations'' (on 
the left and on the right of an element in the 
image of $f$). The latter terms come in pairs 
which cancel each other due to the 
$\kk[\psi]$-bimodule axioms for 
$\lmExtHycomm$ and the 
specific signs in 
Equation~\eqref{eq:lm-Givental-action}. 
\end{proof}

\subsection{The higher order versions of the twisted associative algebra \texorpdfstring{$\lmHycomm$}{lmHycomm}}

We are now ready to introduce the twisted associative algebra which will turn out to represent the homotopy quotient of $\blmBV{k}$ by $\Delta$.

\begin{definition}
Let $k\geqslant 1$. The twisted associative algebra $\blmHycomm{k}$ is defined as the quotient of the extended Losev--Manin algebra $\lmExtHycomm$ by the ideal generated by all $\lmm_{\underline{t}}^{2m}$ where either $t\not\equiv 1\pmod{k-1}$ or $t= 1+p(k-1)$ for some $p\geqslant 0$ but $p\ne m$.
\end{definition}

Note that this definition means that the homological degree $2m$ of each generator $\lmm_{\underline{t}}^{2m}$ of the twisted associative algebra $\blmHycomm{k}$ determines its arity.  In particular, the twisted associative algebra $\blmHycomm{k}$ is always generated by a collection of elements of degrees $0,2,4,\ldots$ whose arities vary depending on~$k$. For example,
\begin{itemize}
	\item $\blmHycomm{1}$ is generated by the elements $\lmm_{\{1\}}^{2m}$ of arity $1$ (for $m=0,1,2,\dots$) subject to the relations  guaranteeing that $\overline \lmm = \sum_{m=0}^\infty \lmm_{\{1\}}^{2m}$ satisfies $[\overline \lmm_{\{1\}}, \overline \lmm_{\{2\}}]=0$,
	\item $\blmHycomm{2}$ is generated by the elements $\lmm_{\underline{1+m}}^{2m}$ (for $m=0,1,2,\dots$), and is isomorphic to $\lmHycomm$.
\end{itemize}
In general, the twisted associative algebra $\blmHycomm{k}$ has the 
following relations, for all values of $d\geqslant 0$, the 
corresponding value $n=2+d(k-1)$, and all $i_1\ne i_2\in
\underline{n}$: 
\begin{equation}
     \sum_{\substack{\underline{n}=I_1\sqcup I_2,\\ i_1\in I_1, i_2\in I_2}}  \lmm^{2d_1}_{I_1}\lmm^{2d_2}_{I_2}=
     \sum_{\substack{\underline{n}=I_1\sqcup I_2,\\ i_1\in I_1, i_2\in I_2}}  \lmm^{2d_2}_{I_2}\lmm^{2d_1}_{I_1}.
\end{equation}
where we further require that $|I_1|=1+d_1(k-1), |I_2|=1+d_2(k-1)$
in both sums. Let us establish that this twisted associative algebra is Koszul, and determine its Koszul dual.

\begin{proposition}\label{prop:bLMHycomm-Koszul}
The suspension of the twisted associative algebra $\blmHycomm{k}^!$ is generated by elements $\lmg_{\underline{1+m(k-1)}}$ of homological degrees $1+2m(k-2)$ for $m\geqslant 0$, each of which generates the trivial representation of the respective symmetric group, subject to the relations 
 \begin{align*}
\sum_{i\in I} \lmg_{\{i\}}\lmg_{\underline{n}\setminus\{i\}} &=\lmg_I\lmg_J \text{ for $\underline{n}=I\sqcup J$ and $|I|>1$ with $|I|\equiv |J|\equiv 1\pmod{(k-1)}$}\\
\sum_{i\in\underline{n}}\lmg_{\{i\}}\lmg_{\underline{n}\setminus\{i\}} &=0 \text{ for $n\geqslant 2$}.
\end{align*}
 This twisted associative algebra is Koszul.   
\end{proposition}

\begin{proof}
The argument is very similar to that of Proposition \ref{prop:tHycommDual}, although it is perhaps easier go in the opposite order. First, we note that these relations annihilate the relations of $\blmHycomm{k}$ under the pairing between the tensor squares of spaces of generators (the suspension ensures that we must compute the usual pairing without sign twists); moreover, in each arity $n$ congruent to $2$ modulo $k-1$ (which are the only arities in which relations appear) all these relations but the last one correspond to subsets of cardinalities bigger than $1$ and congruent to $1$ modulo $k-1$, so they span a subspace whose dimension is equal to the dimension of the annihilator of the space of quadratic relations of $\blmHycomm{k}$ in this arity: it is $(n-1)$-dimensional for the reason identical to that discussed in the proof of Proposition \ref{prop:homology-LM}. 

To establish the Koszul property, we note that, if we ignore the ``wrong'' homological degrees, the relations we consider form a subset of the relations of the twisted associative algebra $\lmHycomm^!$, and so it will be convenient for us to work with that algebra. Let us introduce an ordering of monomials slightly different from the one which one can obtain from the proof of Proposition \ref{prop:homology-LM} by using the compatibility of Gr\"obner bases with the Koszul duality. Specifically, let us say that $\lmg_{I}>\lmg_{J}$ if $|I|>|J|$, and in the case of equality, let us compare the two subsets lexicographically. This means that each for each $\underline{n}=I\sqcup J$ with $|I|>1$ the relation
 \[
\sum_{i\in I} \lmg_{\{i\}}\lmg_{\underline{n}\setminus\{i\}} =\lmg_I\lmg_J 
 \]
has leading term $\lmg_I\lmg_J$, and the relation 
 \[
\sum_{i\in\underline{n}}\lmg_{\{i\}}\lmg_{\underline{n}\setminus\{i\}}=0 
 \]
has leading term $\lmg_{\{n\}}\lmg_{\underline{n-1}}$. Monomials that are normal with respect to these leading terms are precisely the monomials 
 \[
\lmg_{\{i_1\}}\cdot\ldots\cdot\lmg_{\{i_p\}}\lmg_{J},
 \]
where $i_1<\dots<i_p$, $\underline{n}=\{i_1,\ldots,i_p\}\sqcup J$, and $n\in J$. Such monomials are in one-to-one correspondence with $J\setminus\{n\}\subset\underline{n-1}$, and so their number is equal to $2^{n-1}$, the dimension of $\lmHycomm^!(n)$. Since the bound given by the normal monomials is sharp, we obtain a Gr\"obner basis. According to the Diamond Lemma criterion \cite[Th.~4.5.1.4]{MR3642294}, all S-polynomials between the relations have the reduced form zero. However, for our choice of the ordering, a relation of $\lmHycomm^!$ is, ignoring the homological degrees, a relation of $\blmHycomm{k}^!$ if and only if its leading term is an element of $\blmHycomm{k}^!$, and so computing the reduced form of an S-polynomial between two such relations does not produce any other elements. Consequently, the Diamond Lemma implies that the relations of $\blmHycomm{k}^!$ form a Gr\"obner basis for the ordering we are considering, and so this twisted associative algebra is Koszul. 
\end{proof}

\begin{corollary}\label{cor:dim-b-lmgrav}
Suppose that $k>1$. The dimension of $\blmHycomm{k}^!(n)$ is equal to $\sum_{p=0}^n\binom{n-1}{p(k-1)}$.
\end{corollary}

\begin{proof}
We note that the PBW basis of $\blmHycomm{k}^!$ we found consists of all monomials of the form 
 \[
\lmg_{\{i_1\}}\cdot\ldots\cdot\lmg_{\{i_p\}}\lmg_{J},
 \]
where $\underline{n}= I\sqcup J$ is a partition with $|J|\equiv 1\pmod{k}$ and $I =\{i_1,\ldots,i_p\}$, and we further require that $i_1<\cdots<i_p$
and $n\in J$. 
 The number of such monomials is the same as the number of possible sets $J$, which, by removing $n$, are immediately seen to be in one-to-one correspondence with subsets of $\underline{n-1}$ whose cardinality is divisible by $k-1$. 
\end{proof}

\subsection{The case of arbitrary differential order}\label{sec:LM-HigherOrder}

In this section, we shall compute the homotopy quotient of $\blmBV{k}$ by $\Delta$ for any given order $k$. According to \cite[Prop.~2.2]{KMS}, that homotopy quotient can be computed by adjoining the arity zero operations $r_i$
for $i\geqslant 1$, of degree $\deg r_i=2i$, with the differential $\mathrm{d}$ compactly defined by writing 
\begin{equation}
    \mathrm{d} \exp(r(w)) =  \exp(r(w)) (\mathrm{d} + w\Delta).
\end{equation}
Here $w$ is an even formal variable that commutes with all operations $r_i$, as well as with $\mathrm{d}$ and with $\Delta$, and $r(w)= \sum_{i=1}^\infty r_i z^{i(k-1)}$. By abuse of notation, we view $\mathrm{d}$ as an element of arity $0$ and of degree $-1$. The computation consists of two parts. In the first part, inspired by \cite{KMS}, we state the main result using the Givental action on representations of the extended Losev--Manin algebra. In the second part, we prove it using the quadratic-linear Koszul duality theory. 

\subsubsection{Using the Givental action to formulate the main statement }

Let $\mathcal A$ be an arbitrary dg twisted associative algebra, and suppose that $f\colon \lmExtHycomm\to\mathcal A$ is a map that sends all generators except for $\lmm_{\{1\}}^0$ to zero. We shall be using the result of exponentiation of the action of elements of the Lie algebra $\mathfrak{g}_{\mathcal A}$, and to that end we shall introduce a useful language to discuss elements of the form $\exp(r(z)).f$.  

For a given arity $n$, we shall consider ``bamboos'', that is linear graphs of the form

\[
 \begin{tikzpicture}
 \begin{pgfonlayer}{main}
\foreach \x in {0,1,3,4}
  {\node[fill = white] (\x) at (\x,0) {\hphantom{xx}};
   \node [rectangle,draw,fit=(\x),inner sep=0] {};
}
   \node[circle,
  	fill = white, 
  	draw = white, 
  	inner sep = 1 pt,
  	minimum size = 5 pt] at (2,0) {$\cdots$};
  	
\end{pgfonlayer}
	\begin{pgfonlayer}{bg} 
  \foreach \x in {0,1,2,3,4,5}
  \draw[line width = 1.25 pt] (\x-1,0) -- (\x,0);
  \end{pgfonlayer}

 \end{tikzpicture}
\] 
which additionally have tendrils labelled $1,\ldots,n$ distributed between the vertices so that each vertex has at least one tendril (permutations of tendrils attached to the same vertex do not change the object we consider). We always direct this graph from left to right. Each such graph gives rise to an element of $\mathcal A$ as follows. Let us first associate certain data to each half-edge of the bamboo: 
\begin{itemize}
\item the leftmost half-edge (the root of the bamboo graph) is decorated by one of the terms $x_p\psi^p$ of the Taylor series expansion of $\exp(r(\psi))$,
\item the rightmost half-edge (the tip of the bamboo graph) is decorated by one of the terms $y_q\psi^q$ of the Taylor series expansion of $\exp(-r(-\psi))$,
\item each internal edge is decorated by one of the terms $z_{r,s}(\psi')^r(\psi'')^s$ of the Taylor series expansion of 
    \begin{align}
        \frac{\mathrm{Id} - \exp(-r(-\psi'))\exp(r(\psi''))}{\psi'+\psi''}\,,
    \end{align}
   where $\psi'$ and $\psi''$ should be thought of as the psi-classs on the left and the right half-edge, respectively.  
\end{itemize}
To a thus decorated bamboo graph $T$, we can associate an element of $\mathcal A$ denoted by $F(T)$ obtained as the product, in the order of appearance in the graph, of the coefficients $x_p, y_q, z_{r,s} \in \mathcal A(\varnothing)$ and, for each vertex with the attached tendrils labelled $i_1$, \ldots, $i_t$, the elements $f(\lmm^0_{\{i_1\}} \cdots \lmm^0_{\{i_t\}})\in \mathcal A(\{i_1,\ldots,i_t\})$. We also associate to $T$ a combinatorial factor $C(T)$ determined by the powers of $\psi$'s associated to half-edges: for each vertex $v$ such that it has the $a$-th power of the psi-class on the left half-edge and the $b$-th power of the psi-class on the right half-edge, we put $C(v):=\binom{a+b}{a}$, and then define $C(T)$ as the product of $C(v)$ over all vertices.

\begin{lemma}  
The map $\exp(r(z)).f$ on the generators $\lmm_{\underline{n}}^{2p}\in \lmExtHycomm$ is given by $\lmm_{\underline{n}}^{2p} \mapsto \sum_T C(T)F(T)$, where the sum is taken over all decorated bamboos with $n$ tendrils of total degree $2p$. 
\end{lemma}

\begin{proof} It is a direct integration of the Lie algebra action defined in Lemma~\ref{lem:lm:GiventalLie}, see~\cite{ShaZvo-LMCohft}. The only thing that one has to check is the decorations specific for our choice of $f$. In general, each vertex $v$ may be decorated by $f(\psi^{a} \lmm_{\underline{n}}^{2p} \psi^b)$, and in our specific case where $f$ vanishes on all generators but $\lmm^0_{\underline{1}}$, this forces both $a+b=p$ and $p=n-1$, and furthermore  
 \[
f(\psi^a \lmm_{\underline{n}}^{2p} \psi^b) = \binom{a+b}{a}  f(\lmm^0_{\{1\}}) \cdots f(\lmm^0_{\{n\}}).
 \] 
The binomial coefficient here can be obtained by noticing that the constraints $a+b=p$ and $p=n-1$ bring us back to the geometric situation, so we get the intersection number $\int_{\mathcal L\mathcal M(n)} \psi_-^{a} \psi_+^{b}$ which is well known to be equal to the binomial coefficient $\binom{a+b}{a}$. Alternatively, one can recover it algebraically from the formulas for the action from Proposition \ref{prop:psi-bimodule}.  
\end{proof}

In the classical case of the operad $\Hycomm$, it was shown in \cite{MR3019721} that the differential order condition can be encoded using the Givental action. That result has been later proved for the case of $\ncHycomm$ in \cite{DSV-ncHyperCom}. It turns out that a similar result holds in our case as well.  

\begin{lemma}
For any map $f\colon \lmExtHycomm\to \mathcal A$ that sends all generators except for $\lmm_{\{1\}}^0$ to zero, the infinitesimal Givental action of $\mathrm{d}+rz^{k-1}$ annihilates $f$ if and only if $r\in\mathcal A(\varnothing)$ is of differential order at most $k$ with respect to $f(\lmm_{\{1\}}^0)\in\mathcal A(\underline{1})$.
\end{lemma}

\begin{proof}
Similar to \cite[Th.~1]{MR3019721}.  
\end{proof}

This result means that for our purposes it is advantageous to set $w=z^{k-1}$ in the formula
 \[
\mathrm{d} \exp(r(w)) =  \exp(r(w)) (\mathrm{d} + w\Delta).
 \]
 describing the differential in the homotopy quotient, since $\mathrm{d} + z^{k-1}\Delta$ is exactly the element whose Givental action encodes compactly the differential order condition. Motivated by this, let us consider only the Givental symmetries $r(z)$ for which 
  \[
r(z)\in \bigoplus_{t \geqslant 1} \mathcal A(\varnothing)_{2t} z^{(k-1)t}\subset \mathfrak{g}_{\mathcal A}.
 \] 

\begin{lemma}\label{lem:LM-InducedMapOfQuotient}The map $\exp(r(z)).f\colon \lmExtHycomm \to \mathcal A$ factors through the quotient map 
 \[
\pi\colon \lmExtHycomm\to \blmHycomm{k}, 
 \]
that is, there exists a map $g\colon \blmHycomm{k} \to\mathcal A$ such that $\exp(r(z)).f = g\pi$. 
\end{lemma}

\begin{proof} It is sufficient to show that the nontrivial 
decorated bamboos with $n$ tendrils and total degree $2p$ 
exist only for $n=1+p(k-1)$. From the description of the 
$\kk[\psi]$-bimodule structure in 
Proposition~\ref{prop:psi-bimodule}, it immediately follows 
that for each vertex with $k$ tendrils, the sum of degrees 
of psi-classes decorating this vertex must equal to $k-1$, 
and so the sum of degrees of all psi-classes decorating a 
bamboo graph $T$ is equal to $n-|V(T)| = n-|E(T)|-1$, where 
$V(T)$ is the set of vertices and $E(T)$ is the set of 
internal edges of~$T$. On the other hand, the total degree 
of the operators $r(\psi)$ and the sum of the total degree 
of psi-classes and the number of edges are proportional 
as $2$ is to $k-1$. Hence, the total degree is $2p$ if and only 
if $n-1 = p(k-1)$. 
\end{proof}

Let us now consider the map $f\colon \lmExtHycomm \to 
(\blmBV{k}\vee T(r_1,r_2,\dots), \mathrm{d})$ that sends $
\lmm^0_{\underline{1}}$ to the generator $\lmm$ and all 
other generators to $0$. By Lemma~\ref{lem:LM-InducedMapOfQuotient} there exists a map 
 \[
 g\colon \blmHycomm{k}
 	\to (\blmBV{k}\vee T(r_1,r_2,\dots), \mathrm{d})
 \]
such that $g \pi = \exp(r(z)).f$. 

\begin{theorem}\label{thm:LM-quasiiso-map}  
The map $g\colon \blmHycomm{k}\to 	
(\blmBV{k}\vee T(r_1,r_2,\dots), \mathrm{d})$ 
is a quasi-isomorphism. 
\end{theorem}

The proof of this theorem will be gradually constructed in the next sections.

\subsubsection{Quadratic-linear presentation}

Consider the twisted associative algebra $\blmBV{+\infty}$ generated by an element $\Delta$ of arity $0$ and homological degree $1$ and an element $\lmm$ of arity $1$ and of homological degree $0$ subject only to the relations 
 \[
\Delta^2=0, \quad [\lmm_{\{1\}},\lmm_{\{2\}}]=0.
 \]
(We write $+\infty$ to hint that this algebra is obtained as the limit of $\blmBV{k}$ as $k\to+\infty$, and avoid any possible confusion with the notation $\lmBV_\infty$ for a cofibrant replacement of $\lmBV$.) We shall now construct a different presentation of this algebra which will be useful to study the case of $\Delta$ of finite differential order. Let us denote by $\lml_{\underline{n}}$ the $n$-th Koszul-like brace of $\Delta$ with respect to $\lmm$. A gauge symmetry argument \cite{MR3510210,MR3385702} shows that these elements satisfy the relations
\begin{align}
\sum_{\substack {I\sqcup J = \underline{n} } } \lml_{I}   \lml_{J} = 0, \qquad n\geqslant 0.
\end{align}
It is also clear that the quadratic-linear relations
\begin{align}\label{eq:LMBVquadraticlinear}
    [\lml_I,\lmm_{\{i\}}]=\lml_{I\sqcup\{i\}}
\end{align} 
hold in this algebra.

\begin{proposition}\label{prop:LMBV-infty-Koszul}
The twisted associative algebra with the generators $\lmm_{\underline{1}}$ and $\lml_{\underline{n}}$ for
each $n\geqslant 0$, subject to the relations
 \[
 [\lmm_{\{1\}},\lmm_{\{2\}}]=0,	
	\quad
 [\lml_I,\lmm_{\{i\}}]=\lml_{I\sqcup\{i\}}, 
 \quad
 \sum_{\substack {I\sqcup J = \underline{n} } } \lml_{I}   \lml_{J} = 0 
	\text{ for $n\geqslant 0$.} 
	\]
is isomorphic to $\blmBV{+\infty}$. Moreover, this presentation of $\blmBV{+\infty}$ is inhomogeneous Koszul. 
\end{proposition}

\begin{proof}
We shall establish that this quadratic-linear presentation is a Gr\"obner basis of relations, which implies that it is inhomogeneous Koszul. We may order monomials in the free shuffle algebra in such a way that the leading monomials of the relations above are
\[
	\lmm_{\{2\}}\lmm_{\{1\}}, \quad
\lml_\varnothing\lml_{\underline{n}} \text{ for $n\geqslant 0$}
	, 
	\quad \lml_I\lmm_{\{i\}}  \text{ for  $i\notin I$}.	 
	\]
The normal monomials with respect to these leading terms are of the form $\lmm_{\{i_1\}}\lmm_{\{i_2\}}\cdots \lmm_{\{i_s\}} M$ or $\lmm_{\{i_1\}}\lmm_{\{i_2\}}\cdots \lmm_{\{i_s\}} M\lml_\varnothing$, where $M$ is any monomial in the free shuffle algebra generated by $\lml_{\{n\}}$, $n>0$. Let us exhibit a one-to-one correspondence between such monomials and elements forming a basis of $\blmBV{+\infty}$. For that, we simply replace each factor $\lml_{\{i_1,\ldots,i_p\}}$ of $M$ by $\Delta\lmm_{\{i_1\}}\cdots\lmm_{\{i_p\}}$; the fact that the elements thus obtained form a basis is clear. In general, the set of normal monomials with respect to the given relations form a spanning set of the quotient, and we just established that in each arity the cardinality of this (finite) set is equal to the dimension of the quotient, so it must form a basis, and the given relations must form a Gr\"obner basis.  
\end{proof}

If we wish to incorporate the finite differential order condition, it is very easy to do using the presentation we obtained:
 \[
\blmBV{k}\cong \blmBV{+\infty}/(\lml_{\underline{k}}). 
 \]
This leads to a quadratic-linear presentation of $\blmBV{k}$: it is generated by $\lmm_{\{1\}}$
and operations $\lml_{\underline{i}}$ for $i< k$
subject to the following relations:
\[\label{eq:truncLMLieinf}
[\lmm_{\{1\}},\lmm_{\{2\}}]=0,
 \quad
 [\lml_I,\lmm_{\{i\}}]=\lml_{I\sqcup\{i\}}
 \text{ for $|I|<k-1$},
 \quad
[\lml_{I},\lmm_{\{i\}}]=0 \text{ for $|I|=k-1$},
\]
and the family of relations
\[\sum_{{I\sqcup J = \underline{n}}} \lml_{I}   \lml_{J} = 0, \qquad 0\leqslant 
	n\leqslant 2k-2 \text{ and $I,J	
		\leqslant k-1$}. \]
Note that some of the terms that we used as leading terms in 
Proposition \ref{prop:LMBV-infty-Koszul}, for example, 
$\lml_\varnothing\lml_{\underline{n}}$ in the second group of 
relations with $k\leqslant n\leqslant 2k-2$, are no longer 
present (only sets with cardinality at most $k-1$ appear in this sum), and so we cannot use the proof of that proposition to 
conclude that our presentation of $\blmBV{k}$ is inhomogeneous 
Koszul. Nevertheless, that statement turns out to be true. 

\begin{proposition}\label{prop:LMBV-k-Koszul}
The above quadratic-linear presentation of the twisted associative algebra $\blmBV{k}$ is inhomogeneous Koszul for all $k$. 
\end{proposition}

\begin{proof}
Let us focus on $k\geqslant 2$: in the case $k=1$, the presentation is actually homogeneous, and we understood it well in Section \ref{sec:orderone}. Our argument will proceed as follows. We shall first deal separately with the relations~\eqref{eq:truncLMLieinf}
and then show that the commutation relations with the subalgebra generated by $\lmm$, which is a one-generated free twisted commutative algebra, give a twisted associative algebra which is obtained as a sort of Ore extension. 

For \eqref{eq:truncLMLieinf}, it is easier to consider the 
suspension of the Koszul dual algebra; it is generated by 
elements $\lmk_{\underline{i}}$ for $i\geqslant 0$, 
of arity $i$ and homological degree $2i-2$, subject to the 
relations
 \[
\lmk_I\lmk_J=\lmk_{I'}\lmk_{J'} 
	\text{ for $I\sqcup J=I'\sqcup J'$ and $|I|, |J|, |I'|, |J'|<k$.}
 \]
Let us consider the ordering of these generators for which we first compare the arities, and if the arities coincide, we list the elements of the indexing set $I$ of each generator $h_I$ in the increasing order, and compare the lists in the dictionary order. Then the elements that are \emph{not} the leading terms of quadratic relations are
\begin{align*} 
\lmk_\varnothing\lmk_{\underline{n}} & & \text{where $n<k$},\\
\lmk_{\{1,2,\ldots,n-k+1\}}\lmk_{\{n-k+2,\ldots,n\}} & &
\text{ where $k\leqslant n\leqslant 2k-2$}. 
\end{align*}
Moreover, one can easily see that the Koszul dual algebra has a PBW basis of the form $\lmk_{I_1}\lmk_{I_2}\cdots\lmk_{I_p}$ where $I_1+\cdots+I_p$ is a decomposition of $\underline{n}$ as an ordered sum of (possibly empty) intervals for which a non-empty interval cannot be followed by an empty one and
only the first non-empty interval can be of length strictly less than $k-1$. This implies that the algebra presented via relations listed in \eqref{eq:truncLMLieinf} also has a PBW basis, and a quadratic Gr\"obner basis of relations.

Let us now choose the ordering for which the generator $\lmm$ is smaller than any of the generators~$\lmk_{\underline{n}}$. To show that our algebra has a Gr\"obner basis, it is sufficient to show that the commutation relations with $\lmm$ are compatible with the relations listed in \eqref{eq:truncLMLieinf} and  with the relation $[\lmm_{\{1\}},\lmm_{\{2\}}]=0$. For the first assertion, it is enough to show that the commutator 
 \[
\left[\sum_{\substack{I'\sqcup I'' = I\\ |I'|, |I''|<k } } \lml_{I'}   \lml_{I''}, \lmm_{\{i\}}\right]=
\sum_{\substack{I'\sqcup I'' = I\\ |I'|, |I''|<k } } [\lml_{I'}, \lmm_{\{i\}}]   \lml_{I''}+
\sum_{\substack{I'\sqcup I'' = I\\ |I'|, |I''|<k } } \lml_{I'}   [\lml_{I''}, \lmm_{\{i\}}]
 \] 
is a combination of \eqref{eq:truncLMLieinf}. But we have
  \[
\sum_{\substack{I'\sqcup I'' = I\\ |I'|, |I''|<k } } [\lml_{I'}, \lmm_{\{i\}}]   \lml_{I''}+
\sum_{\substack{I'\sqcup I'' = I\\ |I'|, |I''|<k } } \lml_{I'}   [\lml_{I''}, \lmm_{\{i\}}]
=
\sum_{\substack{I'\sqcup I'' = I\\ |I'|<k-1, |I''|<k } } \lml_{I'\sqcup \{i\}}   \lml_{I''}+
\sum_{\substack{I'\sqcup I'' = I\\ |I'|<k, |I''|<k-1 } } \lml_{I'}   \lml_{I''\sqcup\{i\}}
  \]
is precisely one of the relations listed in \eqref{eq:truncLMLieinf}. The second assertion is even simpler:
we compute that
 \[
[[\lmk_I,\lmm_{\{i'\}}],\lmm_{\{i''\}}]=
[\lmk_{I\sqcup\{i'\}},\lmm_{\{i''\}}]= 
\lmk_{I\sqcup\{i'\}\sqcup\{i''\}}, 
 \]
so  $[[\lmk_I,\lmm_{\{i'\}}],\lmm_{\{i''\}}]=
[[\lmk_I,\lmm_{\{i''\}}],\lmm_{\{i'\}}]$,
which is compatible with 
 $
[\lmk_I,[\lmm_{\{i'\}},\lmm_{\{i''\}}]]=0$.
 This altogether implies that the quadratic-linear presentation of $\blmBV{k}$ forms a quadratic Gr\"obner basis, which in turn implies that this presentation is inhomogeneous Koszul.  
\end{proof}

\subsubsection{The Koszul dual dg algebra and its homology}
We can now describe the (suspension of the) Koszul dual twisted associative algebra $(\blmBV{k}^{!},\partial)$.  Its underlying algebra is generated by elements $\lmk_{\underline{i-1}}$ for
$i\in \underline{k}$, of arity $i-1$ and of homological degree $2i-4$, and an element $\lmn$ of arity $1$ and of homological degree $1$, subject to the relations
\begin{align}\label{eq:LM-rel-s-koszul}
[\lmn_{\{1\}},\lmn_{\{2\}}] &= 0,\notag \\ 
\lmk_I\lmk_J &=\lmk_{I'}\lmk_{J'} \text{ whenever $I\sqcup J=I'\sqcup J'$,}\\ %cardinality req. in def. of variables
[\lmk_{I}, \lmn_{\{i\}} ] &= 0 \text{ for each $i\notin I$}.   \notag
\end{align}
The presence of quadratic-linear relation in $\blmBV{k}$ means that the Koszul dual algebra has a nonzero differential $\partial$ which is a derivation given on generators by 
\begin{align}
\partial( \lmk_{\varnothing} ) &= 0, \\
\partial(\lmk_{\underline n})  &=  \sum_{i=1}^n \lmk_{\underline n\setminus \{i\}} \lmn_{\{ i\}} \text{ for $1\leqslant n\leqslant k-1$,} \\
\partial(\lmn_{\{1\}})  &= 0.
\end{align}
For the remaining part of the proof, it will be convenient to introduce auxiliary elements $\mathsf{K}_{\underline n}^{2t}$, which is an element of arity $n$ and degree $2t$ that is equal to the product $\lmk_{I_1}\cdots\lmk_{I_p}$ for any choice of partition $\underline{n}=I_1\sqcup \cdots\sqcup I_{n-t}$ with $|I_s|<k$; notice that $t\leqslant n-1$. Moreover, examining the PBW basis for the second family of relations in~\eqref{eq:LM-rel-s-koszul}
obtained in the proof of Proposition~\ref{prop:LMBV-k-Koszul}, we see that if we write $n=q(k-1)+r$ with $0\leqslant r<k-1$ (so that $q$ and $r$ are, respectively, the quotient and the remainder after the Euclidean division of $n$ by $k-1$), we have $t\leqslant q(k-2)+r-1+\delta_{r,0}$, and the algebra $\blmBV{k}^!$ has a basis 
\begin{gather} \label{eq:bLM-altnotation}
\mathsf{K}_{I}^{2t} \lmn_{\{ j_1\}}\cdots 
\lmn_{\{ j_v\}}, \qquad   n\geqslant 0,\quad 1\leqslant j_1<\cdots < j_v\leqslant n, \quad I=\underline n\setminus\{j_1,\ldots,j_v\},\\ \notag
t\leqslant q(k-2)+r-1+\delta_{r,0}\ , \text{ where }\    |I| = q(k-1)+r, \ 0\leqslant r< k-1.
\end{gather}
To imitate what was done in Section \ref{sec:ordertwo} for $k=2$, we shall allow elements $\mathsf{K}_{\underline{n}}^{2t}$ with the original milder restriction $t\leqslant n$, and consider  a bigger chain complex $(\widehat{\blmBV{+\infty}^!},\partial)$ spanned by the elements
\begin{equation} \label{eq:LM-altnotationFULL}
\mathsf{K}_{I}^{2t} \lmn_{\{ j_1\}}\cdots \lmn_{\{ j_v\}},  \quad  0\leqslant n, \quad 1\leqslant j_1<\cdots < j_v\leqslant n, \quad I=\underline n\setminus\{j_1,\ldots,j_v\},
\quad  t\leqslant |I| ,
\end{equation}
with the differential given by
 \[
\partial\colon \mathsf{K}_{I}^{2t} \lmn_{\{ j_1\}}\cdots \lmn_{\{ j_\ell\}}\mapsto  \sum_{i\in I} \mathsf{K}_{ I\setminus\{i\} }^{2t-2} \lmn_{\{ i\}}\lmn_{\{ j_1\}}\cdots \lmn_{\{ j_\ell\}}.
 \]
Note that this complex is a version of the Koszul complex, and it is manifestly acyclic in each arity $n\geqslant 1$. Importantly for us, $(\blmBV{k}^!,\partial)$ may be viewed as a subcomplex of this complex. 

The homology $H(\blmBV{k}^!,\partial)$ can be computed as follows. In arity $0$, it is spanned by the powers $\lmk_{\varnothing}^p$, $p\geqslant 0$. Since the complex $(\widehat{\blmBV{+\infty}^!},\partial)$ is acyclic in each arity $n\geqslant 1$, the homology of its subcomplex $(\blmBV{k}^!,\partial)$ is exactly the subspace of its elements which are boundaries of elements outside of that subcomplex. It remains to note that elements outside of $(\blmBV{k}^!,\partial)$ whose boundaries belong to this subcomplex are precisely 
 \[
\mathsf{K}_{I}^{2t} \lmn_{\{ j_1\}}\cdots \lmn_{\{ j_\ell\}},\]
where $|I|=q(k-1)+1$, $t=q(k-2)+1$ and $q\geqslant 0$.
Their boundaries in $(\widehat{\blmBV{+\infty}^!},\partial)$ are the elements 
 \[
\sum_{i\in I} \mathsf{K}_{ I\setminus\{i\} }^{2t-2} \lmn_{\{ i\}} \lmn_{\{ j_1\}}\cdots \lmn_{\{ j_\ell\}}\]
with the same parameter constraints.
Moreover, the homology is multiplicatively generated by the elements 
 $
\lmg_{\underline{n}}:=\partial(\mathsf{K}_{\underline{n}}^{2t})$,
again with the same parameter constraints (where now $I=\underline{n}$).
As in Section \ref{sec:ordertwo}, the twisted $A_\infty$-algebra structure on elements of the homology of nonzero arity is in fact a twisted associative algebra: all higher operations applied to several such elements vanish.

It remains to show that the elements $\lmg_{\underline{n}}$ satisfy precisely the relations 
described in Proposition \ref{prop:bLMHycomm-Koszul}. The fact that 
these relations are satisfied is checked analogously to how it is 
done for $k=2$ in Section \ref{sec:ordertwo}, so there is a 
surjective map from $\blmHycomm{k}^!$ onto the homology. To 
conclude, we shall argue as follows. If we consider all twisted 
associative algebras here as shuffle algebras (that is, ignore the 
symmetric group actions), there is a contracting homotopy $h$ of $
(\widehat{\blmBV{+\infty}^!},\partial)$ which annihilates all 
elements not divisible by $\lmn_{\{1\}}$ and sends each element $
\mathsf{K}_{I}^{2t} \lmn_{\{1\}}
	\lmn_{\{ j_1\}}\cdots \lmn_{\{ j_\ell\}}$ to $\mathsf{K}_{I\sqcup\{1\}}^{2t+2} \lmn_{\{ j_1\}}\cdots 
\lmn_{\{ j_\ell\}}$; one immediately sees that $\partial h+h
\partial=\mathrm{id}$. This implies that all boundaries in $
(\widehat{\blmBV{+\infty}^!},\partial)$ can be written as boundaries 
of elements    
 \[
\mathsf{K}_{I}^{2t} \lmn_{\{ j_1\}}\cdots \lmn_{\{ j_\ell\}} 
 \]
satisfying the condition $1\in I$. By direct inspection, such boundaries are linearly independent. In our case, we
have that $|I|=q(k-1)+1$ and $t=q(k-2)+1$ for some $q\geqslant 0$, so the dimension of the space of such boundaries is precisely the number of subsets of $\underline{n-1}$ of cardinality divisible by $k-1$, which, according to Corollary \ref{cor:dim-b-lmgrav}, is equal to the dimension of $\blmHycomm{k}^!(n)$. Thus, the surjection from $\blmHycomm{k}^!$ onto the homology must be an isomorphism. As an immediate by-product, we obtain one of the central results of this section.

\begin{theorem}
The homotopy quotient of $\blmBV{k}$ by $\Delta$ is represented by $\blmHycomm{k}$. 
\end{theorem}

\subsubsection{Proof of Theorem \ref{thm:LM-quasiiso-map}. } Since the dg twisted associative algebra $(\blmBV{k}\vee T(r_1,r_2,\dots), \mathrm{d})$ computes the homotopy quotient by $\Delta$, its homology is isomorphic to $\blmHycomm{k}$. It remains to make the following simple observations. 
\begin{enumerate}
\item The map $g$ is a morphism of dg twisted associative algebras.
\item The images of the generators of $\blmHycomm{k}$ do not vanish in the homology of $\mathrm{d}$ (the image of $\mathrm{d}$ is contained in the ideal generated by~$\Delta$).
\item In the homology of $\mathrm{d}$, the subspace of elements of arity $1+(k-1)t$ and homological degree $2t$ is one-dimensional for each $t\ge0$.
\end{enumerate}
This implies that the map $g$ is surjective on the homology of $\mathrm{d}$, hence it is a quasi-isomorphism.\qed 

\section{Homotopy quotients in the symmetric operad case}
\subsection{Order of differential operator}

Let us recall the definition of the Koszul braces (see \cite{MR1466615,MR3019721,MR837203} for further details on that), and relate them to the notion of a Batalin--Vilkovisky algebra.
%Since it will be convenient, let us introduce the following
%shorthand notation: for $\mathcal O$ a symmetric operad, $m\in\mathcal O$ any operation, and 
%$g \in\mathcal O(n)$ an operation of arity $n$, we define
%\[ 
% \mathsf{Leib}_m(g) = g\circ_n m -
%  	m\circ_1 g - (m\circ_1 g)\tau_n \]
% where $\tau_n$ is the transposition $(n,n+1)\in S_{n+1}$.
% Note we can consider the same definition in case $m$ 
% has other homological degrees, by introducing the appropriate
% signs.
% 
\begin{definition}
Let $\mathcal O$ be a symmetric operad, and let $m\in\mathcal O(\underline{2})_0$ be a commutative associative binary operation satisfying. For an element $f\in\mathcal O(\underline{1})$, 
we define the $n$-th \emph{Koszul brace} $b_n^f\in\mathcal  O(\underline{n})$ by the formulas
\[ b_1^f =f,\qquad
b_{n+1}^f = b_n^f\circ_n m - m\circ_1 b_n^f - (m\circ_1 b_n^f).(n,n+1)\text{ for $n\geqslant 1$}.
\]
We say that the element $f\in\mathcal O(\underline{1})$ is of differential order at most $k$ (with respect to $m$) if $b_{k+1}^f=0$. 
\end{definition}
%merged definitions

The operad $\BV$ of Batalin--Vilkovisky algebras has two equivalent definitions. One of them says that it is generated by an element $\Delta$ or arity $1$ and homological degree $1$ and a commutative binary operation $m$ of arity $2$ and homological degree $0$ subject to the relations
\[ \Delta^2 =0, \quad
m\circ_1 m =m\circ_2m,\quad\
b_{3}^\Delta =0.
\]
The other one says that the same operad admits a quadratic-linear presentation using two commutative binary operations $m$ and $\ell$ of respective degrees $0$ and $1$, and a unary generator $\Delta$ of degree~$1$, satisfying the relations
\[
\Delta^2 = 0 , \quad 
m\circ_1m =m\circ_2m, \quad 
(\ell\circ_1\ell).(1+(123)+(132))=0 
\]
saying that $\Delta$ is of square zero, $m$ is associative
and $\ell$ is a Lie bracket,
along with the relations
\begin{gather*}
	\Delta\circ_1 m - m\circ_1 \Delta
	- m\circ_2 \Delta =\ell, \\
\ell\circ_2 m - m\circ_1 \ell - 
	m\circ_1 \ell .(23)  = 0, \\
\Delta \circ_1 \ell + 
\ell\circ_1\Delta + \ell\circ_2\Delta = 0 
\end{gather*}
expressing that $\Delta$ is of order two for $m$ and
that $\ell$ is precisely the obstruction for it to be
of order one. The last relation, saying that $\Delta$ is
a derivation for $\ell$, in fact follows from the other
relations. Examining the first of those definitions, it is very easy to generalise it to use other values of the differential order. 

\begin{definition}
The symmetric operad $\bBV{k}$ is generated by an element $\Delta$ or arity $0$ and homological degree $1$ and a commutative binary operation $m$ of arity $2$ and homological degree $0$ subject to the relations
\[ 
\Delta^2=0,\quad
m\circ_1m=m\circ_2m,\quad
b_{k+1}^\Delta=0.
\]
\end{definition}

\subsection{The case of differential order one}\label{sec:DMorderone}
 The symmetric operad $\bBV{1}$ is generated by an element $\Delta$ of arity $0$ and homological degree $1$ and an element $m$ of arity $2$ and homological degree $0$ subject to the relations
\[
\Delta^2=0, \quad
 m\circ_1m=m\circ_2m, 
 \quad \Delta\circ_1m-m\circ_1\Delta-m\circ_2\Delta=0.
\]
It is a (homogeneous) quadratic symmetric operad, and one can immediately see that it is Koszul; in fact, it has a quadratic Gr\"obner basis for which the corresponding PBW basis consists of elements 
 \[
m^{[n]} \text{ and } (\cdots((m^{[n]}\circ_{i_1}\Delta)
	\circ_{i_2}\Delta)\cdots)\circ_{i_p}\Delta
\text{ for $n\ge 0$ and $I=\{i_1,\ldots,i_p\}
	\subset\underline{n+1}$.}
 \]
where we denote by $m^{[n]}$ the $n$-fold composite of the operation 
$m$: $m^{[0]}=\mathrm{id}$, $m^{[n+1]}=m^{[n]}\circ_1 m$.
Briefly, we consider all left combs $m^{[n]}$ with a possibly
empty subset of leaves decorated by $\Delta$.

The Koszul dual twisted associative coalgebra $\blmBV{1}^{\ac}$ has a PBW basis consisting of the elements 
 \[
u_{n,r}:=(s\Delta)^r (sm)^{[n]} \text{ for $r, n\geqslant 0$},
 \]
with the obvious cooperad structure (splitting $r$ into a sum and splitting the set of arguments of $(sm)^{[n]}$ into a disjoint union). The homotopy quotient of $\bBV{1}$ by $\Delta$ is represented by the dg symmetric operad which is the quotient of the cobar construction $\Omega(\bBV{1}^{\ac})$ by the two-sided ideal generated by the elements $s^{-1}u_{0,r}$ for $r>0$. Thus quotient creates from the cobar construction another free symmetric operad equipped with a quadratic differential: it is generated by the elements $s^{-1}u_{\underline{n},r}$ with $n>0$, and the differential as above but with all terms involving $s^{-1}u_{0,r}$ being suppressed. This differential is the differential of the cobar construction of the symmetric cooperad which is the quotient of $\bBV{1}^{\ac}$ by the \emph{subspace} spanned by the elements $u_{0,r}$ for $r>0$. 
The linear dual symmetric operad is generated by the elements $u_{1,p}$ for $p\geqslant 0$, of arity $2$ and homological degree $2p+1$ subject to the relations
\begin{align*}
u_{1,p}\circ_1u_{1,q}+u_{1,p}\circ_2 u_{1,q}&=0,\\
u_{1,p}\circ_1u_{1,q}-u_{1,p-1}\circ_1 u_{1,q+1}&=0.
\end{align*}
It is easy to see that this symmetric operad is Koszul; in 
fact, it has a quadratic Gr\"obner basis for which the
map $u_{n,r}\longmapsto (u_{1,0})^{[n-1]}\circ_1 u_{1,r}$
for $r\geqslant 0$ and $n>0$ gives a bijection between
the corresponding linear and the target PBW basis.
Therefore, the homology 
of its cobar complex is the Koszul dual symmetric operad which 
has generators $m^{2p}$ of arity $2$ and homological degree 
$2p$ for $p\geqslant 0$ with the relations stating that the 
formal sum $\overline{m}=\sum_{p=0}^\infty m^{2p}$ satisfies
 \[ 
	\overline{m}\circ_1\overline{m}=
		\overline{m}\circ_2\overline{m}. 
 \]
To make sense of these relations, one has to separate the terms 
by total homological degree $2d\geqslant 0$ into infinitely 
many relations 
 \[
\sum_{p+q=d}m^{2p}\circ_1m^{2q}=\sum_{p+q=d}m^{2p}\circ_2m^{2q}. 
 \]
involving finitely many terms each. We established the following result.

\begin{theorem}
The homotopy quotient of $\bBV{1}$ by $\Delta$ is represented by the symmetric operad with generators $m^{2p}$ of arity $2$ and homological degree $2p$ for $p=0,1,2,\dots$ subject to the relations 
 \[
\sum_{p+q=d}m^{2p}\circ_1m^{2q}=\sum_{p+q=d}m^{2p}\circ_2m^{2q}	
\text{ for  $d\geqslant 0$}. 
 \]
\end{theorem}

The computation of the homotopy quotient of $\bBV{2}$ (that is, the operad $\BV$) by $\Delta$ was first done in full by Drummond-Cole and Vallette \cite{MR3029946}. Our approach for the case of finite order at least two is completely uniform and closer to \cite{KMS}, and we shall proceed with developing a general formalism instead of focusing on the particular case $k=2$.

\subsection{Extended hypercommutative operad and Givental symmetries}

We shall now define a symmetric operad generalizing the hypercommutative operad of Getzler \cite{MR1363058} which is one of the protagonists of the calculation of homotopy quotients by operators of higher order.

\begin{definition}
The \emph{extended hypercommutative operad}, denoted $\ExtHycomm$, is the symmetric operad generated by fully symmetric operations $m_n^{2p}$ of all possible arities $n\geq 2$ and all possible non-negative even homological degrees $2p\geq 0$, subject to the relations that can be described as follows. Consider the sums of all generators of a fixed arity $\overline m_n\coloneqq \sum_{p=0}^\infty m_n^{2p}$, and impose, for each arity $n\ge 3$ and each choice of three distinct elements $i,j,j'\in\underline{n}$, the condition
\begin{equation}\label{eq:HCrel}
\sum_{\substack{I\sqcup J =\underline{n} \\ i\in I;\, j,j'\in J}} \overline m_{I\sqcup\{\star\}}\circ_\star \overline m_{J}=\sum_{\substack{I\sqcup J = \underline{n} \\ j\in I;\, i,j'\in J}} \overline m_{I\sqcup\{\star\}}\circ_\star \overline m_{J} .
\end{equation}
Note that if each operation $\overline m_n$ were of homological degree $2(n-2)$, these identities would precisely define the operad $\Hycomm$ of hypercommutative algebras discussed in Section \ref{sec:DMrecall}. In our case, each of these operations is an infinite sum of operations of different homological degrees, and so one should separate these identities by homological degrees into infinitely many relations involving finitely many terms each. 
\end{definition}

In the case of the spaces $\overline{\mathcal M}_{0,1+n}$, it is well known that the so called psi-classes play important role in algebraic and geometric results involving those spaces. Let us define a formal algebraic version of those classes in the case of $\ExtHycomm$, where there is no longer a geometric interpretation of those as cohomology classes. We start with a preparatory lemma.

\begin{lemma}\label{DM:action-indep}
The element \[
\sum_{\substack 
	{I\sqcup J_1= \underline{n} \\ 
		j_1,j_2\in J_1	} 
				} 
				\overline m_{I\sqcup\{\star\}}	
					\circ_\star 
				\overline m_{J_1},
 \]
does not depend on the choice of the sequence of distinct elements $j_1,\ldots,j_2\in\underline{n}$. 
Similarly, for each $i\in\underline{n}$, the element
 \[
\sum_{
	\substack{
		I\sqcup J_1 = \underline{n} 
			\\  j_1\in J_1,i\in I
				}
					}
	\overline  m_{J_1\sqcup\{\star\}}
		\circ_\star 
		\overline m_{I} 
 \]
does not depend on the choice of 
$j_1\in\underline{n}$. 
\end{lemma}

\begin{proof}
Let us first show that the element 
 \[
\sum_{\substack {I\sqcup J_1= \underline{n} \\ j_1,j_2\in J_1} } \overline m_{I\sqcup\{\star\}}\circ_\star\overline m_{J_1} 
 \]
does not change if we replace $j_2$ by another element $j_3$. To show that 
 \[
\sum_{\substack {I\sqcup J_1 = \underline{n} \\ j_1,j_2\in J_1}} \overline m_{I\sqcup\{\star\}}\circ_\star \overline m_{J_1}=\sum_{\substack {I\sqcup J_1 = \underline{n} \\ j_1,j_3\in J_1}} \overline m_{I\sqcup\{\star\}}\circ_\star \overline m_{J_1} ,
 \]
we note that this equation is obtained by adding the same element
 \[
\sum_{\substack {I\sqcup J_1 = \underline{n} \\ j_1,j_2,j_3\in J_1}} \overline m_{I\sqcup\{\star\}}\circ_\star \overline m_{J_1}
 \]
to both sides of the defining relation \eqref{eq:HCrel} for the elements $j_1,j_2,j_3$. Moving from a pair $\{j_1,j_2\}$ to any other pair can be accomplished by two exchanges like that, so the first assertion is proved.

Let us prove the second assertion. For that, we shall show that the element
 \[
\sum_{\substack {I\sqcup J_1 = \underline{n} \\ j_1\in J_1, i\in I}}  \overline m_{J_1\sqcup\{\star\}}\circ_\star \overline m_{I} 
 \]
does not change if we replace $j_1$ by another element $j_2$.  To show that
 \[  
\sum_{\substack {I\sqcup J_1= 
	\underline{n} \\ j_1\in J_1, i\in I}}  
	\overline m_{J_1\sqcup\{\star\}}\circ_\star
	\overline m_{I}=\sum_{\substack 
		{I\sqcup J_1= \underline{n} \\ j_2\in J_1, i\in I}}  
		\overline m_{J_1\sqcup\{\star\}}\circ_\star
			\overline m_{I} ,
 \]
we note that this equation is obtained by adding the same element
 \[
\sum_{\substack {I\sqcup J_1= 
	\underline{n} \\ j_1,j_2\in J_1, i\in I}}  
		\overline m_{J_1\sqcup\{\star\}}
			\circ_\star\overline m_{I}
 										\]
to both sides of the defining relation \eqref{eq:HCrel} for the elements $j_1,j_2,i$. This completes the proof of
the lemma.
\end{proof}

\begin{proposition}\label{prop:DM-psi-bimodule}
If, for a formal variable $\psi$ of homological degree $-2$, we let 
\begin{align}
\psi\circ_1 \overline m_{\underline{n}}&:= 
\sum_{\substack {I\sqcup J_1 = 
	\underline{n} \\ 
		 j_1,j_2\in J_1} }
		\overline m_{I\sqcup\{\star\}}
			\circ_\star 
				\overline m_{J_1}, 
\label{eq:DMleftact}  \\ 
\overline m_{\underline{n}}\circ_i \psi &:=
\sum_{\substack {I\sqcup J= \underline{n} \\ j_1\in J,i\in I}}
	 	 \overline m_{J\sqcup\{\star\}}
	 	 	\circ_\star  \overline m_{I},	 			\label{eq:DMrightact} 
\end{align}
(for some pair of distinct elements $j_1,j_2\in\underline{n}$ in the first case, and for some $j_1\in\underline{n}$ in the second case) these formulas make $\ExtHycomm$ a symmetric operad over the ring $\kk[\psi]$. 
\end{proposition}

\begin{proof}
From Lemma \ref{DM:action-indep}, it follows that the right 
hand sides of our formulas do not depend on the choices one can 
make. We have to verify that Formulas \eqref{eq:DMleftact} 
and \eqref{eq:DMrightact} are compatible with the relations of 
$\ExtHycomm$, that is the action on each relation produces an 
element that vanishes as a consequence of some relations, and 
also to check that the two actions define a structure of 
symmetric operad over $\kk[\psi]$, for which it is enough to 
check that
\begin{align*}
(\psi\circ_1\overline m_n)\circ_i \psi
	&=\psi\circ_1(\overline m_n\circ_i\psi) 
		\text{ for all $i\in \underline{n}$,}\\
(\overline m_n\circ_{i}\psi)\circ_{j}\psi
	&=(\overline m_n\circ_{j}\psi)\circ_{i}\psi 
		\text{ for all $i\ne j\in \underline{n}$.}
\end{align*}
Both claims are checked analogously to those checked in the proof of Proposition \ref{prop:psi-bimodule}.
\end{proof}

Let us show how the psi-classes can be used to define symmetries of representations of $\ExtHycomm$ that generalise the Givental action \cite{ShaZvo-LMCohft}. Precisely, let $\mathcal O$ be an arbitrary dg symmetric operad, and consider the set of dg symmetric operad morphisms $\Hom(\ExtHycomm,\mathcal O)$. For a formal variable $z$ of degree~$-2$, let us consider the complete Lie algebra $\mathfrak{g}_{\mathcal O}$ which is the kernel of the ``augmentation''
 \[
\bigoplus_{p=0}^\infty 
	\mathcal O(\underline{1})_{2p}\otimes
		\kk\llbracket z \rrbracket
			\twoheadrightarrow 
				\mathcal O(\underline{1})_0
 \]
annihilating elements of $\mathcal O$ of positive homological degrees as well as all positive powers of $z$. 
We shall now define an action of $\mathfrak{g}_{\mathcal O}$ by infinitesimal symmetries of $\Hom(\ExtHycomm,\mathcal O)$. For that, we put
\begin{align}
	((rz^p).f) (\overline m_{n}) &=  r \circ_1 f(\psi^{p}\overline m_{n}) + (-1)^{p-1} \sum_{i=1}^n f(\overline m_{n}\circ_i\psi^p)\circ_i r 
	\\ \notag
	&+ \sum_{i+j=p-1} (-1)^{j+1} \sum_{I \sqcup J = \underline{n}} f(\overline m_{J\sqcup\{\star\}}\circ_\star\psi^{j} )\circ_{\star} \big(r\circ_1 f(\psi^i \overline m_I)\big),
\end{align}
and extend it linearly to $\mathfrak{g}_{\mathcal O}$. 
%Note that in the last sum we always have $|I|\geqslant 2$ and 
%$|J|\geqslant 1$ by our definition of $\overline{m}$.
It turns out that this way we obtain infinitesimal symmetries 
of $\Hom(\ExtHycomm,\mathcal O)$. By this, we mean to consider
the `square-zero extension' $\ExtHycomm_\varepsilon :=\ExtHycomm[\varepsilon]/(\varepsilon^2)$ over $\kk_\varepsilon :=
\kk[\varepsilon]/(\varepsilon^2)$, and obtain morphisms of 
$\kk_\varepsilon$-linear operads $\ExtHycomm_\varepsilon
\longrightarrow \mathcal O_\varepsilon$, as the following
lemma states.

\begin{lemma}\label{lem:comm:GiventalLie}
For any $f\in \Hom(\ExtHycomm,\mathcal O)$ and any
$r\in \mathfrak g_\mathcal {O}$, we have that 
 \[
f+\varepsilon r.f\in \Hom_{\kk_\varepsilon}(\ExtHycomm_\varepsilon,\mathcal O_\varepsilon) ,
 \]
and for all $\lambda_1,\lambda_2\in \mathfrak{g}_{\mathcal O}$, we have that
$
[\lambda_1,\lambda_2].f=\lambda_1.(\lambda_2.f)-\lambda_2.(\lambda_1.f).$
\end{lemma}

\begin{proof} 
We need to check that for each $p\geqslant 1$, and each $f\in \Hom(\ExtHycomm,\mathcal O)$, the elements 
\[
	f(\overline m_{\underline{n}})+
		\varepsilon ((rz^p).f)(\overline m_{\underline{n}})
			\in \mathcal O_\varepsilon \]
satisfy the defining relations of $\ExtHycomm$. For the 
constant coefficient, this is exactly the same property of $f$. For the coefficient of $\varepsilon$, we need to check that 
 \[
\sum_{\substack{I\sqcup J =\underline{n} \\ i\in I;\, j,j'\in J}} 
((rz^p).f)(\overline m_{I\sqcup\{\star\}})\circ_\star f(\overline m_{J})+f(\overline m_{I\sqcup\{\star\}})\circ_\star ((rz^p).f)(\overline m_{J})
 \] 
is symmetric in $i,j\in\underline{n}$. This is done in the same way as it is done for twisted associative algebras in Lemma \ref{lem:lm:GiventalLie}, using the explicit
formulas for the action of $\kk[\psi]$. The argument establishing the commutator formula is also analogous to the twisted associative algebra case.
\end{proof}

\subsection{The higher order versions of the symmetric operad \texorpdfstring{$\Hycomm$}{Hycomm}}

We are now ready to introduce the symmetric operad which will turn out to represent the homotopy quotient of $\bBV{k}$ by~$\Delta$.

\begin{definition}
Let $k\geqslant 1$. The symmetric operad $\bHycomm{k}$ is defined as the quotient of the extended hypercommutative operad $\ExtHycomm$ by the ideal generated by all $m_{\underline{t}}^{2p}$ where either $t\not\equiv 2\pmod{k-1}$ or $t= 2+v(k-1)$ for some $v\ge 0$ but $v\ne p$.
\end{definition}

Note that this definition means that the homological degree $2p$ of each generator $m_{\underline{t}}^{2p}$ of the symmetric operad $\bHycomm{k}$ determines its arity.  In particular, the symmetric operad $\bHycomm{k}$ is always generated by a collection of elements of non-negative even degrees 
whose arities vary depending on~$k$. For example,
\begin{itemize}
    \item $\bHycomm{1}$ is generated by symmetric binary 
    operations $m_{\underline{2}}^{2p}$ of arity $2$ for
      $p\geqslant 0$ subject to the relations guaranteeing that 
      $\overline m = \sum_{p=0}^\infty m_{\underline{2}}^{2p}$ 
      is associative.
    \item $\bHycomm{2}$ is generated by symmetric operations 
    $m_{\underline{2+p}}^{2p}$ for $p\geqslant 0$, and is 
    isomorphic to $\Hycomm$.
\end{itemize}
In general, the operad $\bHycomm{k}$ has the following relations, for all values of $d\geqslant 0$, the corresponding value $n=3+d(k-1)$, and all triples of distinct elements $i,j,j'\in\underline{n}$: 
\begin{equation}
\sum_{\substack{I\sqcup J =\underline{n} \\ i\in I;\, j,j'\in J}}
	m_{I\sqcup\{\star\}}^{2d_1}\circ_\star m_{J}^{2d_2}=
\sum_{\substack{I\sqcup J =\underline{n} \\ j\in I;\, i,j'\in J}} 
	m_{I\sqcup\{\star\}}^{2d_1}\circ_\star m_{J}^{2d_2},
\end{equation}
where we further require that  $|I|=1+d_1(k-1)$ and $|J|=2+d_2(k-1)$
in both sums. 
Let us establish that this operad is Koszul, and determine its Koszul dual.

\begin{proposition}\label{prop:bHycomm-Koszul}
The suspension of the symmetric operad $\bHycomm{k}^!$ is generated by elements $g_{\underline{2+m(k-1)}}$ of homological degrees $1+2m(k-2)$ for $m\geqslant 0$, each of which generates the trivial representation of the respective symmetric group, subject to the relations 
\begin{align*}
\sum_{\{i,j\}\subset J} 
	g_{\underline{n}\setminus\{i,j\}
		\sqcup\{\star\}}\circ_\star 
			g_{\{i,j\}} &= g_{I\sqcup\{\star\}}
				\circ_\star g_{J} \text{ for $\underline{n}=I\sqcup J$ with $|J|>2$}\\
\sum_{\{i,j\}\subset\underline{n}} 
	g_{\underline{n}\setminus\{i,j\}
		\sqcup\{\star\}}\circ_\star 
			g_{\{i,j\}} &=0  \text{ for $n\geqslant 3$}.
\end{align*}
where in the first family of relations we require
that $|I|\equiv |J|\equiv 2\pmod{k-1}$.
This operad is Koszul.   
\end{proposition}

\begin{proof}
The argument is very similar to that of Proposition \ref{prop:DMHycommDual}. First, we note that these relations annihilate the relations of $\bHycomm{k}$ under the pairing between the weight two components of the respective free operads (the suspension ensures that we must compute the usual pairing without sign twists); moreover, in each arity $n$ congruent to $3$ modulo $k-1$ (which are the only arities in which relations appear) all these relations but the last one correspond to subsets of cardinalities bigger than $2$ and congruent to $2$ modulo $k-1$, so they span a subspace whose dimension is equal to the dimension of the annihilator of the space of quadratic relations of $\bHycomm{k}$ in this arity. To establish the Koszul property, we note that, if we ignore the ``wrong'' homological degrees, the relations we consider form a subset of the relations of the symmetric operad $\Hycomm^!$, and so it will be convenient for us to use the known argument for that operad discussed in the proof of Proposition \ref{prop:DMHycommDual}. For the ordering from that proof, a relation of $\Hycomm^!$ is, ignoring the homological degrees, a relation of $\bHycomm{k}^!$ if and only if its leading term is an element of $\bHycomm{k}^!$, and so computing the reduced form of an S-polynomial between two such relations does not produce any other elements. Consequently, the Diamond Lemma implies that the relations of $\bHycomm{k}^!$ form a Gr\"obner basis for the ordering we are considering, and so this operad is Koszul. 
\end{proof}

The following result is a generalisation of \cite[Lemma 5.17]{MR3084563}. 

\begin{corollary}\label{cor:dim-b-grav}
Suppose that $k>1$. The dimension of $\bHycomm{k}^!(n)$ is equal to the sum of coefficients at powers of $t$ divisible by $k-1$ of the polynomial
 \[
 (2+t)(3+t)\cdots(n-1+t). 
 \]
\end{corollary}

\begin{proof}
The normal monomials with respect to the leading terms of the Gr\"obner basis from the proof of Proposition \ref{prop:bHycomm-Koszul} are precisely
 \[
g_p(u_1,u_2,\ldots,u_{p-1},\mathrm{id}_{\{j\}}),  
 \]
where all $u_i$ for $i\in \underline{p-1}$ are ``left combs'' obtained by iterations of $g_2$ (these left combs form a basis of the shifted Lie suboperad).

Let us compute the number of monomials like that. To do this, we note 
that we may equivalently enumerate basis elements
 \[
g_{p-1}(u_1,u_2,\ldots,u_{p-1}),  
 \]
where the arity of $u_{p-1}$ is at least~$2$ (this follows from a simple bijection joining $u_{p-1}$ and $\mathrm{id}_{\{j\}}$ into $u_2(u_{p-1},\mathrm{id}_{\{j\}})$). The  number of elements of that kind is equal to
 \[
\sum_{\substack{a_1+\cdots+a_{p-1}=n, \\ a_i\geqslant 1, a_{p-1}\ge2}}\frac{(a_1-1)!(a_2-1)!\cdots(a_{p-1}-1)!a_1a_2\cdots a_{p-1}}{(a_1+a_2+\cdots+a_{p-1})(a_2+\cdots+a_{p-1})\cdots a_{p-1}}\binom{a_1+\cdots+a_{p-1}}{a_1,a_2,\ldots,a_{p-1}}  
 \]
where each factor $(a_i-1)!$ counts the number of left combs of arity~$a_i$, and the remaining factor is known~\cite{MR3203368} to be equal to the number of shuffle permutations of the type $(a_1,\ldots,a_{p-1})$. This can be rewritten in the form 
 \[
\sum_{\substack{a_1+\cdots+a_{p-1}=n, \\ a_i\geqslant 1, a_{p-1}\geqslant 2}}\frac{(a_1+\cdots+a_{p-1}-1)!}{(a_2+\cdots+a_{p-1})(a_3+\cdots+a_{p-1})\cdots a_{p-1}}
 \]
and if we introduce new variables $a_i'=a_i+\cdots+a_{p-1}$, it takes the form
 \[
\sum_{2\leqslant a_{p-1}'<\cdots<a_1'\leqslant n-1}\frac{(n-1)!}{a_2'\cdots a_{p-1}'},
 \]
which clearly is the coefficient of $t^{p-2}$ in the product
 \[
(n-1)!\left(1+\frac{t}{2}\right)\left(1+\frac{t}{3}\right)\cdots\left(1+\frac{t}{(n-1)}\right)=(2+t)(3+t)\cdots(n-1+t), 
 \]
so the dimension of $\bHycomm{k}^!(n)$ is equal to the sum of coefficients of $t^{p-2}$ with $p\equiv 2\pmod{k-1}$ of this polynomial, that is the sum of coefficients at powers of $t$ divisible by $k-1$.
\end{proof}

\subsection{The case of arbitrary differential order}

In this section, we shall compute the homotopy quotient of 
$\bBV{k}$ by $\Delta$ for any given order $k$. As in 
Section \ref{sec:LM-HigherOrder}, that homotopy quotient can 
be computed by adjoining the arity zero operations $r_i$ for 
$i\geqslant 1$, of degree $\deg r_i=2i$, with the differential 
$\mathrm{d}$ compactly defined by writing 
\begin{equation}
    \mathrm{d} \exp(r(w)) =  \exp(r(w)) (\mathrm{d} + w\Delta).
\end{equation}
The computation consists of two parts. In the first part, inspired by \cite{KMS}, we state the main result using the Givental action on representations of the extended hypercommutative operad. In the second part, we prove it using the quadratic-linear Koszul duality theory. 

\subsubsection{Using the Givental action to formulate the main statement}

Let $\mathcal O$ be an arbitrary dg symmetric operad, and suppose that $f\colon \ExtHycomm\to\mathcal O$ is a map that sends all generators except for $m_{\underline{2}}^0$ to zero. We shall use the result of exponentiation of the action of elements of the Lie algebra $\mathfrak{g}_{\mathcal O}$, and to that end we shall introduce a useful language to discuss elements of the form $\exp(r(z)).f$.  

For a given arity $n$, we shall consider non-planar rooted trees with $n$ leaves labelled by $\underline{n}$; such a tree has $n$ half-edges corresponding to leaves and another half-edge corresponding to the root. Each such tree gives rise to an element of $\mathcal O$ as follows. Let us first associate certain data to each half-edge of the tree: 
\begin{itemize}
\item the half-edge corresponding to the root is decorated by one of the terms $x_p\psi^p$ of the Taylor series expansion of $\exp(r(\psi))$,
\item each half-edge corresponding to a leaf is decorated by one of the terms $y_q\psi^q$ of the Taylor series expansion of $\exp(-r(-\psi))$,
\item each internal edge is decorated by one of the terms $z_{r,s}(\psi')^r(\psi'')^s$ of the Taylor series expansion of 
    \begin{align}
        \frac{\mathrm{Id} - \exp(-r(-\psi'))\exp(r(\psi''))}{\psi'+\psi''}\,,
    \end{align}
   where $\psi'$ and $\psi''$ should be thought of as the psi-classs on the half-edge closer to the root and further from the root, respectively.  
\end{itemize}
To a thus decorated tree $T$, we can associate an element of $\mathcal O$ denoted by $F(T)$ obtained as the composition according to the tree $T$, of the coefficients $x_p, y_q, z_{r,s} \in \mathcal O(\underline{1})$ and, for each vertex with the half-edges $e_1$, \ldots, $e_t$, the elements $f((m^0_{\underline{2}})^{[t]})\in \mathcal O(\{i_1,\ldots,i_t\})$. We also associate to $T$ a combinatorial factor $C(T)$ determined by the powers of $\psi$'s associated to half-edges: for each vertex $v$ such that it has the $k_0$-th power of the psi-class on the half-edge corresponding to the root and the power $k_1$, \ldots, $k_t$ of the psi-class on the half-edges corresponding to the leaves, we put $C(v):=\binom{t-2}{k_0,k_1,\dots,k_t}$, and then define $C(T)$ as the product of $C(v)$ over all vertices.

\begin{lemma}\label{lem:comm:treeformula}  The map $\exp(r(z)).f$ on the generators $m_{n}^{2p}\in \ExtHycomm$ is given by $m_{n}^{2p} \mapsto \sum_T C(T)F(T)$, where the sum is taken over all decorated rooted trees with $n$ labeled leaves of  total degree $2p$. 
\end{lemma}

\begin{proof} It is a direct integration of the Lie algebra 
action defined in Lemma~\ref{lem:comm:GiventalLie}. See for 
example \cite{DSV-circle,DSS,KMS,MR2568443}. The only thing 
that one has to check is the decorations specific for our 
choice of $f$. In general, each vertex $v$ may be decorated by 
$f(\psi^{k_0}m^{2p}_{t}(\psi^{k_1}\otimes \psi^{k_2}\otimes
\cdots\otimes \psi^{k_t}))$, and in our specific case where $f$ 
vanishes on all generators but $m_{\underline{2}}^0$, this 
forces $\sum_{i=0}^t k_i=p$ and $p=t-2$, and furthermore
 \[
f(\psi^{k_0}m^{2p}_{t}(\psi^{k_1}, \psi^{k_2},\ldots, \psi^{k_t}))=\binom{t-2}{k_0,k_1,\dots,k_t}f((m^0_{\underline{2}})^{[t]})).
 \]
The multinomial coefficient here can be obtained by noticing that the constraints $\sum_{i=0}^t k_i=p$ and $p=t-2$ bring us back to the geometric situation, so we get the intersection number 
 \[
\int_{\overline{\mathcal M}_{0,1+t}} \psi_0^{k_0}\psi_1^{k_1}\cdots\psi_t^{k_t} ,
 \] 
which is well known to be equal to the multinomial coefficient $\binom{t-2}{k_0,k_1,\dots,k_t}$. Alternatively, one can recover it algebraically from the formulas for the action from Proposition \ref{prop:DM-psi-bimodule}.
\end{proof}

We now recall a result of the first two authors and Vallette allowing to express the differential order condition via the Givental action.

\begin{lemma}[{Theorem 1 in \cite{MR3019721}}]
For any map $f\colon \ExtHycomm\to \mathcal O$ that sends all generators except for $m^0_{\underline{2}}$ to zero, the infinitesimal Givental action of $\mathrm{d}+rz^{k-1}$ annihilates $f$ if and only if $r\in\mathcal O(\underline{1})$ is of differential order at most $k$ with respect to $f(m^0_{\underline{2}})\in\mathcal O(\underline{2})$.
\end{lemma}

Once again, this result means that for our purposes it is advantageous to set $w=z^{k-1}$ in the formula
 \[
\mathrm{d} \exp(r(w)) =  \exp(r(w)) (\mathrm{d} + w\Delta).
 \]
 describing the differential in the homotopy quotient, since $\mathrm{d} + z^{k-1}\Delta$ is exactly the element whose Givental action encodes compactly the differential order condition. Motivated by this, let us consider only the Givental symmetries $r(z)$ for which 
  \[
r(z)\in \bigoplus_{t \geq 1} \mathcal O(\underline{1})_{2t} z^{(k-1)t}\subset \mathfrak{g}_{\mathcal O}.
 \] 

\begin{lemma} \label{lem:DM-InducedMapOfQuotient}The map $\exp(r(z)).f\colon \ExtHycomm \to \mathcal O$ factors through the quotient map 
 \[
\pi\colon \ExtHycomm\to \bHycomm{k}, 
 \]
that is, there exists a map $g\colon \bHycomm{k} \to \mathcal O$ such that $\exp(r(z)).f = g\pi$. 
\end{lemma}

\begin{proof} 
It is sufficient to show that the nontrivial decorated trees with $n$ labeled leaves and the total degree $2p$ exist only for $n=2+(k-1)p$. 

From the description of the $\kk[\psi]$-action in Proposition \ref{prop:DM-psi-bimodule}, it immediately follows that for each vertex with $k$ inputs, the sum of degrees of psi-classes decorating this vertex must equal to $k-2$, and so the sum of degrees of all psi-classes decorating a tree $T$ is equal to $n-2|V(T)| = n-2|E(T)|-2$, where $V(T)$ is the set of vertices and $E(T)$ is the set of internal edges of the tree~$T$.
On the other hand, the total degree of the operators $r(\psi)$ and the sum of the total degree of psi-classes and double of the number of edges are proportional as $2$ to $k-1$. Hence, the total degree is $2p$ if and only if $n-1 = p(k-1)$. 
\end{proof}

Let us now consider the map $f\colon \ExtHycomm \to (\bBV{k}\vee T(r_1,r_2,\dots), \mathrm{d})$ that sends $m^0_{\underline{2}}$ to $m$ and all other generators to $0$. By Lemma~\ref{lem:DM-InducedMapOfQuotient} there exist a map $g\colon \bHycomm{k}\to (\bBV{k}\vee T(r_1,r_2,\dots), \mathrm{d})$ such that $g \pi = \exp(r(z)).f$. 

\begin{theorem} \label{thm:DM-quasiiso-map}  
The map $g\colon \bHycomm{k}\to (\bBV{k}\vee T(r_1,r_2,\dots), \mathrm{d})$ is a quasi-isomorphism. 
\end{theorem}

As in the case of twisted associative algebras, the proof of this theorem will be gradually constructed in the next sections.

\subsubsection{Quadratic-linear presentation}

Consider the symmetric operad $\bBV{+\infty}$ generated by an element $\Delta$ of arity $0$ and homological degree $1$ and an element $m$ of arity $2$ and of homological degree $0$ subject only to the relations 
 \[
\Delta^2= 0 ,\quad m\circ_1 m-m\circ_2 m=0.
 \]
(As above, we write $+\infty$ to hint that this operad is obtained as the limit of $\bBV{k}$ as $k\to+\infty$, and avoid any possible confusion with the notation $\BV_\infty$ for a cofibrant replacement of $\BV$). We shall now construct a different presentation of this operad which will be useful to study the case of $\Delta$ of finite differential order. Let us denote by $\ell_{\underline{n}}$ the $n$-th Koszul brace of $\Delta$ with respect to $m$. A gauge symmetry argument \cite{MR3510210,MR3385702} shows that these elements satisfy the relations
 \[
\sum_{I\sqcup J=\underline{n}}\ell_{I\sqcup\{\star\}}\circ_\star \ell_J=0.
 \]
The definition of the Koszul braces indicates that the quadratic-linear relation
\begin{align}\label{eq:DMBVquadraticlinear}
\ell_{\underline{n+1}}= 
		\ell_{\underline{n}}\circ_n m
		- m\circ_1 \ell_{\underline{n}}
		- (m\circ_1 \ell_{\underline{n}}).(n,n+1)
\end{align} 
is satisfied (as well as all the relations obtained from it by the action of $S_{n+1}$). 

\begin{proposition}\label{prop:DMBV-infty-Koszul}
The symmetric operad with the generators $m_{\underline{2}}$ and $\ell_{\underline{n}}$, one for each $n\geqslant 0$, subject to the relations
\begin{gather*}
m\circ_1 m-m\circ_2 m =0,\\
\sum_{I\sqcup J=\underline{n}}
	\ell_{I\sqcup\{\star\}}\circ_\star 
		\ell_J =0,\\
		\ell_{\underline{n+1}}= 
		\ell_{\underline{n}}\circ_n m
		- m\circ_1 \ell_{\underline{n}}
		- (m\circ_1 \ell_{\underline{n}}).(n,n+1)
\end{gather*}
is isomorphic to $\bBV{+\infty}$. Moreover, this presentation of $\bBV{+\infty}$ is inhomogeneous Koszul. 
\end{proposition}

\begin{proof}
Let us denote by $\mathcal P$ the operad presented by the given relations. We know that these relations hold in $\bBV{+\infty}$, so there is a surjection of symmetric operads $\mathcal P\twoheadrightarrow\bBV{+\infty}$. We may order monomials in the corresponding free shuffle operad in such a way that the leading monomials of the given relations of $\mathcal P$ and of the relations obtained from them by the symmetric group actions are
\[ m\circ_1 m, 
	\quad (m\circ_1m)(23), 
	\quad 
\ell_{\underline{1}}\circ_1\ell_{\underline{n}}, 	
 \quad
(\ell_{\underline{n}}\circ_p m)(p,n)(q,n+1)
\]
where last relation is present for all $1\leqslant p<q\leqslant n+1$. 
The normal monomials with respect to these leading terms can be 
described as commutative products (using the iterations of the product 
$m$) of elements of the species $\mathcal F\circ\mathcal D$, where $\mathcal F$ is 
the underlying species of the free shuffle operad generated by 
$\ell_{\underline{n}}$ for $n>1$, and $\mathcal D$ is $\kk[\Delta]/\Delta^2$ 
(supported at one-element sets). Let us exhibit a one-to-one 
correspondence between such monomials and elements forming a basis of 
$\bBV{+\infty}$. For that, we simply replace each factor operation 
$\ell_{\{i_1,\ldots,i_p\}}$ by $\Delta\circ_1 m^{[p-1]}_{\{i_1,
\ldots,i_p\}}$; the fact that the elements thus obtained form a basis 
of $\bBV{+\infty}$ is clear. 

In general, the set of normal monomials with respect to the given relations form a spanning set of the quotient operad, so the cardinality of that set in each arity is an upper bound on the dimension of the quotient in that arity. In our case, we have shown that this upper bound is equal to the lower bound given by the dimension of the respective component of $\bBV{+\infty}$, so the operad $\mathcal P$ is isomorphic to $\bBV{+\infty}$, and the given relations form a Gr\"obner basis. Finally, the latter means that our presentation is inhomogeneous Koszul.
\end{proof}

If we wish to incorporate the finite differential order condition, it is very easy to do using the presentation we obtained:
 \[
\bBV{k}\cong \bBV{+\infty}/(\ell_{\underline{k+1}}). 
 \]
This leads to a quadratic-linear presentation of $\bBV{k}$:
\begin{gather}
m\circ_1 m =  m\circ_2 m, \notag\\
\sum_{I\sqcup J=\underline{n}}
		\ell_{I\sqcup\{\star\}}\circ_\star \ell_J 	
		=0 \quad\text{ for $1\leqslant n \leqslant 2k-1$, $|I|\leqslant k-1$ and $|J|\leqslant k$}, \label{eq:truncDMLieinf}\\
	\ell_{\underline{n}}\circ_n m 
		- m\circ_1 \ell_{\underline{n}}
		- (m\circ_1\ell_{\underline{n}}).(n,n+1)
	  =
	  	\begin{cases} \ell_{\underline{n+1}} 
	  		&\text{ for $n<k$,
	  		} \\
	  	0 & \text{ for $n=k$.} 
	  	\end{cases}
	 	\label{brace1}
	 	 \end{gather}
As in the case of twisted associative algebras, some of the terms that we used as leading terms in Proposition \ref{prop:DMBV-infty-Koszul}, for example, $\ell_{\underline{1}}\circ_1\ell_{\underline{n}}$ in the second group of relations with $k+1\leqslant n\leqslant 2k-1$, are no longer present, and so we cannot use the proof of that proposition to conclude that our presentation of $\bBV{k}$ is inhomogeneous Koszul. Nevertheless, that statement turns out to be true. 

\begin{proposition}\label{prop:DMBV-k-Koszul}
The above quadratic-linear presentation of the symmetric operad $\bBV{k}$ is inhomogeneous Koszul for all $k$. 
\end{proposition}

\begin{proof}
Let us focus on $k\ge 2$: in the case $k=1$, the presentation is actually homogeneous, and we understood it well in Section \ref{sec:DMorderone}. Our argument will proceed as follows. We shall first deal separately with the relations~\eqref{eq:truncDMLieinf}
and then show that the relations describing the interaction of the suboperad generated by $\ell_{\underline{n}}$ for
$n\in\underline{k}$, with the suboperad generated by $m$ (isomorphic to the commutative operad) form a quadratic Gr\"obner basis, defining a filtered distributive law between these operads~\cite{MR3302959}. 

For \eqref{eq:truncDMLieinf}, it is easier to consider the suspension of the Koszul dual operad; it is generated by totally symmetric elements $h_{\underline{i}}$ for
 $i\in \underline{k}$, of arity $i$ and homological degree $2i-4$, subject to the relations
 \[
h_{I\sqcup\{\star\}}\circ_\star h_J= h_{I'\sqcup\{\star\}}\circ_\star h_{J'} \quad \text{ for } I\sqcup J=I'\sqcup J'
 \]
where we further require that
$|I|,|I'|<k$ and $|J|, |J'|\leqslant k$.
Let us consider the path degree-lexicographic ordering of 
monomials in the corresponding free shuffle operad. Then the 
elements that are \emph{not} the leading terms of quadratic 
relations are  
\begin{align*} 
h_{\underline{1}}\circ_1 h_{\underline{n}},
	& \quad n\leqslant k,\\
h_{\underline{n-k+1}}\circ_{n-k+1} h_{\underline{k-1}},
	& \quad k\leqslant n\leqslant 2k-2. 
\end{align*}
It follows that this operad has a PBW basis of the form 
 \[
h_{I_1\sqcup\{\star\}}\circ_\star\cdots\circ_\star h_{I_p\sqcup\{\star\}}
 \]
where $I_1\sqcup\cdots\sqcup I_{p-1}\sqcup(I_p\sqcup\{\star\})$ is a decomposition of $\underline{n}\sqcup\{\star\}$ as a disjoint union of (possibly empty) intervals for which a non-empty interval cannot be followed by an empty one and only the first non-empty interval can be of length strictly less than $k-1$. This implies that the operad presented via relations listed in \eqref{eq:truncDMLieinf} also has a PBW basis, and a quadratic Gr\"obner basis of relations.

To proceed, one uses word operads \cite{MR4114993}, and 
considers the monoid of ``quantum monomials'' $\mathsf{QM}=
\langle x,y,q\mid xq=qx,yq=qy,yx=xyq\rangle$ and the map from 
the free shuffle operad generated by the elements $m$ and 
$\ell_p$ to the word operad associated to $\mathsf{QM}$ sending 
$m$ to $(x,x)$ and $\ell_p$ to $(y,y,\ldots,y)$; the 
corresponding ordering moves the operation $m$ towards the root 
in the normal forms. Since we know that the suboperads 
generated by $m$ and by $\{\ell_p\}_{p\ge 1}$ separately have 
Gr\"obner bases, it remains to check that all ``mixed'' S-
polynomials can be reduced to zero. There are three groups of 
``mixed'' S-polynomials to consider: those of \eqref{brace1} 
with the associativity relations for $m$, those of 
\eqref{brace1} between themselves, and those of \eqref{brace1} 
with the relations for the operations $\ell_p$. In each of 
these cases, the result is obtained by a somewhat tedious but 
direct calculation; for the latter group of S-polynomials which 
is perhaps the hardest to deal with directly, it is beneficial 
to note that the relations between the operations $\ell_p$ in 
the free operad generated by these operations can be viewed as 
commutators in the convolution Lie algebra between the linear 
dual of the commutative operad and the latter free operad, and 
the requisite property of S-polynomials ultimately boils down 
to the property of the commutator of derivations being a 
derivation; the reader is invited to consult \cite[Prop.~3.13]{BaMa} for another interpretation of the same observation. 
Thus, the quadratic-linear presentation of 
$\bBV{k}$ forms a quadratic Gr\"obner basis, which in turn 
implies that this presentation is inhomogeneous Koszul.  
\end{proof}

\subsubsection{The Koszul dual dg operad and its homology}

We can now describe the (suspension of the) Koszul dual operad 
$(\bBV{k}^{!},\partial)$.  Its underlying operad is easily seen 
to be generated by totally symmetric elements $h_{\underline i}$ for 
$i\in \underline{k}$, of arity $i$ and of homological degree $2i-4$, 
and a symmetric element $g$ of arity $2$ and of homological degree 
$1$, subject to the relations
\begin{align}\label{eq:DM-rel-s-koszul}
(g\circ_1 g).(1+(123)+(132)) &=0,\notag\\
h_{I\sqcup\{\star\}}\circ_\star h_J
	- h_{I'\sqcup\{\star\}}\circ_\star h_{J'}
	 &= 0 \quad \text{ for } I\sqcup J=I'\sqcup J', 
\notag\\
g\circ_1 h_{\underline{n}}+ \sum_{i=1}^n (h_{\underline{n}}\circ_n g).(i,n) &=0 \quad \text{ for $n\in \underline{k}$ }.
\end{align}
where as before we require that
$|I|, |I'|<k$ and $|J|, |J'|\leqslant k$.
 The results of the previous section imply that these relations 
give a distributive law between the operad $\mathcal H_k$ generated 
by all the elements $h_{\underline i}$ for $i\in \underline{k}$, 
and the shifted Lie operad generated by the element $g$. The 
presence of quadratic-linear relations in $\bBV{k}$ means that 
the Koszul dual operad has a nonzero differential $\partial$ 
which is a derivation given on generators by 
$\partial(h_{\underline{1}}) = 0$, $\partial(g)  = 0$
and by 
\[
\partial(h_{\underline n})  =  
	\sum_{\{i,j\}\subset\underline{n} 	
\text{ and } i\ne j} 
	(h_{\underline{n-1}}\circ_{n-1} g).((i,n)(j,n+1))
	 	\quad \text{for $2\leqslant n\leqslant k$}.
\]
For the rest of the proof, it will be important to note that we have a chain of inclusions of dg operads
 \[
\mathcal H_1\subset\cdots\subset \mathcal H_k\subset\cdots 
 \]
we denote by $\mathcal H_{+\infty}$ the union of all of them, and by $\bBV{+\infty}^!$ the union of all the dg operads $(\bBV{k}^{!},\partial)$. That latter union, viewed as a chain complex, is homotopy equivalent to the (right) Koszul complex of the shifted Lie operad (since the operad $\mathcal H_{+\infty}$ is the Koszul dual of the dg operad of shifted Lie-infinity algebras); as such, it is acyclic. Thus, the long exact sequence of the homology implies that the homology of the complex $(\bBV{k}^{!},\partial)$ is isomorphic to the subspace of $\bBV{k}^{!}$ obtained as boundaries of elements outside of that subcomplex.

As in the case of the twisted associative algebras, let us
denote by $K_{\underline{n}\sqcup\{\star\}}^{2t}$ an element of arity 
$n+1$ and degree $2t$ that is equal to the composition 
$h_{I_1\sqcup\{\star\}}\circ_\star\cdots\circ_\star 
h_{I_{n-t}\sqcup \{\star\}}$ for any choice of partition 
$\underline{n}=I_1\sqcup \cdots\sqcup I_{n-t}$ with $t\leqslant n$. 
Exactly as in the case of twisted associative algebras, we find that 
the since we are interested in the quotient of the cobar construction 
by $s^{-1}(\delta^r)^\vee$ for $r>0$, and one can choose the 
contracting homotopy for which vanishes on elements that do not 
contain $\delta$, %\delta is h_1 
the higher structures of an operad up to homotopy 
vanish, and the homology has just an operad structure. As an operad, 
the homology is immediately seen to be generated by the elements 
 \[
g_{\underline{n}}:=\partial(K_{\underline{n}}^{2t}), \quad n=q(k-1)+2,\quad t=q(k-2)+2,\quad q\geqslant 0.
 \] 

It remains to show that the elements $g_{\underline{n}}$ satisfy precisely the relations described in Proposition \ref{prop:bHycomm-Koszul}. The fact that these relations are satisfied is obtained by direct inspection, so there is a surjective map from $\bHycomm{k}^!$ onto the homology. To conclude, we shall argue as follows. From the proof of Corollary \ref{cor:dim-b-grav}, we know the normal forms in $\bHycomm{k}^!(n)$, and so it enough to show that we can find linearly independent elements spanning the homology which are in one-to-one correspondence with those normal forms. For each normal form 
 \[
g_p(u_1,u_2,\ldots,u_{p-1},\mathrm{id}_{\{j\}}),  
 \]
let us consider the corresponding element 
 \[
K_{I}^{2t}\circ(u_1',\ldots, u_{p-1}',\mathrm{id}_{j}),  
 \]
where $u_i'$ is obtained from $u_i$ by replacing $g_2$ with $g$. (Note that $t$ is computed uniquely from $n$.) By a direct inspection, the boundaries of such elements are linearly independent in $\bBV{+\infty}^!$, and therefore the surjection from $\bHycomm{k}^!$ onto the homology must be an isomorphism. As an immediate by-product, we obtain one of the central results of this section.

\begin{theorem}
The homotopy quotient of $\bBV{k}$ by $\Delta$ is represented by $\bHycomm{k}$. 
\end{theorem}

\subsubsection{Proof of Theorem \ref{thm:DM-quasiiso-map} } The argument is essentially the same as in the proof of Theorem \ref{thm:LM-quasiiso-map}. The dg operad $(\bBV{k}\vee T(r_1,r_2,\dots), \mathrm{d})$ computes the homotopy quotient by $\Delta$, thus its homology is isomorphic to $\bHycomm{k}$. In addition, the following statements are true. 
\begin{enumerate}
\item The map $g$ is a morphism of dg operads.
\item The images of the generators of $\bHycomm{k}$ do not vanish in the homology of $\mathrm{d}$ (the image of $\mathrm{d}$ is contained in the ideal generated by~$\Delta$).
\item In the homology of $\mathrm{d}$, the subspace of elements of arity $2+(k-1)t$ and homological degree $2t$ is one-dimensional for each $t\geqslant 0$.
\end{enumerate}
This implies that the map $g$ is surjective on the homology of $\mathrm{d}$, hence it is a quasi-isomorphism.\qed

\section{Homotopy quotients in the non-symmetric operad case}

\subsection{Order of differential operator}

Let us recall the definition of the B\"orjeson braces (see
\cite{MR3624036,DSV-ncHyperCom,MR3385702} for further details), 
and relate them to the notion of a noncommutative 
Batalin--Vilkovisky algebra.

\begin{definition}
Let $\mathcal O$ be a non-symmetric operad, and let $\mu\in
\mathcal O(\underline{2})_0$ be an associative binary operation.
For an element $\varphi\in\mathcal O(\underline{1})$, we define 
the $k$-th \emph{B\"orjeson brace} $\beta_k^\varphi\in
\mathcal O(\underline{k})$ by the formulas
\begin{align*}
\beta_1^\varphi&:=\varphi, \quad
\beta_2^\varphi:=\varphi\circ_1\mu-\mu\circ_1\varphi-\mu\circ_2\varphi,\\
\beta_3^\varphi&:=\varphi\circ_1\mu\circ_2\mu-\mu\circ_1(\varphi\circ_1\mu)-\mu\circ_2(\varphi\circ_1\mu)+\mu\circ_2(\mu\circ_1\varphi),\\
\beta_n^\varphi&:=\beta_{n-1}^\varphi\circ_2\mu \text{ for $n\geqslant 4$}. 
\end{align*} 
We say that an element $\varphi\in\mathcal O(\underline{1})$ is of differential order at most $k$ (with respect to $\mu$) if $\beta_{k+1}^\varphi=0$.
\end{definition}

The non-symmetric operad $\ncBV$ of noncommutative Batalin--Vilkovisky algebras has two equivalent definitions. One of them says that it is generated by an element $\Delta$ or arity $1$ and homological degree $1$ and an element $\mu$ of arity $2$ and homological degree $0$ subject to the relations
\[ \Delta^2=0, \quad 
\mu\circ_1\mu=\mu\circ_2\mu, \quad
\beta_{3}^\Delta=0.
\] 
The other one says that the same operad admits a quadratic-linear presentation using two binary operations $\mu$ and $\lambda$ of respective degrees $0$ and $1$, and a unary generator $\Delta$ of degree $1$, subject to the relations
\begin{align*}
\Delta\circ_1\Delta 				&=0, 
	& \mu\circ_1\mu-\mu\circ_2\mu &=0,\\
\mu\circ_1\lambda-\lambda\circ_2\mu &=0, 
	& \lambda\circ_1\mu-\mu\circ_2\lambda&=0, \\
\lambda\circ_1\lambda+\lambda\circ_2\lambda &=0, 
	&\Delta\circ_1 \mu-\mu\circ_1\Delta-\mu\circ_2\Delta &=\lambda, \\
& & \Delta\circ_1 \lambda+\lambda\circ_1
	\Delta+\lambda\circ_2\Delta &=0. 
\end{align*}    
Examining the first of those definitions, it is very easy to generalise it to use other values of the differential order. 

\begin{definition}
The non-symmetric operad $\bncBV{k}$ is generated by an element $\Delta$ or arity $0$ and homological degree $1$ and an element $\mu$ of arity $2$ and homological degree $0$ subject to the relations
\[
\Delta^2=0, \quad
\mu\circ_1\mu=\mu\circ_2\mu, \quad
\beta_{k+1}^\Delta=0.
	\]
\end{definition}

\subsection{The case of differential order one}\label{sec:NCorderone}

 The non-symmetric operad $\bncBV{1}$ is generated by an element $\Delta$ of arity $0$ and homological degree $1$ and an element $\mu$ of arity $2$ and homological degree $0$ subject to the relations
\[ 
\Delta^2=0, \quad
\mu\circ_1\mu=\mu\circ_2\mu, 
\quad \Delta\circ_1\mu-\mu\circ_1\Delta-\mu\circ_2\Delta=0.
\]
It is a (homogeneous) quadratic non-symmetric operad, and one can immediately see that it is Koszul; in fact, it has a quadratic Gr\"obner basis for which the corresponding PBW basis consists of elements 
 \[
\mu^{[n]} \text{ and } (\cdots((\mu^{[n]}\circ_{i_1}\Delta)
	\circ_{i_2}\Delta)\cdots)\circ_{i_p}\Delta
\text{ for $n\geqslant 0$ and $I=\{i_1,\ldots,i_p\}
	\subset\underline{n+1}$.}
 \]
where we denote by $\mu^{[n]}$ the $n$-fold composite of the operation 
$\mu$: $\mu^{[0]}=\mathrm{id}$, $\mu^{[n+1]}=\mu^{[n]}\circ_1 \mu$.
As before, we consider all left combs $\mu^{[n]}$ with a possibly
empty subset of leaves decorated by $\Delta$.

The Koszul dual twisted associative coalgebra 
$\bncBV{1}^{\ac}$ has a PBW basis consisting of the elements 
 \[
u_{n,r}:=(s\Delta)^r (s\mu)^{[n]}, \quad r 
	\text{ and } n\geqslant 0,
 \]
with the obvious cooperad structure (splitting $r$ into a sum and 
splitting the set of arguments of $(s\mu)^{[n]}$ into a disjoint 
union). The homotopy quotient of $\bncBV{1}$ by $\Delta$ is 
represented by the dg non-symmetric operad which is the quotient 
of the cobar construction $\Omega(\bncBV{1}^{\ac})$ by the two-
sided ideal generated by the elements $s^{-1}u_{0,r}$ with $r>0$. 
Thus quotient creates from the cobar construction another free 
non-symmetric operad equipped with a quadratic differential: it 
is generated by the elements $s^{-1}u_{\underline{n},r}$ with 
$n>0$, and the differential as above but with all terms involving 
$s^{-1}u_{0,r}$ being suppressed. This differential is the 
differential of the cobar construction of the non-symmetric 
cooperad which is the quotient of $\bncBV{1}^{\ac}$ by the 
\emph{subspace} spanned by the elements $u_{0,r}$ with $r>0$. 
The linear dual non-symmetric operad is generated by the elements 
$u_{1,p}$ with $p\geqslant 0$, of arity $2$ and homological degree $2p+1$ 
subject to the relations
\begin{gather*}
u_{1,p}\circ_1u_{1,q}+u_{1,p}\circ_2 u_{1,q}=0,\\
u_{1,p}\circ_1u_{1,q}=u_{1,p-1}\circ_1 u_{1,q+1}.
\end{gather*}
It is easy to see that this non-symmetric operad is Koszul; in fact, it has a quadratic Gr\"obner basis whose linear basis  corresponds to a PBW basis under the map
$ u_{n,r} \longmapsto \{(u_{1,0})^{[n-1]}\circ_1 u_{1,r}$ where $r \geqslant 0$ and $n>0$. Therefore, the homology of its cobar complex is the Koszul dual non-symmetric operad which has generators $\mu^{2p}$ of arity $2$ and homological degree $2p$ for $p\geqslant 0$ with the relations stating that the formal sum $\overline{\mu}=\sum_{p=0}^\infty \mu^{2p}$ satisfies
 \[ 
\overline{\mu}\circ_1\overline{\mu}=\overline{\mu}\circ_2\overline{\mu}. 
 \]
To make sense of these relations, one has to separate the terms by total homological degree $2d\geqslant 0$ into infinitely many relations 
 \[
\sum_{p+q=d}\mu^{2p}\circ_1\mu^{2q}=\sum_{p+q=d}\mu^{2p}\circ_2\mu^{2q}. 
 \]
involving finitely many terms each. We established the following result.

\begin{theorem}
The homotopy quotient of $\bncBV{1}$ by $\Delta$ is represented by the non-symmetric operad with generators $\mu^{2p}$ of arity $2$ and homological degree $2p$ for $p\geqslant 0$ 
subject to the relations 
 \[
\sum_{p+q=d}\mu^{2p}\circ_1\mu^{2q}=\sum_{p+q=d}\mu^{2p}\circ_2\mu^{2q}, \quad d\geqslant 0. 
 \]
\end{theorem}

The computation of the homotopy quotient of $\bBV{2}$ (that is, the operad $\BV$) by $\Delta$ was first done in \cite{DSV-ncHyperCom}. Our approach for the case of finite order at least two is completely uniform, and so we shall proceed with developing a general formalism instead of focusing on that particular case.

\subsection{Extended noncommutative hypercommutative operad and Givental symmetries} %\volodya{recheck this section}

We shall now define a non-symmetric operad generalizing the noncommutative hypercommutative operad of \cite{DSV-ncHyperCom} which is one of the protagonists of the calculation of homotopy quotients by operators of higher order.

\begin{definition}
The \emph{extended noncommutative hypercommutative operad}, 
denoted $\ncExtHycomm$, is the non-symmetric operad generated by 
operations $\mu_t^{2p}$ of all possible arities $t\geqslant 2$ 
and all possible non-negative even homological degrees $2p
\geqslant 0$, subject to the relations that can be described as 
follows. Consider the sums of all generators of a fixed arity $
\overline \mu_t\coloneqq \sum_{p=0}^\infty \mu_t^{2p}$, and impose, for each arity $n\ge 3$ and each choice of $i-1\in \underline{n-2}$, the condition
\begin{equation}\label{eq:ncHCrel}
\sum_{j=1}^{i-1} 
	\overline \mu_{n-i+j}\circ_j \overline \mu_{i-j+1}=
		\sum_{j=i+1}^n
		 \overline \mu_{n-j+i}\circ_i \overline \mu_{j-i+1}.
\end{equation}
In plain words, the sum of of all possible graftings of two corollas of total arity
	$n$, where $i-1$ and $i$ appear at the top corolla, and $i+1$ appears at the bottom corolla,
is equal to the sum	of all possible graftings of two corollas of total arity
	$n$, where $i$ and $i+1$ appear at the top corolla, and $i-1$ appears at the bottom corolla.
\end{definition}

Note that if each operation $\overline \mu_t$ were of homological 
degree $2t-4$, these identities would precisely define the operad 
$\ncHycomm$ of hypercommutative algebras \cite{DSV-ncHyperCom}. 
In our case, each of these operations is an infinite sum of 
operations of different homological degrees, and so one should 
separate these identities by homological degrees into infinitely 
many relations involving finitely many terms each. 

In the case of the spaces $\mathcal B(n)$, one can also define psi-classes that play important role in algebraic and geometric results involving those spaces. Let us define a formal algebraic version of those classes in the case of $\ncExtHycomm$, where there is no longer a geometric interpretation of those as cohomology classes. We start with a preparatory lemma,
whose proof is completely analogous to that of~\ref{DM:action-indep} which we have already given. We note that in that
statement, we are considering sums of grafting of two
corollas of a fixed total arity, and fixing two elements at
the top corolla. In the nc analog, a two element subset 
is replaced by a two element saturated ordered subset:

\begin{lemma}\label{NC:action-indep}
For each $i\in \underline{n-1}$ the sum 
\[ \sum_{i,i+1 \text{ at top}} 
	 \overline{\mu}_{n_1} \circ_j \overline{\mu}_{n_2}  \]
	of all possible graftings of two corollas of total arity
	$n$, where $i$ and $i+1$ appear at the top corolla, is
	independent of $i$.	 
	 Similarly, for every $i\in\underline{n}$ the
	sum	\[ 
	 \sum_{\text{$i$ at top}} 
	 \overline{\mu}_{n_1} \circ_j \overline{\mu}_{n_2}\]
	 consisting of all possible graftings 
	 of two corollas of total arity
	$n$, where $i$ appears at the top corolla, is
	independent of $i$. \qed
\end{lemma}

With this lemma at hand, one can extend the algebra over which 
the ns operad $\ncExtHycomm$ is defined to $\kk[\psi]$ with 
$\deg \psi = -2$, where the action of $\psi$ on $\overline \mu_n$ is given as in the statement of the
following proposition. The proof is completely
analogous to the ones we have already given, so we omit it.
 
\begin{proposition}\label{prop:NC-psi-bimodule}
If, for a formal variable $\psi$ of degree $-2$ we let
\begin{align}
	\psi \circ_1 \overline \mu_n  & = \sum_{i,i+1 \text{ at top}} 
	 \overline{\mu}_{n_1} \circ_j \overline{\mu}_{n_2} 
	& \quad \text{ for any }  i\in \underline{n-1};
	\\
	\overline \mu_n \circ_i \psi+ \psi\circ_1\overline \mu_n &=  
	 \sum_{\text{$i$ at top}} 
	 \overline{\mu}_{n_1} \circ_j \overline{\mu}_{n_2}
	\notag & \quad  \text{ for any } i\in \underline{n},
\end{align}
where $n_1+n_2 = n+1$ and $j$ encodes the graftings described in Lemma \ref{NC:action-indep},
these formulas define a $\kk[\psi]$-bimodule structure
on $\ncExtHycomm$.\qed
\end{proposition}

In particular, as it happens in the case of the classical
non-commutative hypercommutative operad, the action at the
first and last leaves are zero~\cite[Proposition 6.1.1]{DSV-ncHyperCom}. The interested reader may check that this implies
that the squares of the actions at internal leaves are also 
zero.
Let us show how the psi-classes can be used to define symmetries 
of representations of $\ncExtHycomm$ that generalise the Givental 
action \cite{ShaZvo-LMCohft}. Precisely, let $\mathcal O$ be an 
arbitrary dg non-symmetric operad, and consider the set of dg 
non-symmetric operad morphisms $\Hom(\ncExtHycomm,\mathcal O)$. For a 
formal variable $z$ of degree $-2$, let us consider the complete 
Lie algebra $\mathfrak{g}_{\mathcal O}$ which is the kernel of the 
``augmentation''
 \[
	\bigoplus_{p=0}^\infty 
		\mathcal O(\underline{1})_{2p}
			\otimes\kk \llbracket z \rrbracket
				\twoheadrightarrow \mathcal O(\underline{1})_0
 \]
annihilating elements of $\mathcal O$ of positive homological degrees as well as all positive powers of $z$. 
We shall now define an action of $\mathfrak{g}_{\mathcal O}$ by infinitesimal symmetries of $\Hom(\ncExtHycomm,\mathcal O)$. For that, we put
\begin{align*}
	((rz^t).f) (\overline \mu_{n}) &=  
		r \circ_1 f(\psi^{t}\circ_1\overline \mu_{n}) + (-1)^{t+1} \sum_{i\in \underline{n}} f(\overline \mu_{n}\circ_i \psi^{t})\circ_i r 
	\\ \notag
	&+ \sum_{i+j=t-1} (-1)^{j+1} 
		\sum_{\substack
			{k ,\,q+2\, \in \underline{n} }
						}
				 f(\overline \mu_{n-q+1}\circ_{k} \psi^j) 
				 	\circ_{k} 
				 		\left(r\circ_1 
				 			f(\psi^i \circ_1 \overline \mu_{q})
				 				\right).
\end{align*}
and extend it linearly to $\mathfrak{g}_{\mathcal O}$. It turns out that this way we obtain infinitesimal symmetries of $\Hom(\ExtHycomm,\mathcal O)$.

\begin{proposition}\label{prop:nc:GiventalLie}
For any $f\in \Hom(\ncExtHycomm,\mathcal O)$ we have 
$
f+\varepsilon r.f\in \Hom_{\kk_\varepsilon}(\ncExtHycomm_\varepsilon,\mathcal O_\varepsilon)$,
and for all $\lambda_1,\lambda_2\in \mathfrak{g}_{\mathcal O}$, we have
 $
[\lambda_1,\lambda_2].f=\lambda_1.(\lambda_2.f)-\lambda_2.(\lambda_1.f)$.
\end{proposition}

\begin{proof}
Let us begin by noting that the formulas for the Givental
action just before the statement of this proposition
simplify in the non-commutative case. Indeed,
if $rz^t$ has $t=1$ then the formula reads
\begin{align*}
	((rz).f) (\overline \mu_{n}) &=  
		r \circ_1 f(\psi \circ_1\overline \mu_{n}) + 
		 \sum_{j \in \underline{n}} f(\overline \mu_{n}\circ_j \psi)\circ_j r 
	\\ \notag
	&+  
		\sum_{\substack
			{j \in \underline{n} }
						}
				 f(\overline \mu_a) 
				 	\circ_j 
				 		\left(r\circ_1 
				 			f(\overline \mu_b)
				 				\right),
\end{align*}
so let us begin by considering this case first, in which
the resulting equation consists of trees with exactly three internal
vertices. Both sides of the
corresponding equation summing over two level trees where either
$i-1$ and $i$ appear at the top and $i+1$ at the bottom,
or where $i-1$ appears at the bottom and $i$ and $i+1$
at the top will have terms of various types:
\begin{enumerate}
\item Terms where $r$ appears at the root, in which
case the freedom to choose how the psi-class
acts at the root allows us to assume all summands consist
of sequential compositions of trees. In this case,
the summands will consist of terms where the ``$(i-1,i,i+1)$-associativity''
appears at the bottom two trees, and we are grafting one more tree
at the top, or where it appears at the top, and we have grafted
a tree at the bottom. 
\item There will be terms where $i-1$ appears at the bottom
and $i$ in the middle tree, while $r$ appears at the top (either
before or after the third tree). All of these terms are obtained
by grafting the third tree and $r$ into the left hand side 
of associativity equation, and will cancel with the corresponding
right hand side.
\item Terms where $r$ divides the first tree from the other two, in
which case $i-1$ is at the bottom, and split from $i$ and $i+1$ by $r$,
which appear at the top. These terms appear twice in the left
hand side of the ``$(i-1,i,i+1)$-associativity equation'' for
the deformed terms and with opposite signs, so they cancel.
Similarly, there will be corresponding terms for the right
hand side of the equation that cancel out.
\item In a similar fashion, there are terms in the left
hand side where $r$ separates the first two trees from the last,
and $i-1$ at the bottom from $i,i+1$ at the topmost tree. These
terms appear twice with opposite signs, and they cancel. 
\item Finally, the remaining the terms will consist of associativity
equations appearing along with a third tree (in parallel or sequentially)
and an appearance of $r$ at a leaf.
\end{enumerate}
This accounts for all possible terms, and proves the case for $t=1$.
In case $t>1$, the action reads
\begin{align*}
	((rz^t).f) (\overline \mu_{n}) &=  
		r \circ_1 f(\psi \circ_1\overline \mu_{n}) -
		 \sum_{j \in \underline{n}}
		  f(\overline \mu_{a}\circ_j \psi)\circ_j r
		  \circ_1 f(\psi^{t-2}\circ_1 \overline\mu_{b})
	\\ \notag
	&+  \sum_{j \in \underline{n} }
			 f(\overline \mu_a) 
				 	\circ_j
				 		r\circ_1 
				 			f(\psi^{t-1} \circ_1 \overline \mu_{b}),
\end{align*}
and the observation we make is that for each $s\geqslant 1$ the
element $\psi^s\circ_1\overline{\mu}_b$ can be chosen to consist
of all possible ways of decomposing $\overline{\mu}_b$ into a right
comb (with the same total number of leaves) with $s$ internal
edges. One can then notice
that the resulting action above still contains three terms
of the same combinatorial shape as the action in the case
$t=1$, where $r$ is either at the root, preceded by an action
of the psi-class at an internal vertex, or dividing the
tree into a corolla and now a sum of right combs, and 
carry out the same analysis as above to conclude.
\end{proof}

\subsection{The higher order versions of the non-symmetric operad \texorpdfstring{$\ncHycomm$}{ncHycomm}}
%\volodya{recheck this section}
We are now ready to introduce the non-symmetric 
operad which will turn out to represent the 
homotopy quotient of $\bncBV{k}$ by $\Delta$.

\begin{definition}
Let $k\geqslant 1$. The non-symmetric operad 
$\bncHycomm{k}$ is defined as the quotient of 
the extended noncommutative hypercommutative 
operad $\ncExtHycomm$ by the ideal generated by 
all $\mu_{{t}}^{2p}$ where either 
$t\neq 2\pmod{k-1}$ or $t= 2+v(k-1)$ for some 
$v\geqslant 0$ but $v\ne p$.
\end{definition}

Note that this definition means that the homological degree $2p$ of each generator $\mu_{{t}}^{2p}$ of the non-symmetric operad $\bncHycomm{k}$ determines its arity. In particular, the non-symmetric operad $\bncHycomm{k}$ is always generated by a collection of elements of degrees $0,2,4,\ldots$ whose arities vary depending on~$k$. For example,
\begin{itemize}
    \item $\bncHycomm{1}$ is generated by binary operations $\mu_{{2}}^{2p}$ of arity $2$ for $p\geqslant 0$ subject to the relations guaranteeing that $\overline \mu = \sum_{p=0}^\infty \mu_{{2}}^{2p}$ is associative.
    \item $\bncHycomm{2}$ is generated by operations $\mu_{{2+p}}^{2p}$ for
    $p\geqslant 0$, and is isomorphic to $\ncHycomm$.
\end{itemize}
In general, the operad $\bncHycomm{k}$ has the following relations, for all values of $d\geqslant 0$, the corresponding value $n=3+d(k-1)$, and all choices of $q-1\in \underline{n-2}$: 
\begin{equation}
\sum_{ 1\leqslant i \leqslant q-1} 
	\overline \mu_{n-q+i}^{2d_1}
		\circ_i 
		\overline \mu_{q-i+1}^{2d_2}=
\sum_{\substack{q+1\leqslant i\leqslant n}} 
	\overline \mu_{n-i+q}^{2d_1}
		\circ_k 
		\overline \mu_{i-q+1}^{2d_2} .
\end{equation}
where $d_1+d_2=d$ and
 we further require that $q-i+1=2+d_2(k-1)$ in the
left hand side sum and that $i-q+1=2+d_2(k-1)$ in the
right hand side sum.
Let us establish that this operad is Koszul, and determine its Koszul dual.

\begin{proposition}\label{prop:bNCHycomm-Koszul}
The suspension of the non-symmetric operad 
$\bncHycomm{k}^!$ is generated by elements 
$\widetilde\gamma_{t+2}:=\gamma_{(k-1)t+2}$ of arity $(k-1)t+2$ and 
homological degrees $1+2t(k-2)$ for 
$t\geqslant 0$, 
subject to the relations 
\begin{align*}
\sum\limits_{j=r}^{r+(t-2)(k-1)}
	\tilde\gamma_{{n-1}}\circ_j\tilde\gamma_{{2}}
		&=\tilde\gamma_{{n-t+1}}
			\circ_r\tilde\gamma_{t},
				\text{ for all }
					3\leqslant t<n 
						\text{ and }
							1\leqslant r
								\leqslant
			 \mathrm{ar}(\lambda_{n-k+1}).
						\\
\sum\limits_{j=1}^{(k-1)n+2}\tilde\gamma_{{n+2}}\circ_j
	\tilde\gamma_{{2}}&=0 \text{ for all $n\geqslant 0$}.
\end{align*}
This operad is Koszul.   
\end{proposition}

\begin{proof}
The argument is very similar to that of 
Proposition~\ref{prop:NCHycommDual}. First, we 
note that these relations annihilate the 
relations of $\bncHycomm{k}$ under the pairing 
between the weight two components of the 
respective free operads (the suspension ensures 
that we must compute the usual pairing without 
sign twists); moreover, in each arity $n$ 
congruent to $3$ modulo $k-1$ (which are the 
only arities in which relations appear) all 
these relations but the last one correspond to 
subsets of cardinalities bigger than $2$ and 
congruent to $2$ modulo $k-1$, so they span a 
subspace whose dimension is equal to the 
dimension of the annihilator of the space of 
quadratic relations of $\bHycomm{k}$ in this 
arity. To establish the Koszul property, we 
note that, if we ignore the ``wrong'' 
homological degrees, the relations we consider 
form a subset of the relations of the non-
symmetric operad $\ncHycomm^!$, and so it will 
be convenient for us to use the known argument 
for that operad discussed in the proof of 
Proposition \ref{prop:NCHycommDual}. For the 
ordering from that proof, a relation of 
$\ncHycomm^!$ is, ignoring the homological 
degrees, a relation of $\bncHycomm{k}^!$ if and 
only if its leading term is an element of 
$\bncHycomm{k}^!$, and so computing the reduced 
form of an S-polynomial between two such 
relations does not produce any other elements. 
Consequently, the Diamond Lemma implies that 
the relations of $\bncHycomm{k}^!$ form a 
Gr\"obner basis for the ordering we are 
considering, and so this operad is Koszul. 
\end{proof}

\begin{corollary}\label{cor:dim-b-ncgrav}
Suppose that $k>1$. The dimension of $\bncHycomm{k}^!(n)$ is equal to $\sum_{p=0}^n\binom{n-1}{p(k-1)}$.
\end{corollary}

\begin{proof}
The normal monomials with respect to the leading terms of the Gr\"obner basis from the proof of Proposition \ref{prop:bHycomm-Koszul} are precisely
$
\gamma_p(\mathrm{id},u_1,\ldots,u_{p-1})$,
where all $u_i$ for $i\in\underline{p-1}$
 are ``left combs'' obtained by iterations of $\gamma_2$ (these left combs form a basis of the shifted associative suboperad) and $p=2
 \pmod{k-1}$ (this will be implicit in what
 follows).

Let us compute the number of monomials like that. To do this, we note that we may equivalently enumerate basis elements
 \[
\gamma_{p-1}(u_1,u_2,\ldots,u_{p-1}).  
 \]
Using the stars-and-bars counting method, we see that the number of such elements is equal to the number of subsets of $\underline{n-1}$ of cardinality $p-2$. Recalling that $p$ is the arity of $\gamma_p$, we see that $p-2$ is divisible by $k-1$, and the claim follows.  
\end{proof}

\subsection{The case of arbitrary differential order}

In this section, we shall compute the homotopy quotient of $\bncBV{k}$ by $\Delta$ for any given order $k$. As in Section \ref{sec:LM-HigherOrder}, that homotopy quotient can be computed by adjoining the arity zero operations $r_i$ for each $i\geqslant 1$, of degree $\deg r_i=2i$, with the differential $\mathrm{d}$ compactly defined by writing 
\begin{equation}
    \mathrm{d} \exp(r(w)) =  \exp(r(w)) (\mathrm{d} + w\Delta).
\end{equation}
The computation consists of two parts. In the first part, inspired by \cite{KMS}, we state the main result using the Givental action on representations of the extended hypercommutative operad. In the second part, we prove it using the quadratic-linear Koszul duality theory.

\subsubsection{Using the Givental action to formulate the main statement}

Let $\mathcal O$ be an arbitrary dg non-symmetric operad, and suppose that $f\colon \ncExtHycomm\to\mathcal O$ is a map that sends all generators except for $\mu_{\underline{2}}^0$ to zero. We shall be using the result of exponentiation of the action of elements of the Lie algebra $\mathfrak{g}_{\mathcal O}$, and to that end we shall introduce a useful language to discuss elements of the form $\exp(r(z)).f$.  

For a given arity $n$, we shall consider planar rooted trees with $n$ leaves; such a tree has $n$ half-edges corresponding to leaves and another half-edge corresponding to the root. Each such tree gives rise to an element of $\mathcal O$ as follows. Let us first associate certain data to each half-edge of the tree: 
\begin{itemize}
\item the half-edge corresponding to the root is decorated by one of the terms $x_p\psi^p$ of the Taylor series expansion of $\exp(r(\psi))$,
\item each half-edge corresponding to a leaf is decorated by one of the terms $y_q\psi^q$ of the Taylor series expansion of $\exp(-r(-\psi))$,
\item each internal edge is decorated by one of the terms $z_{r,s}(\psi')^r(\psi'')^s$ of the Taylor series expansion of 
    \begin{align}
        \frac{\mathrm{Id} - \exp(-r(-\psi'))\exp(r(\psi''))}{\psi'+\psi''}\,,
    \end{align}
   where $\psi'$ and $\psi''$ should be thought of as the psi-classs on the half-edge closer to the root and further from the root, respectively.  
\end{itemize}
To a thus decorated tree $T$, we can associate 
an element of $\mathcal O$ denoted by $F(T)$ 
obtained as the composition according to the 
tree $T$, of the coefficients $x_p, y_q, 
z_{r,s} \in \mathcal O(\underline{1})$ and, for each 
vertex with the half-edges $e_1, \ldots, e_t$, 
the element $f((\mu^0_{\underline{2}})^{[t]})
\in \mathcal O(\{i_1,\ldots,i_t\})$. We also associate 
to $T$ a combinatorial factor $C(T)$ determined by 
the powers of $\psi$s associated to half-edges: 
for each vertex $v$ such that it has the $k_0$-th 
power of the psi-class on the half-edge 
corresponding to the root and the power $k_1, \ldots, k_t$ of the psi-class on the half-edges corresponding to the leaves, we put 
$C(v)=1$ if $(k_0,\ldots,k_t)$ is a partition of
$t-2$ with $k_1=k_t=0$ and $k_i\leqslant 1$
for all $i\in \underline{t-1}$, and let it be
zero otherwise. We then define $C(T)$ as the product of $C(v)$ over all vertices.

\begin{lemma}  The map $\exp(r(z)).f$ on the generators $\mu_{n}^{2p}\in \ncExtHycomm$ is given by $\mu_{n}^{2p} \mapsto \sum_T C(T)F(T)$, where the sum is taken over all decorated plane rooted trees with $n$ labeled leaves of the total degree $2p$. 
\end{lemma}

\begin{proof} It is a direct integration of the Lie algebra action defined in Proposition~\ref{prop:nc:GiventalLie}, see~\cite{DSV-ncHyperCom}, and the explanation of our formulas is fully parallel to the one given in the proof of Lemma~\ref{lem:comm:treeformula}. The only difference is that the intersection number
 \[
\int_{\overline{\mathcal M}_{0,1+\ell}} \psi_0^{k_0}\psi_1^{k_1}\cdots\psi_\ell^{k_\ell} 
 \text{ 
is replaced by
  }
\int_{\mathcal B(\ell)} \psi_0^{k_0}\psi_1^{k_1}\cdots\psi_\ell^{k_\ell} 
 \] 
which is known to be given by the above formula for $C(v)$, see \cite[Prop.~6.1.5]{DSV-ncHyperCom}.
\end{proof}

We now recall a result of the first two authors and Vallette allowing to express the differential order condition via the Givental action.

\begin{lemma}[Theorem 7.1.2 in {\cite{DSV-ncHyperCom}}]
For any map $f\colon \ncExtHycomm\to \mathcal O$ that sends all generators except for $\mu^0_{{2}}$ to zero, the infinitesimal Givental action of $\mathrm{d}+rz^{k-1}$ annihilates $f$ if and only if $r\in\mathcal O({1})$ is of differential order at most $k$ with respect to $f(\mu^0_{{2}})\in\mathcal O({2})$.
\end{lemma}

Once again, this result means that for our purposes it is advantageous to set $w=z^{k-1}$ in the formula
 \[
\mathrm{d} \exp(r(w)) =  \exp(r(w)) (\mathrm{d} + w\Delta).
 \]
 describing the differential in the homotopy quotient, since $\mathrm{d} + z^{k-1}\Delta$ is exactly the element whose Givental action encodes compactly the differential order condition. Motivated by this, let us consider only the Givental symmetries $r(z)$ for which 
  \[
r(z)\in \bigoplus_{t \geqslant 1} \mathcal O({1})_{2t} z^{(k-1)t}\subset \mathfrak{g}_{\mathcal O}.
 \]

\begin{lemma} \label{lem:NC-InducedMapOfQuotient}
The map $\exp(r(z)).f\colon \ncExtHycomm \to \mathcal O$ passes through the quotient map $\pi\colon \ncExtHycomm\to \bncHycomm{k}$, that is, there exists a map $g\colon \bncHycomm{k} \to \mathcal O$ such that $\exp(r(z)).f = g\pi$. 
\end{lemma}

\begin{proof} 
It is sufficient to show that the nontrivial decorated plane rooted trees with $n$ labeled leaves and the total degree $2p$ exist only for $n=2+(k-1)p$. Since the correspondence between the arities of vertices and degrees of psi classes in this case is exactly the same as in the commutative case, the computation in the proof of Lemma~\ref{lem:DM-InducedMapOfQuotient} works here as well without any changes. 
\end{proof}

Let us now consider the map $f\colon \ncExtHycomm \to (\bncBV{k}\vee T(r_1,r_2,\dots), \mathrm{d})$ that sends $\mu^0_{{2}}$ to $\mu$ and all other generators to $0$. By Lemma~\ref{lem:NC-InducedMapOfQuotient} there exist a map $g\colon b\ncHycomm\to (\bncBV{k}\vee T(r_1,r_2,\dots), \mathrm{d})$ such that $g \pi = \exp(r(z)).f$. 

\begin{theorem} \label{thm:NC-quasiiso-map}
The map $g\colon \ncHycomm_k\to (\bncBV{k}\vee T(r_1,r_2,\dots), \mathrm{d})$ is a quasi-isomorphism. 
\end{theorem}

As in the case of twisted associative algebras, the proof of this theorem will be gradually constructed in the next sections.

\subsubsection{Quadratic-linear presentation}

Consider the non-symmetric operad $\bncBV{+\infty}$ generated by an element $\Delta$ of arity $0$ and homological degree $1$ and an element $\mu$ of arity $2$ and of homological degree $0$ subject only to the relations 
 \[
\Delta^2=0,\quad \mu\circ_1 \mu = \mu\circ_2 \mu.
 \]
(As above, we write $+\infty$ to hint that this operad is obtained as the limit of $\bncBV{k}$ as $k\to+\infty$, and avoid any possible confusion with the notation $\BV_\infty$ for a cofibrant replacement of $\BV$.) We shall now construct a different presentation of this operad which will be useful to study the case of $\Delta$ of finite differential order. Let us denote by $\beta_{{n}}$ the $n$-th B\"orjeson brace of $\Delta$ with respect to $m$. A gauge symmetry argument \cite{MR3510210,MR3385702} shows that these elements satisfy the relations
 \begin{equation}
 	\label{nc:Stasheffrels}
\sum_{p+q=n+1}\beta_{{p}}\circ \beta_{{q}}=0
 \end{equation}
where we write $\nu \circ \nu'$ for the pre-Lie composition in any operad (the sum of all
possible insertions of $\nu'$ into $\nu$).
The definition of the B\"orjeson braces indicates that the quadratic-linear relations 
\begin{equation} \label{nc:QLrels}
 \beta_{n+1} = 
 	\begin{cases}
 	 	\beta_n\circ_1 \mu-\mu\circ_2 \beta_n, \\
 	 	\beta_n\circ_i \mu, \text{ for any $1<i<n$},\\
 	 	\beta_n\circ_n \mu-\mu\circ_1 \beta_n. 
 	  			\end{cases}
	\end{equation}
are satisfied for all $n\geqslant 2$. 

\begin{proposition}\label{prop:NCBV-infty-Koszul}
The non-symmetric operad with the generators $\mu$ and $\beta_{{n}}$ for $n\geqslant 1$, subject to associativity
of $\mu$ and the relations~\eqref{nc:Stasheffrels} 
and~\eqref{nc:QLrels} 
is isomorphic to $\bncBV{+\infty}$. Moreover, this 
presentation of $\bncBV{+\infty}$ is inhomogeneous Koszul. 
\end{proposition}

\begin{proof}
Let us denote by $\mathcal P$ the operad presented by the given relations. We know that these relations hold in $\bncBV{+\infty}$, so there is a surjection of non-symmetric operads $\mathcal P\twoheadrightarrow\bncBV{+\infty}$. We may order monomials in the corresponding free non-symmetric operad in such a way that the leading monomials of the given relations of $\mathcal P$ are
\[ \mu\circ_1 \mu, \quad
\beta_{{1}}\circ_1\beta_{{n}},
	\quad
\beta_{{n}}\circ_i \mu,
\]
the latter one is present for all $1\leqslant i
\leqslant n$. The normal monomials with respect to these leading terms can be described as associative products (using the iterations of the product $m$) of elements of the L-species $\mathcal F\circ\mathcal D$, where $\mathcal F$ is the underlying species of the free non-symmetric operad generated by $\beta_{{n}}$ for $n>1$, and $\mathcal D$ is $\kk[\Delta]/\Delta^2$ (supported at one-element sets, where we identify $\Delta$
with $\beta_1$). 
Let us exhibit a one-to-one correspondence between such monomials and elements forming a basis of $\bncBV{+\infty}$. For that, we simply replace each factor operation $\beta_p$ by $\Delta\circ_1 \mu^{[p-1]}$; the fact that the elements thus obtained form a basis of $\bncBV{+\infty}$ is clear. 

In general, the set of normal monomials with respect to the given relations form a spanning set of the quotient operad, so the cardinality of that set in each arity is an upper bound on the dimension of the quotient in that arity. In our case, we have shown that this upper bound is equal to the lower bound given by the dimension of the respective component of $\bncBV{+\infty}$, so the operad $\mathcal P$ is isomorphic to $\bncBV{+\infty}$, and the given relations form a Gr\"obner basis. Finally, the latter means that our presentation is inhomogeneous Koszul.
\end{proof}

If we wish to incorporate the finite differential order condition, it is very easy to do using the presentation we obtained:
 \[
\bncBV{k}\cong \bncBV{+\infty}/(\beta_{{k+1}}). 
 \]
This leads to a quadratic-linear presentation of $\bncBV{k}$
with an associative binary operation $\mu$, and operations
$\beta_1,\ldots,\beta_k$ satisfying the relations obtained
by setting $\beta_j=0$ for $j\geqslant k+1$ in~\eqref{nc:Stasheffrels}
and~\eqref{nc:QLrels}:
\begin{align}
\sum_{p+q=n+1}
	\beta_{{p}}\circ \beta_{{q}}&=0, 
		\qquad \text{for $1\leqslant n\leqslant 2k-1$,}
		\label{ncbrace1}	\\
\beta_{n+1} &= 
 	\begin{cases}
 	 	\beta_n\circ_1 \mu-\mu\circ_2 \beta_n, \\
 	 	\beta_n\circ_i \mu, \text{ for any $1<i<n$},\\
 	 	\beta_n\circ_n \mu-\mu\circ_1 \beta_n. 
 	  			\end{cases}	\label{ncbrace2}
\end{align}
for $n\leqslant k-1$, and the corresponding \emph{quadratic}
relations for $n=k$ (where $\beta_{k+1}=0$).
As in the case of twisted associative algebras, some of the terms that 
we used as leading terms in Proposition~\ref{prop:NCBV-infty-Koszul}, 
for example, $\beta_{1}\circ_1\beta_{{n}}$ in the second group of 
relations with $k+1\leqslant n\leqslant 2k-1$, are no longer present, 
and so we cannot use the proof of that proposition to conclude that our 
presentation of $\bncBV{k}$ is inhomogeneous Koszul. Nevertheless, that 
statement turns out to be true. 

\begin{proposition}\label{prop:NCBV-k-Koszul}
The above quadratic-linear presentation of the non-symmetric operad $\bncBV{k}$ is inhomogeneous Koszul for all $k$. 
\end{proposition}

\begin{proof}
The proof is quite similar to that of 
Proposition~\ref{prop:DMBV-k-Koszul}, yet we give some details. 
Let us focus on $k\geqslant 2$: in the case $k=1$, the presentation is 
actually homogeneous, and we understood it well in 
Section~\ref{sec:NCorderone}. Our argument will proceed as
follows. We shall first deal separately with the relations
of~\eqref{ncbrace1}
and then show that the relations describing the interaction of the suboperad generated by $\beta_{{n}}$, $1\leqslant n\leqslant k$, with the suboperad generated by $m$ (isomorphic to the commutative operad) form a quadratic Gr\"obner basis, defining a filtered distributive law between these operads~\cite{MR3302959}. 

To deal with \eqref{ncbrace1}, it is easier to consider the suspension of the Koszul dual operad; it is generated by elements $\chi_{{i}}$ for $i\in \underline{k}$, of arity $i$ and homological degree $2i-4$, subject to the relations
 \[
\chi_{p}\circ_i \chi_q= \chi_{p'}\circ_{i'} \chi_{q'} 
	\text{ for $p+q=p'+q'$,
		 $i\in\underline{p}$ and $i'\in\underline{p'}$}.\]
Let us consider the path degree-lexicographic ordering of monomials in the corresponding free non-symmetric operad. 
Then the elements that are \emph{not} the leading terms of quadratic relations are  
\[ \chi_{{1}}\circ_1 \chi_{j} \text{ for $j\in\underline{k}$ and } 
\chi_{j}\circ_j \chi_{{k-1}} \text{ for $j
 \in \underline{k-1}$}.
 \]
Moreover, one can easily see that the Koszul dual operad has a PBW basis of the form 
 \[
\chi_{I_1+\{\star\}}\circ_\star\cdots\circ_\star \chi_{I_p+\{\star\}}
 \]
where $I_1+\cdots+I_{p-1}+(I_p+\{\star\})$ is a decomposition of $\underline{n}+\{\star\}$ as an ordered sum of (possibly empty) intervals for which a non-empty interval cannot be followed by an empty one and only the first non-empty interval can be of length strictly less than $k-1$. This implies that the operad presented via relations listed in \eqref{ncbrace1} also has a PBW basis, and a quadratic Gr\"obner basis of relations.

To proceed, one uses word operads \cite{MR4114993}, and considers the monoid of ``quantum monomials'' $\mathsf{QM}=\langle x,y,q\mid xq=qx,yq=qy,yx=xyq\rangle$ and the map from the free non-symmetric operad generated by the elements $\mu$ and $\beta_p$ to the word operad associated to $\mathsf{QM}$ sending $\mu$ to $(x,x)$ and $\beta_p$ to $(y,y,\ldots,y)$; the corresponding ordering moves the operation $\mu$ towards the root in the normal forms. Since we know that the suboperads generated by $\mu$ and by $\{\beta_p\}_{p\ge 1}$ separately have Gr\"obner bases, it remains to check that all ``mixed'' S-polynomials can be reduced to zero. There are three groups of ``mixed'' S-polynomials to consider: those of \eqref{ncbrace1} and \eqref{ncbrace2} with the associativity relations for $\mu$, those of \eqref{ncbrace1} and \eqref{ncbrace2} between themselves, and those of \eqref{ncbrace1} and \eqref{ncbrace2} with the relations for the operations $\beta_p$. In each of these cases, the result is obtained by a somewhat tedious but direct calculation; for the latter group of S-polynomials which is perhaps the hardest to deal with directly, it is beneficial to note that the relations between the operations $\beta_p$ in the free operad generated by these operations can be viewed as commutators in the convolution Lie algebra between the linear dual of the associative operad and the latter free operad, and the requisite property of S-polynomials ultimately boils down to the property of the commutator of derivations being a derivation. Thus, the quadratic-linear presentation of $\bncBV{k}$ forms a quadratic Gr\"obner basis, which in turn implies that this presentation is inhomogeneous Koszul.  
\end{proof}

\subsubsection{The Koszul dual dg non-symmetric operad and its homology}
We can now describe the (suspension of the) Koszul dual operad 
$(\bncBV{k}^{!},\partial)$.  Its underlying non-symmetric operad is  generated by operations
$\chi_i$ for 
$i\in \underline{k}$, of arity $i$ and of homological degree $2i-4$, 
and an associative product $\mu$ of arity $2$ and of homological degree 
$1$, subject to the relations
\begin{align}\label{eq:ns-rel-s-koszul}
\chi_{p}\circ_i \chi_q &= \chi_{p'}\circ_{i'} \chi_{q'} 
	\text{ for $p+q=p'+q'$,
		 $i\in\underline{p}$ and $i'\in\underline{p'}$}.
				\notag\\
\chi_s\circ_1 \mu &= \mu \circ_2 \chi_s \\
\chi_s\circ_s \mu &= \mu \circ_1 \chi_s
\end{align}
where we require that
$s,p,q,p',q'\leqslant k$.
 The results of the previous section imply that these relations 
give a distributive law between the operad $\mathcal A_k$ generated 
by all the elements $\chi_i$ for $i\in \underline{k}$, 
and the shifted associative operad generated by the element $\mu$. The 
presence of quadratic-linear relations in $\bncBV{k}$ means that 
the Koszul dual operad has a nonzero differential $\partial$ 
which is a derivation given on generators by 
$\partial(\chi_1) = 0$, $\partial(\mu)  = 0$
and by 
\[
\partial(\chi_{n+1})  =  
	\sum_{i=1}^n \chi_n\circ_i \mu
	 	\quad \text{for $1\leqslant n\leqslant k-1$}.
\]
For the rest of the proof, it will be important to note that we have a chain of inclusions of dg operads
 \[
\mathcal A_1\subset\cdots\subset \mathcal A_k\subset\cdots 
 \]
we denote by $\mathcal A_{+\infty}$ the union of all of 
them, and by $\bncBV{+\infty}^!$ the union of all the dg 
operads $(\bncBV{k}^{!},\partial)$. That latter union, viewed as a chain complex, is homotopy equivalent to the (right) Koszul complex of the shifted associative operad 
(since the operad $\mathcal A_{+\infty}$ is the Koszul dual of the dg operad of shifted A-infinity algebras); as such, it is acyclic.
Thus, the long exact 
sequence of the homology implies that the homology of the 
complex $(\bncBV{k}^{!},\partial)$ is isomorphic to the 
subspace of $\bncBV{k}^{!}$ obtained as boundaries of 
elements outside of that subcomplex.

As in the case of szmmetric operads, let us
denote by $K_{n+1}^{2t}$ an element of arity 
$n+1$ and degree $2t$ that is equal to the composition 
$\chi_{p_1+1}\circ_1\cdots\circ_1
\chi_{p_{n-t}+1}$ for any choice of partition 
$n=p_1+ \cdots + p_{n-t}$ with $t\leqslant n$. 
Exactly as in the case of twisted associative algebras, we find that 
the since we are interested in the quotient of the cobar construction 
by $s^{-1}(\chi_1^r)^\vee$ for $r>0$, and one can choose the 
contracting homotopy for which vanishes on elements that do not 
contain $\chi_1$, %\delta is h_1 
the higher structures of an operad up to homotopy 
vanish, and the homology has just an operad structure. As an operad, 
the homology is immediately seen to be generated by the elements 
 \[
\gamma_n :=\partial(K_{\underline{n}}^{2t}), \quad n=q(k-1)+2,\quad t=q(k-2)+2,\quad q\geqslant 0.
 \] 

It remains to show that the elements $\gamma_{{n}}$ satisfy precisely the relations described in Proposition \ref{prop:bNCHycomm-Koszul}. The fact that these relations are satisfied is obtained by direct inspection, so there is a surjective map from $\bncHycomm{k}^!$ onto the homology. To conclude, we shall argue as follows. From the proof of Corollary \ref{cor:dim-b-ncgrav}, we know the normal forms in $\bncHycomm{k}^!(n)$, and so it enough to show that we can find linearly independent elements spanning the homology which are in one-to-one correspondence with those normal forms. For each normal form 
 \[
\gamma_p(u_1,u_2,\ldots,u_{p-1},\mathrm{id}_{\{j\}}),  
 \]
let us consider the corresponding element 
 \[
K_{I}^{2t}\circ(u_1',\ldots, u_{p-1}',\mathrm{id}_{j}),  
 \]
where $u_i'$ is obtained from $u_i$ by replacing $\gamma_2$ with $\mu$. (Note that $t$ is computed uniquely from $n$.) By a direct inspection, the boundaries of such elements are linearly independent in $\bncBV{+\infty}^!$, and therefore the surjection from $\bncHycomm{k}^!$ onto the homology must be an isomorphism. As an immediate by-product, we obtain one of the central results of this section.

\begin{theorem}
The homotopy quotient of $\bncBV{k}$ by $\Delta$ is represented by $\bncHycomm{k}$. 
\end{theorem}

\subsubsection{Proof of Theorem \ref{thm:NC-quasiiso-map} } The argument is essentially the same as in the proof of Theorem \ref{thm:LM-quasiiso-map}. The dg operad $(\bncBV{k}\vee T(r_1,r_2,\dots), \mathrm{d})$ computes the homotopy quotient by $\Delta$, thus its homology is isomorphic to $\bncHycomm{k}$. In addition, the following statements are true. 
\begin{enumerate}
\item The map $g$ is a morphism of dg operads.
\item The images of the generators of $\bncHycomm{k}$ do not vanish in the homology of $\mathrm{d}$: as the image of $\mathrm{d}$ is contained in the ideal generated by~$\Delta$.
\item In the homology of $\mathrm{d}$, the subspace of elements of arity $2+(k-1)t$ and homological degree $2t$ is one-dimensional for each $t\ge0$.
\end{enumerate}
This implies that the map $g$ is surjective on the homology of $\mathrm{d}$, hence it is a quasi-isomorphism.\qed

%\input{realversion.tex}
%\bibliographystyle{alpha} 
%\bibliography{homoquo.bib} 

\printbibliography
\end{document}